\documentclass{amsart}

\usepackage{calc}
\usepackage{tikz}
\usepackage{pgfplots}
\usetikzlibrary{shapes}
\usepackage{amsmath,amssymb}
 \usepackage[foot]{amsaddr}
  \newcommand\sn[1]{{\color{black} #1}}
  \newcommand\ma[1]{{\color{black} #1}}

\let\oldtocsection=\tocsection

\let\oldtocsubsection=\tocsubsection

\let\oldtocsubsubsection=\tocsubsubsection

\renewcommand{\tocsection}[2]{\vspace{0.5em}\hspace{0em}\oldtocsection{#1}{#2}}
\renewcommand{\tocsubsection}[2]{\vspace{0.5em}\hspace{1em}\oldtocsubsection{#1}{#2}}
\renewcommand{\tocsubsubsection}[2]{\vspace{0.5em}\hspace{2em}\oldtocsubsubsection{#1}{#2}}
\usepackage{graphicx,cancel,xcolor,hyperref,comment,graphicx,geometry}

\let\originallesssim\lesssim
\DeclareRobustCommand{\lesssim}{%
  \mathrel{\mathpalette\lowersim\originallesssim}%
}
\usepackage{tikz}
\usepackage{graphicx,url,etoolbox}
\usepackage{lipsum}
\usepackage{dsfont}
\usepackage{subfigure}
\usetikzlibrary{hobby,patterns}

\usepackage{fancyhdr}
\usepackage{lastpage}

\newtheorem{theoreme}{Theorem}[section]

\newtheorem{lemma}[theoreme]{Lemma}
\newtheorem{rem}[theoreme]{Remark}

\newtheorem{definition}[theoreme]{Definition}
\theoremstyle{definition}

\numberwithin{equation}{section}

 \renewenvironment{proof}{{\bfseries \noindent Proof.}}{\demo}
\newcommand\xqed[1]{%
  \leavevmode\unskip\penalty9999 \hbox{}\nobreak\hfill
  \quad\hbox{#1}}
\newcommand\demo{\xqed{$\square$}}

\hypersetup{bookmarks, colorlinks, urlcolor=blue, citecolor=red, linkcolor=blue, hyperfigures, pagebackref,
    pdfcreator=LaTeX, breaklinks=true, pdfpagelayout=SinglePage, bookmarksopen=true,bookmarksopenlevel=2}

\usepackage{comment}
\def\u2{\u^2}
\def\u3{\u^3}
\def\u4{\u^4}
\def\u5{\u^5}
\def\y1{\y^1}
\def\y2{\y^2}
\def\y3{\y^3}
\def\y4{\y^4}
\def\y5{\y^5}

\def\HH{\mathcal H}

\def\AA{\mathcal A}

\def\la {{\lambda}}

\newcommand {\nc}   {\newcommand}
\nc {\be}   {\begin{equation}} \nc {\ee}   {\end{equation}} \nc
{\beq}  {\begin{eqnarray}} \nc {\eeq}  {\end{eqnarray}} \nc {\beqs}
{\begin{eqnarray*}} \nc {\eeqs} {\end{eqnarray*}}
\def\edc{\end{document}}

\providecommand{\abs}[1]{\lvert#1\rvert}

\usepackage{tikz}
\usetikzlibrary{decorations.pathmorphing,patterns,scopes,intersections,calc}
\usepackage{caption}

\begin{document}
\title[\fontsize{7}{9}\selectfont  ]{Stability Results for Novel Serially-connected Magnetizable Piezoelectric and Elastic Smart-System Designs}
\author{Mohammad Akil$^{1}$, Serge Nicaise$^{1}$, Ahmet \"Ozkan \"Ozer$^{2}$  and Virginie R\'egnier$^{1}$  \vspace{0.5cm}\\
$^1$Univ. Polytechnique  Hauts-de-France, INSA Hauts-de-France,  CERAMATHS-Laboratoire de Mat\'eriaux C\'eramiques et de Math\'ematiques, F-59313 Valenciennes, France.\\ \vspace{0.5cm}
$^2$ Department of Mathematics, Western Kentucky University, Bowling Green, KY 42101, USA.
Email: mohammad.akil@uphf.fr, serge.nicaise@uphf.fr, ozkan.ozer@wku.edu, virginie.regnier@uphf.fr}

\setcounter{equation}{0}
\begin{abstract}
In this paper, the stability of longitudinal vibrations for transmission problems  of two smart-system designs are studied: (i) a serially-connected Elastic-Piezoelectric-Elastic  design with a local damping acting only on the piezoelectric layer and (ii) a serially-connected Piezoelectric-Elastic design with a local damping acting on the elastic part only. Unlike the existing literature,  piezoelectric layers are considered magnetizable, and therefore, a fully-dynamic PDE model, retaining interactions  of electromagnetic fields (due to Maxwell's equations) with the mechanical vibrations, is considered.  The design (i) is shown to have exponentially stable solutions. However, the nature of the stability of solutions of the design (ii), whether it is polynomial or exponential, is dependent  entirely upon the arithmetic nature of a quotient involving all  physical parameters. Furthermore, a polynomial decay rate is provided in terms of a measure of irrationality of the quotient. Note that this type of result is totally new (see Theorem \ref{Pol-PE} and Condition $\rm{\mathbf{(H_{Pol})}}$). The main tool used throughout the paper is the multipliers technique which requires an adaptive selection of  cut-off functions together with a particular attention to the sharpness of the estimates  to optimize the results.
\\[0.1in]
\end{abstract}

\maketitle
\pagenumbering{roman}
\maketitle
\pagenumbering{arabic}
\setcounter{page}{1}
{\color{black}
\section{Introduction}

Piezoelectric materials are multi-functional smart materials (most notably Lead Zirconate Titanate) used to develop electric displacement that is directly proportional
to an applied mechanical stress \cite{Smith}.    They can  be used as actuators/sensors, and also be integrated to a mother host structure \cite{Shi}.
Due to their small size and high power density, they have become more and more promising in industrial applications such as implantable biomedical devices and sensors \cite{Dag1,Dag2}, wearable human-machine interface sensors \cite{Dong},  and nano-positioners and micro-sensors  due to the excellent advantages of the fast response time, large mechanical force, and extremely fine resolution \cite{Ru}.

In deriving a mathematical model for the equations of motion on a piezoelectric beam, actuated by a voltage source, three major effects and their interrelations need to be considered: mechanical, electrical, and magnetic.  Mechanical effects are mostly modeled through Kirchhoff, Euler-Bernoulli, or Mindlin-Timoshenko   small (linear) \cite{Lag} or large (nonlinear) \cite{Ala} displacement  assumptions, where the constitutive relations between the nonzero stress and strain tensors are used to model longitudinal displacements of the centerline (stretching/compression), transverse    displacements (bending), and rotations of the beam. It is also reported that the small displacement assumptions lead to the bending and rotational motions completely  immune the applied voltage \cite{Morris-Ozer2013}. These tensors are coupled to the electric/magnetic displacements and electric/magnetic field tensors. There are mainly three approaches to include electromagnetic effects  due the Maxwell's equations: electrostatic, quasi-static, and fully-dynamic \cite[p. 336]{Lions}.
Electrostatic and quasi-static approaches are widely used in voltage-controlled piezoelectric beam models - see e.g. \cite{Smith} and the references therein. These models completely exclude magnetic effects and their coupling  with electrical and mechanical effects.  Even though the electro-static and quasi-static approaches are sufficient for defining piezoelectricity, electromagnetic waves generated by mechanical fields need to be accounted for in the calculation of radiated electromagnetic power from a vibrating piezoelectric acoustic device, e.g. see \cite{Yang} and the references therein.  For this reason, the fully dynamic models of piezoelectric beams are needed to be understood well.
In fact,  the dynamic effects  for (acoustic) magnetizable piezoelectric beams  are pronounced and must be taken into account in the modeling \cite{Morris-Ozer2013,Voss}.

Denote by $v(x,t)$ and  $p(x,t)$ the longitudinal vibrations of the center line of the beam and  the total charge accumulated at the electrodes of a single piezoelectric beam. Assuming that the beam is fixed at the left end $x=0$ and free at the right end $x=L$,  the equations of motion  is a system of partial differential equations \cite{Morris-Ozer2013} as the following
\begin{equation}\label{piezo}\left\{
\begin{array}{ll}
\rho v_{tt}-\alpha v_{xx}+\gamma \beta p_{xx}=0,& (x,t) \in(0,L)\times (0,\infty)\\
\mu p_{tt}-\beta p_{xx}+\gamma \beta v_{xx}=0,&\\
v(0,t)=p(0,t)=0,&\\
 \alpha v_x(L,t)-\gamma \beta p_x(L,t)=g(t),&\\
 \beta p_x(L,t)-\gamma \beta v_x(L,t)=-V(t),& t\in (0,\infty)\\
 (v,p,v_t,p_t)(x,0)=(v_0,p_0,v_1,p_1),& x\in [0,L]
\end{array}	
\right.
\end{equation}
where  $\rho$, $\alpha$, $\beta$, $\gamma$, and $\mu$ are mass density per unit volume, elastic stiffness,  impermeability, piezoelectric constant, and  magnetic permeability of the beam, respectively, and $g(t)$ and $V(t)$ are strain and voltage actuators, and
\begin{equation}\label{alpha1}
\alpha_1:=\alpha-\gamma^2\beta>0.	
\end{equation}

By the electrostatic approach, the model above is simplified to a single wave equation model by  taking $\mu \equiv 0$ and $p_t=0,$ considering $g(t)\equiv 0,$ and \eqref{alpha1}, e.g. see \cite{Morris-Ozer2014},
\begin{equation}\label{electro}\left\{
\begin{array}{ll}
\rho v_{tt}-\alpha_1 v_{xx}=0,& (x,t) \in(0,L)\times (0,\infty)\\
v(0,t)=0, \quad  \alpha_1 v_x(L,t)=\gamma V(t)& t\in (0,\infty)\\
 (v, v_t)(x,0)=(v_0,v_1),& x\in [0,L].
\end{array}	
\right.
\end{equation}
The model by the quasi-static approach is the same as \eqref{electro} yet $p_t \ne 0.$}

The exact observability/stabilizability and the type of stability of the solutions (\ref{piezo}) of the  PDE model by each approach differs substantially. For example, the PDE model obtained by electrostatic/qua\-si-static approach is the boundary damped wave equation in \eqref{electro}, and it is known to be exactly observable/exponentially stabilizable with one state measurement $v_t(x,t)$ on the boundary $x=L$, e.g. see \cite{Chen,Lag1}. In deep contrast to this result,  the PDE model obtained by the the fully-dynamic approach in \eqref{piezo} can not be exactly observable/exponentially stabilizable for almost all choices of material parameters with only one state measurement, $v_t(x,t)$ or $p_t(x,t),$ on the boundary $x=L$, e.g. see \cite{Morris-Ozer2013,Morris-Ozer2014}. Explicit polynomial decay estimates are obtained for more regular initial data and for a small class of materials satisfying certain number-theoretical conditions \cite{Morris-Ozer2014,OzerMCSS}.  The same model (\ref{piezo}) is considered in  \cite{Ramos2018} for the open-loop sensor configuration (i.e. $g(t),V(t)\equiv 0$) with a dissipative damping term $\delta v_t$ with $\delta>0,$ acting only in the first equation of \eqref{piezo}.  It is also reported that two nonzero state feedback measurements $v_t(L,t)$ (tip velocity) and $p_t(L,t)$ (total current on the electrodes)  are necessary to achieve exact observability/exponential stabilizability \cite{Ramos,Wilson}.
This underlines the fact  that the two boundary damping terms or one viscous damping term are both able to  exponentially dissipate non-stabilizing (high-frequency) magnetic effects.
There is also a large literature considering the model \eqref{piezo} under thermal  effects, fractional-type damping, and distributed or boundary-type memory and delay terms, see \cite{Abdelaziz1}-\cite{AnLiuKong},\cite{DosSan}-\cite{Feng2}, \cite{Abdelaziz2,MR4450079} and the references therein.

    \begin{figure}[htb!]
 \vspace{-0.1in}
    \centering
        {{\includegraphics[width=4.5in]{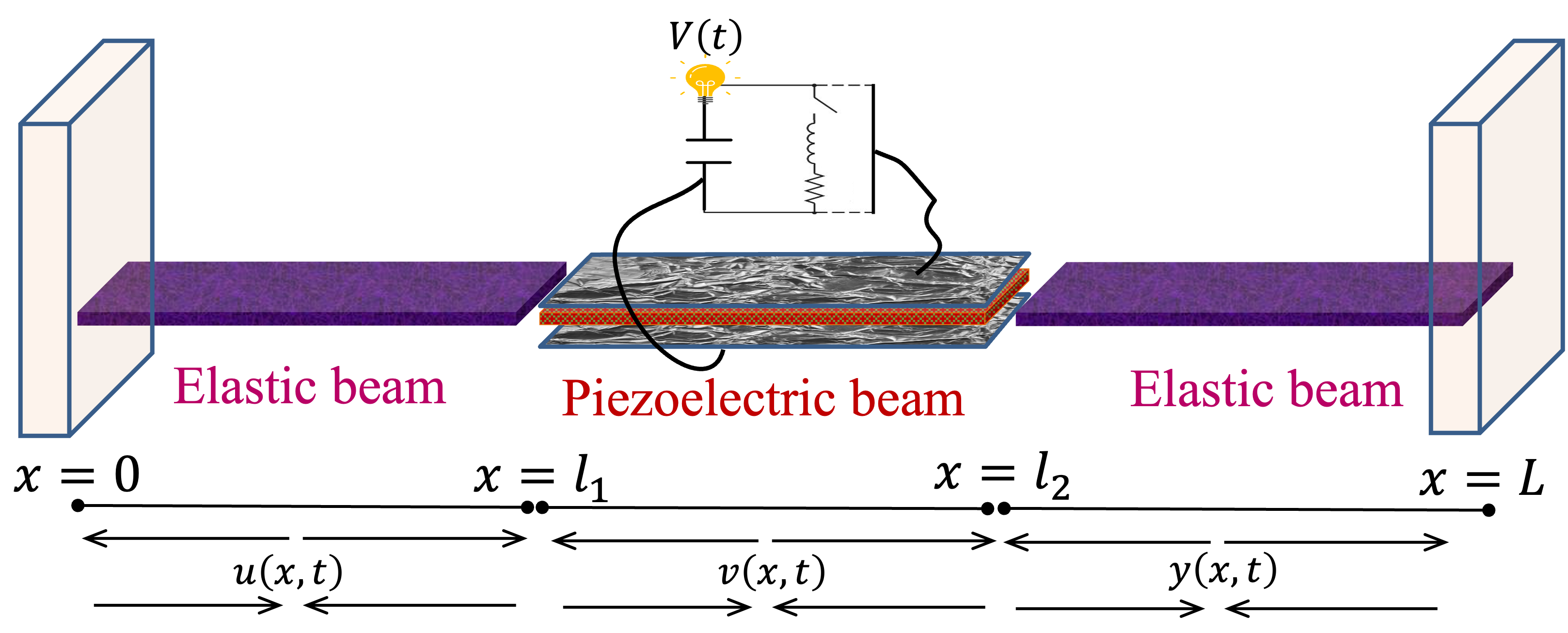} }}
    \caption{\footnotesize  Serially-connected  Elastic-Piezoelectric-Elastic transmission system clamped at both ends. The piezoelectric material  itself is an elastic material covered by electrodes at their top and bottom surfaces,  and connected to an external electric circuit. As  the elastic layers stretches or shrinks, the piezoelectric beam stretches or shrinks as well, and therefore, charges separate and line up in the vertical direction, and electric field (voltage) is induced in the electrodes. The overall motions on the system are considered to be only longitudinal.}%
    \label{EPE1}%
\end{figure}

A serially-connected smart system  is an elastic structure consisting of  longitudinally attached  fully-elastic and piezoelectric layers, see Figs. \ref{EPE1} and \ref{EP}.
Use of piezoelectric materials for a serially-connected design in various transmission mechanisms of aerospace, aviation, automobile,
ships, and robots has boosted substantially in the last decade, see \cite{Chen2,Ling} and the references therein. A rigorous mathematical treatment for a transmission problem of a three serially-connected purely-elastic waves/strings/beams is provided  in \cite{Fatori2}. Indeed, if the outer wave equations have both viscous damping terms,  an exponential stability result is shown to be immediate.
 Several authors have also studied transmission problems of serially-connected strings/beams with e.g. a thermoelastic material \cite{Naso1} or a viscoelastic material \cite{Rivera}.

To the best of our best knowledge, serially-connected transmission systems involving elastic and magnetizable piezoelectric systems are not treated mathematically in the literature, especially with Condition $\rm{\mathbf{(H_{Pol})}}$, which appears in Section $\ref{SecPol}$. The goal of this paper is to fix this gap by considering two particular designs, for which we obtain novel decay rates of the energy. 

The first design, whose PDE model is described below in \eqref{EPE}, is the transmission problem of an Elastic-Piezoelectric-Elastic system, as in Fig. \ref{EPE1}, with only one local damping acting on the longitudinal displacement of the center line of the piezoelectric material:

\begin{equation}\label{EPE}\tag{${\rm E/P/E}$}
\left\{\begin{array}{ll}
u_{tt}-c_1u_{xx}=0,& (x,t) \in(0,l_1)\times (0,\infty),\\
\rho v_{tt}-\alpha v_{xx}+\gamma\beta p_{xx}+d_2(x) v_t=0,& (x,t) \in (l_1,l_2)\times (0,\infty),\\
\mu p_{tt}-\beta p_{xx}+\gamma \beta v_{xx}=0,&(x,t) \in(l_1,l_2)\times (0,\infty),\\
y_{tt}-c_2y_{xx}=0,&(x,t) \in(l_2,L)\times (0,\infty),\\
u(0,t)=y(L,t)=0,&\\
v(l_1,t)=u(l_1,t),&\\
v(l_2,t)=y(l_2,t),&\\
\alpha v_x(l_1,t)-\gamma\beta p_x(l_1,t)=c_1u_x(l_1,t),&\\
\alpha v_x(l_2,t)-\gamma\beta p_x(l_2,t)=c_2y_x(l_2,t),&\\
\beta p_x(l_1,t)=\gamma\beta v_x(l_1,t),&\\
\beta p_x(l_2,t)=\gamma\beta v_x(l_2,t),&  t\in (0,\infty),\\
(u, v , p, y, u_t, v_t, p_t, y_t) (\cdot,0)=(u_0 ,v_0, p_0, y_0,u_1, v_1, p_1, y_1) (\cdot)
\end{array}
\right.	
\end{equation}
where $0<l_1<l_2<L$, $\ma{c_1,c_2>0}$ and  $d_2\in L^{\infty}(l_1,l_2)$, such that
\begin{equation}\label{LD-P}\tag{$\rm{LD-P}$}
d_2(x)\geq d_{2,0}>0\quad \text{in}\quad (a_2,b_2)\subset (l_1,l_2),\ \text{and}\  d_2(x)\geq 0\quad \text{in}\quad (l_1,l_2)\backslash (a_2,b_2).	
\end{equation}

    \begin{figure}[htb!]
 \vspace{-0.1in}
    \centering
           {{\includegraphics[width=3in]{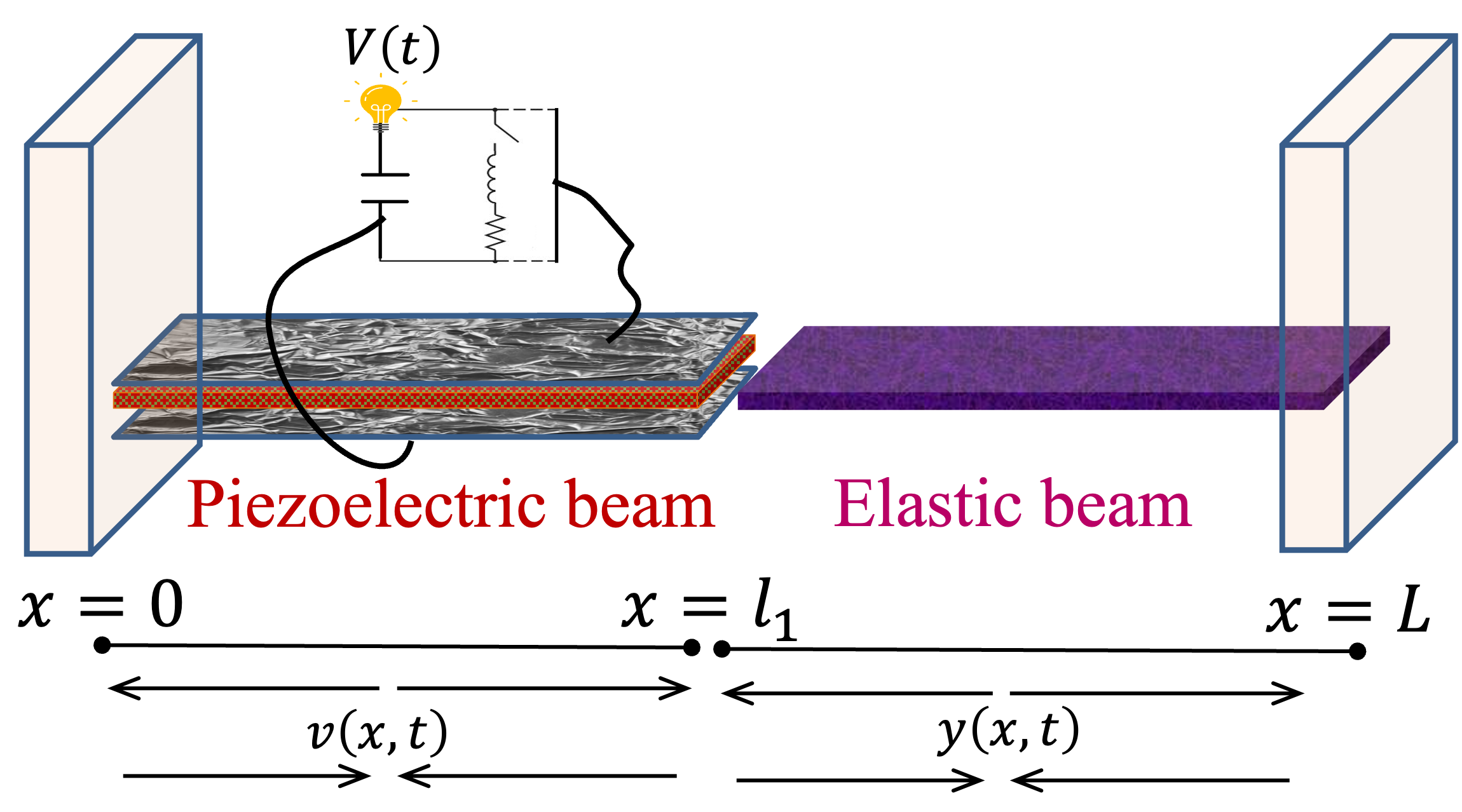} }}%
    \caption{ Serially connected  Elastic-Piezoelectric transmission line clamped at both ends.}%
    \label{EP}%
\end{figure}

The second design, whose PDE model is described below in \eqref{PE},  is for the transmission problem of a Piezoelectric-Elastic system, as in Fig. \ref{EP}, with only one local damping acting on the elastic part:
\begin{equation}\label{PE}\tag{${\rm P/E}$}
\left\{\begin{array}{ll}
\rho v_{tt}-\alpha v_{xx}+\gamma\beta p_{xx}=0,&(x,t) \in  (0,l_1)\times (0,\infty),\\
\mu p_{tt}-\beta p_{xx}+\gamma \beta v_{xx}=0,&(x,t) \in  (0,l_1)\times (0,\infty),\\
y_{tt}-c_2y_{xx}+d_1(x)y_t=0,& (x,t) \in (l_1,L)\times (0,\infty),\\
v(0,t)=p(0,t)=y(L)=0,\\
v(l_1,t)=y(l_1,t),&\\
\alpha v_x(l_1,t)-\gamma\beta p_x(l_1,t)=c_2y_x(l_1,t),&\\
\beta p_x(l_1,t)=\gamma\beta v_x(l_1,t), & t\in (0,\infty),\\
(v, p, y,v_t,p_t,y_t)(\cdot,0)=(v_0, p_0, y_0, v_1, p_1, y_1)(\cdot),
\end{array}
\right.	
\end{equation}
where $0<l_1<L$, $\ma{c_2>0}$ and $d_1\in L^{\infty}(l_1,L)$ such that
\begin{equation}\label{LD-E}\tag{$\rm{LD-E}$}
d_1(x)\geq d_{1,0}>0\quad \text{in}\quad (a_1,b_1)\subset (l_1, L),\ \text{and}\  d_1(x)\geq 0\quad \text{in}\quad (l_1, L)\backslash (a_1,b_1).
\end{equation}

\subsection{Our Contributions and Main Results}
{\color{black}
In this study, we investigated two transmission problems involving alternating magnetizable piezoelectric and elastic beams, addressing distinct scenarios of distributed and boundary damping. The first scenario, an Elastic-Piezoelectric-Elastic design, with a single local damping acting on the longitudinal displacement of the piezoelectric layer's center line, yields an immediate exponential stability result. Conversely, the second scenario, a Piezoelectric-Elastic design, with a lone local damping affecting the elastic part, exhibits stability dependent on the arithmetic nature of a quotient involving system parameters.

Our main results are listed as the following.
\begin{theoreme}\label{EPE-EXP}
If \eqref{LD-P} holds, the $C_0-$semigroup of contractions $(e^{tA_{EPE}})_{t\geq 0}$ is exponentially
stable, i.e. there exists $M\geq 1$ and $\omega >0$   such
that
\begin{equation}\label{EPE-EXP-EQ1}
\|e^{t\mathcal{A}_{EPE}}U_0\|_{\mathcal{H}}\leq Me^{-\omega t}\|U_0\|_{\mathcal{H}},\quad \forall U_0\in \HH,
\end{equation}
where $\HH$ and $\mathcal{A}_{EPE}$ are defined in \eqref{defH} and \eqref{defAEPE}, respectively.
\end{theoreme}

\begin{theoreme}\label{PE-EXP}
Assume \ma{that \eqref{LD-E} and }  $\rm{\mathbf{(H_{Exp})}},$ defined in Section \ref{SecPol}, hold. Then, the $C_0-$semigroup of contractions $(e^{t\mathcal{A}_{PE}})_{t\geq 0}$ is exponentially
stable, i.e. there exist $M\geq 1$ and $\omega >0$ such that
\begin{equation}\label{EPE-EXP-EQ1}
\|e^{t\mathcal{A}_{PE}}U_0\|_{\mathcal{H}_{PE}}\leq Me^{-\omega t}\|U_0\|_{\mathcal{H}_{PE}},\qquad  \forall U_0\in \mathcal{H}_{PE},
\end{equation}
where $ \mathcal{H}_{PE}$ and $\mathcal{A}_{PE}$ are defined in \eqref{defHPE} and \eqref{defAPE}, respectively.
\end{theoreme}

\begin{theoreme}\label{Pol-PE}
Assume \ma{that \eqref{LD-E} and }  $\rm{\mathbf{(H_{Pol})}},$ defined in Section \ref{SecPol}, hold.  Then, there exists a constant $C>0$ such that  the energy of the system \eqref{PE} satisfies the following estimate for all $t>0$
\begin{equation}\label{Pol-PE-Eq1}
\|e^{t\mathcal{A}_{PE}}U_0\|^2\leq \frac{C}{t^{\frac{2}{4\varpi\left(\frac{\sigma_+}{\sigma_-}\right)-4}}}\|U_0\|^2_{D(\mathcal{A}_{PE})},\quad \forall U_0\in D(\AA_{PE}),	
\end{equation}
where $ \mathcal{H}_{PE}$ and $\mathcal{A}_{PE}$ are defined in \eqref{defHPE} and \eqref{DAPE}, respectively.

\end{theoreme}
These results are succinctly summarized in Tables \ref{tablea} and \ref{tableb} for the EPE and PE designs, respectively. Additionally, Table \ref{table 1} provides a comprehensive overview of exponential and polynomial stabilities, contributing theoretical insights with practical applications. This discussion underscores unique facets of the stability landscape crucial for robust system performance.

The paper's structure is organized to provide detailed insights into each transmission problem. In Section \ref{Section-EPE-LD-P}, we establish well-posedness and exponential stability of the Elastic-Piezoelectric-Elastic model \eqref{EPE} under conditions (\ref{LD-P}) on the damping function $d_2$. In Section \ref{PE-LD-E}, we analyze the well-posedness and strong stability of the Piezoelectric-Elastic system \eqref{PE} under conditions (\ref{LD-E}) on the damping function $d_1$. The decay rate of the energy is intricately linked to the arithmetic nature of a quotient involving all physical parameters. Specifically, if the quotient is a rational number different from $\frac{2n_+-1}{2n_--1}$ for all $n_+,n_-\in \mathbb{N}$ (refer to \eqref{SS-SC-PE} in Theorem \ref{SS-PE}), exponential decay is proven. Conversely, if the quotient is an irrational number, the energy is demonstrated to decrease polynomially, provided the irrationality measure of this quotient is finite. The proof methodology relies on the multipliers technique, necessitating a judicious selection of cut-off functions and meticulous attention to estimate sharpness to optimize results.

It is noteworthy that our arguments in Section \ref{Section-EPE-LD-P} lead to a significant result for the electrostatic/quasi-static design, similar to the one in \cite{Fatori2}, with local damping exclusively in the middle layer. This design is shown to be exponentially stable, marking a notable advancement over the results in \cite{Fatori2}, where exponential stability was achieved solely with fully-distributed viscous damping terms for the outer layers, as outlined in Remark \ref{rem-electro}.

\begin{table}[h]
\begin{center}
\begin{tabular}{|p {1.3cm}|c|c|c|c|c|c|}
\hline  & Elastic & \textbf{}& Piezoelectric & \textbf{}& Elastic & Type of Stability\\
\hline & $(0,l_1)$& $x=l_1$& $(l_1,l_2)$& $x=l_2$& $(l_2,L)$& \\
\hline
 Damping & X & X & Partial Viscous Damping & X & X &
Exponential
 \\

\hline
\end{tabular}
\end{center}
\caption{Summary of results for the EPE design. }
\label{tablea}
 \end{table}

 \begin{table}[h]
\begin{center}
\begin{tabular}{|p {1.3cm}|c|c|c|c|c|}
\hline  & Piezo & \textbf{}& Elastic&    Type of Stability\\
\hline & $(0,l_1)$& $x=l_1$& $(l_1,L)$ & \\
\hline
 Damping & X & X & Partial Viscous Damping &
Exponential or Polynomial
 \\
\hline
\end{tabular}
\end{center}
\caption{Summary of results for the EPE design.}
\label{tableb}
 \end{table}
`X' indicates no damping in the corresponding domain.
}

\section{Stability results for the \eqref{EPE} system}\label{Section-EPE-LD-P}
\sn{Note that the assumption  \eqref{LD-P}  applies to all results in this section. For simplicity, the repetition is avoided unless it is necessary to state.}
\subsection{Well-Posedness}\label{WP-EPE}
\noindent This section is devoted to establish the well-posedness of the system \eqref{EPE} by a semigroup approach.
The natural energy of system  \eqref{EPE} is defined by
$$
E(t)=\frac{1}{2}\int_0^{l_1}\left(\abs{u_t}^2+c_1\abs{u_x}^2\right)dx+\frac{1}{2}\int_{l_1}^{l_2}\left(\rho\abs{v_t}^2+\alpha_1\abs{v_x}^2+\mu\abs{p_t}^2+\beta\abs{\gamma v_x-p_x}^2\right)dx+\frac{1}{2}\int_{l_2}^{L}\left(\abs{y_t}^2+c_2\abs{y_x}^2\right)dx.
$$
\begin{lemma}\label{denergy}
The energy $E(t)$ is dissipative along  the regular solutions $\ma{(u,v,p,y)}$ of the system \eqref{EPE}, i.e.
\begin{equation}\label{denergy1}
\frac{d}{dt}E(t)=-\int_{l_1}^{l_2}d_2\abs{v_t}^2dx.
\end{equation}	
\end{lemma}
\begin{proof}
First, multiplying $\eqref{EPE}_1$ by $\overline{u}_t$, integrate by parts over $(0,l_1),$ and take the real part  to get
\begin{equation}\label{denergy1}
\frac{1}{2}\frac{d}{dt}\int_0^{l_1}|u_t|^2dx+\frac{c_1}{2}\frac{d}{dt}\int_0^{l_1}|u_x|^2dx-c_1\Re\left(u_x(l_1,t)\overline{u_t(l_1,t)}\right)=0.	
\end{equation}
Next, multiply $\eqref{EPE}_2$ by $\overline{v_t}$, integrate by parts over $(l_1,l_2),$ and take the real part to get
\begin{equation}\label{denergy2}
\begin{array}{l}
\displaystyle
\frac{\rho}{2}\frac{d}{dt}\int_{l_1}^{l_2}|v_t|^2dx+\frac{\alpha}{2}\frac{d}{dt}\int_{l_1}^{l_2}|v_x|^2dx-\alpha\Re\left(v_x(l_2,t)\overline{v_t(l_2,t)}\right)+\alpha\Re\left(v_x(l_1,t)\overline{v_t(l_1,t)}\right)\\[0.1in]
\displaystyle
-\gamma\beta\Re\left(\int_{l_1}^{l_2}p_x\overline{v_{xt}}dx\right)+\gamma\beta\Re\left(p_x(l_2,t)\overline{v_t(l_2,t)}\right)-\gamma\beta\Re\left(p_x(l_1,t)\overline{v_t(l_1,t)}\right)=-\int_{l_1}^{l_2}d_2|v_t|^2dx.		
\end{array}
\end{equation}
Now, multiply $\eqref{EPE}_3$ by $\overline{p}_t$, integrate by parts over $(l_1,l_2)$, and take the real part to get
\begin{equation*}
\begin{array}{l}
\displaystyle
\frac{\mu}{2}\frac{d}{dt}\int_{l_1}^{l_2}|p_t|^2dx+\frac{\beta}{2}\frac{d}{dt}\int_{l_1}^{l_2}|p_x|^2dx-\beta\Re\left(p_x(l_2,t)\overline{p_t(l_2,t)}\right)+\beta\Re\left(p_x(l_1,t)\overline{p_t(l_1,t)}\right)\\[0.1in]
\displaystyle
-\gamma\beta\Re\left(\int_{l_1}^{l_2}v_x\overline{p_{xt}}dx\right)+\gamma\beta\Re\left(v_x(l_2,t)\overline{p_t(l_2,t)}\right)-\gamma\beta\Re\left(v_x(l_1,t)\overline{p_t(l_1,t)}\right)=0.
\end{array}
\end{equation*}
By implementing $\eqref{EPE}_{10}$ and $\eqref{EPE}_{11},$
\begin{equation}\label{denergy3}
\frac{\mu}{2}\frac{d}{dt}\int_{l_1}^{l_2}|p_t|^2dx+\frac{\beta}{2}\frac{d}{dt}\int_{l_1}^{l_2}|p_x|^2dx-\gamma\beta\Re\left(\int_{l_1}^{l_2}v_x\overline{p_{xt}}dx\right)=0,
\end{equation}
and multiplying $\eqref{EPE}_4$ by $\overline{y}$ and integrating by parts over $(l_2,L)$ lead to
\begin{equation}\label{denergy4}
\frac{1}{2}\frac{d}{dt}\int_{l_2}^{L}|y_t|^2dx+\frac{c_2}{2}\frac{d}{dt}\int_{l_2}^{L}|y_x|^2dx+c_2\Re\left(y_x(l_2,t)\overline{y_t(l_2,t)}\right)=0.	
\end{equation}
Thus, by adding \eqref{denergy2} and \eqref{denergy3} and noting \eqref{alpha1},
\begin{equation}\label{denergy5}
\begin{array}{l}
\displaystyle
\frac{1}{2}\frac{d}{dt}\left(\int_{l_1}^{l_2}\left(\rho|v_t|^2+\alpha_1|v_x|^2+\mu|p_t|^2+\beta |\gamma v_x-p_x|^2\right)dx\right)+\Re\left(\left(\alpha v_x(l_1,t)-\gamma\beta p_x(l_1,t)\right)\overline{v_t(l_1,t)}\right)\\[0.1in]
\displaystyle
- \Re\left(\left(\alpha v_x(l_2,t)-\gamma\beta p_x(l_2,t)\right)\overline{v_t(l_2,t)}\right)=-\int_{l_1}^{l_2}d_2|v_t|^2dx.	
\end{array}	
\end{equation}
In the final step of the proof, add \eqref{denergy1}, \eqref{denergy4} and \eqref{denergy5}, use the continuity conditions $\eqref{EPE}_6$ and $\eqref{EPE}_7$ and the transmission conditions  $\eqref{EPE}_8$ and $\eqref{EPE}_9.$ Hence,  \eqref{denergy} follows.
\end{proof}

In order to have a unique solution to \eqref{EPE}, \ma{the following Hilbert spaces are introduced.} For any real numbers $a,b$ such that $a<b$,
$$
\begin{array}{ll}
\displaystyle
L^2_{\ast}(a,b)=\left\{f\in L^2(a,b);\quad \int_{a}^{b}fdx=0\right\},&\quad  H_L^1(a,b)=\left\{f\in H^1(a,b);\quad f(a)=0\right\},\\[0.1in]
H^1_{\ast}(a,b)=H^1(a,b)\cap L^2_{\ast}(a,b),& \quad H_R^1(a,b)=\left\{f\in H^1(a,b);\quad f(b)=0\right\}.
\end{array}
$$
The energy space $\mathcal{H}$ is now defined by
\begin{eqnarray}\label{defH}
\begin{array}{c}
\displaystyle
\mathcal{H}=\left\{\left(u,u^1,v,z,p,q,y,y^1\right) \in H_L^1(0,l_1)\times L^2(0,l_1)\times H^1(l_1,l_2)\times L^2(l_1,l_2)\times H_{\ast}^1(0,l_1)\times\right.\\
\displaystyle \left. L_{\ast}^2(l_1,l_2)\times H_R^1(l_2,L)\times L^2(l_2,L)~:~\ u(l_1)=v(l_1),\ y(l_2)=v(l_2)\right\},
\end{array}
\end{eqnarray}
and for $U=\left(u,u^1,v,z,p,q,y,y^1\right)\in \mathcal{H}$, a norm on $\mathcal{H}$ can be chosen of  the following form
\begin{equation}\label{norm}
\begin{array}{ll}
\displaystyle
\|U\|_{\mathcal{H}}^2=& \int_0^{l_1}\left(c_1\abs{u_x}^2+\abs{u^1}^2\right)dx+\int_{l_1}^{l_2}\left(\alpha_1\abs{v_x}^2+\rho\abs{z}^2+\beta\abs{\gamma v_x-p_x}^2+\mu\abs{q}^2\right)dx \\[0.1in]
\displaystyle
&\ma{+\int_{l_2}^{L}\left(c_2\abs{y_x}^2+\abs{y^1}^2\right)dx}.
\end{array}
\end{equation}
noting that the standard norm on $\mathcal{H}$ is
\begin{equation}\label{norm-s}
\begin{array}{l}
\|U\|_{\rm s}^2=\|u_x\|^2_{L^2(0,l_1)}+\|u^1\|_{L^2(0,l_1)}+\|v_x\|^2_{L^2(l_1,l_2)}+\|v\|_{L^2(l_1,l_2)}^2+\|z\|_{L^2(l_1,l_2)}^2+\|p_x\|_{L^2(l_1,l_2)}^2\\
\hspace{1.75cm}+\|q\|_{L^2(l_1,l_2)}^2+\|y_x\|_{L^2(l_2,L)}^2+\|y^1\|_{L^2(l_2,L)}^2.
\end{array}
\end{equation}
\begin{lemma}\label{equiv-norm}
The norm defined by \eqref{norm} is equivalent to the standard norm \eqref{norm-s} on $\mathcal{H}$, i.e. for all \\
$\ma{U=\left(u,u^1,v,z,p,q,y,y^1\right)\in \mathcal{H}}$,  there exist two positive constants $C_1$, $C_2$, independent of $U,$ such that
\begin{equation}\label{equiv-norm1}
C_1\|U\|_{s}^2\leq \|U\|_{\mathcal{H}}^2\leq C_2\|U\|_{s}^2.
\end{equation}
\end{lemma}
\begin{proof}
The inequality on the right with $C_2=\max(c_1,1,\alpha_1+2\beta\max(\gamma^2,1),\mu,\rho,c_2)$ is immediate by Young's inequality since
$$
\beta\|\gamma v_x-p_x\|_{L^2(l_1,l_2)}^2\leq 2\beta \gamma^2\|v_x\|_{L^2(l_1,l_2)}^2+2\beta\|p_x\|_{L^2(l_1,l_2)}^2\leq 2\beta \max(\gamma^2,1)(\|v_x\|_{L^2(l_1,l_2)}^2+\|p_x\|_{L^2(l_1,l_2)}^2).
$$

We have $\displaystyle{u(l_1)=\int_0^{l_1}u_x~dx},$ and  by the transmission condition $\ma{u(l_1)=v(l_1)},$
$$
v(x)=u(l_1)+\int_{l_1}^{x}v_t(t)~dt.
$$
Applying Young's and \ma{Cauchy-Schwarz inequalities} leads to
\begin{eqnarray}
\label{equiv-norm3}
&\abs{u(l_1)}^2\leq l_1\|u_x\|^2_{L^2(0,l_1)},	\\
\label{equiv-norm2}
&\abs{v(x)}^2\leq 2\abs{u(l_1)}^2+2(l_2-l_1)\|v_x\|^2_{L^2(l_1,l_2)}.
\end{eqnarray}
As\eqref{equiv-norm3} and \eqref{equiv-norm2} are considered together
\begin{equation}\label{equiv-norm4}
\int_{l_1}^{l_2}\abs{v(x)}^2dx\leq \underbrace{2(l_2-l_1)\max(l_1,l_2-l_1)}_{:=c_3}\left(\|u_x\|^2_{L^2(0,l_1)}+\|v_x\|^2_{L^2(l_1,l_2)}\right). 	
\end{equation}
Next, Young's inequality is applied to get
\begin{equation}\label{equiv-norm5}
\|p_x\|^2_{L^2(l_1,l_2)}\leq 2\|p_x-\gamma v_x\|^2_{L^2(l_1,l_2)}+2\gamma^2\|v_x\|^2. 	
\end{equation}
By combining \eqref{equiv-norm4} and \eqref{equiv-norm5}
$$
\begin{array}{l}
\|U\|_{s}^2 \leq (1+c_3)\|u_x\|^2_{L^2(0,l_1)}+\|u^1\|_{L^2(0,l_1)}^2+(1+2\gamma^2+c_3)\|v_x\|^2_{L^2(l_1,l_2)}+\|z\|_{L^2(l_1,l_2)}^2\\
\hspace{1.75cm}+2\|p_x-v_x\|_{L^2(l_1,l_2)}^2+\|q\|_{L^2(l_1,l_2)}^2+\|y_x\|_{L^2(l_2,L)}^2+\|y^1\|_{L^2(l_2,L)}^2.
\end{array}
$$
Hence,this leads to the left inequality of  \eqref{equiv-norm1} with
$$
C_1=\frac{1}{\max\left(1,(1+c_3)c_1^{-1},(1+2\gamma^2+c_3)\alpha_1^{-1},\rho^{-1},2\beta^{-1},\mu^{-1},c_2^{-1}\right)}.
$$
\end{proof}

\noindent Define the unbounded linear operator $\mathcal{A}_{EPE}:D(\mathcal{A}_{EPE})\subset \mathcal{H}\rightarrow \mathcal{H}$ by
\begin{eqnarray}\label{defAEPE}
\mathcal{A}_{EPE}\begin{pmatrix}
u\\ u^1\\ v\\ z\\ p\\ q\\ y\\ y^1	
\end{pmatrix}=\begin{pmatrix}
u^1\\
c_1 u_{xx}-d_1u^1\\
z\\
\frac{1}{\rho}\left(\alpha v_{xx}-\gamma\beta p_{xx}-d_2 z\right)\\
q\\
\frac{1}{\mu}\left(\beta p_{xx}-\gamma\beta v_{xx}\right)\\
y^1\\
c_2y_{xx}-d_3 y^1	
\end{pmatrix},\quad \forall U=(u,u^1,v,z,p,q,y,y^1)\in D(\AA_{EPE})
\end{eqnarray}
with the domain
$$
D(\AA_{EPE})=\left\{\begin{array}{l}
U=\left(u,u^1,v,z,p,q,y,y^1\right)\in \mathcal{H};\ u^1\in H_L^1(0,l_1),\ z\in H^1(l_1,l_2),\ q\in H_{\ast}^1(l_1,l_2),\\[0.1in] y^1\in H_R^1(l_2,L),
\,u\in H^2(0,l_1)\cap H_L^1(0,l_1),v\in H^2(l_1,l_2),\ p\in H^2(l_2,l_1)\cap H_{\ast}^{1}(l_1,l_2),\\[0.1in] y\in H^2(l_2,L)\cap H_R^1(l_2,L),~~
\alpha v_x(l_1)-\gamma\beta p_x(l_1)=c_1u_x(l_1),~~\alpha v_x(l_2)-\gamma\beta p_x(l_2)=c_2y_x(l_2),\\[0.1in] \beta p_x(l_1)=\gamma \beta v_x(l_1),~~
\beta p_x(l_2)=\gamma \beta v_x(l_2),~~ u^1(l_1)=z(l_1),~~ \text{and}\quad y^1(l_2)=z(l_2)
\end{array}
\right\}.
$$

\begin{rem}\label{NewCT}
Using \eqref{alpha1}, direct calculations show that
 the transmission conditions
 \[
 \alpha v_x(l_1)-\gamma\beta p_x(l_1)=c_1u_x(l_1) \hbox{ and }
  \beta p_x(l_1)=\gamma \beta v_x(l_1)
 \]
  are equivalent to the  transmission conditions
  \[
  \alpha_1 v_x(l_1)=c_1u_x(l_1) \hbox{ and } \alpha_1p_x(l_1)=c_1\gamma u_x(l_1),
  \]
  while
 the transmission conditions
 \[
 \alpha v_x(l_2)-\gamma\beta p_x(l_2)=c_2y_x(l_2) \hbox{ and }
 \beta p_x(l_2)=\gamma \beta v_x(l_2),\]
  are equivalent to the  transmission conditions
\[
\alpha_1 v_x(l_2)=c_2y_x(l_2) \hbox{ and } \alpha_1 p_x(l_2)=c_2\gamma y_x(l_2).\]
\end{rem}

\noindent If \ma{$(u,v,p,y)$ is a sufficiently regular solution of } the system \eqref{EPE}, it can be  transformed into a first-order evolution equation on the Hilbert space $\mathcal{H}$ as the following
\begin{equation}\label{evolution}
U_t=\mathcal{A}_{EPE}U,\quad U(0)=U_0,	
\end{equation}
where $U=(u,u_1,v,v_t,p,p_t,y,y_t)$ and $U_0=(u_0,u_1,v_0,v_1,p_0,p_1,y_0,y_1)$. \ma{By the arguments of Lemma \ref{denergy}}, for all $U=(u,u^1,v,z,p,q,y,y^1)\in D(\mathcal{A}_{EPE})$,
\begin{equation}\label{diss}
\Re\left(\left<\AA_{EPE}U,U\right>_{\mathcal{H}}\right)=-\int_{l_1}^{l_2}d_2\abs{z}^2dx,	
\end{equation}
which implies that $\mathcal{A}_{EPE}$ is dissipative. Now, let $F=(f_1,f_2,f_3,f_4,f_5,f_6,f_7,f_8)\in \mathcal{H}$. By the Lax-Milgram Theorem, one can prove  the existence of a unique  $U\in D(\mathcal{A}_{EPE})$ of
$$
-\mathcal{A}_{EPE}U=F.
$$
Therefore, the unbounded linear operator $\mathcal{A}_{EPE}$ is $m-$dissipative in $\mathcal{H},$ and consequently, $0\in \rho(\mathcal{A}_{EPE})$. Moreover, $\mathcal{A}_{EPE}$ generates a $C_0-$semigroup of contractions $\left(e^{t\mathcal{A}_{EPE}}\right)_{t\geq 0}$, \ma{by }  the Lumer-Phillips theorem. Therefore, the solution of the Cauchy problem \eqref{evolution} admits the following representation
$$
U(t)=e^{t\mathcal{A}_{EPE}}U_0,\quad t\geq 0,
$$
which leads to the well-posedness of \eqref{evolution}. The following result is immediate.
\begin{theoreme}
Letting $U_{0}\in \mathcal{H}$, the system \eqref{evolution} admits a unique weak solution $U$ satisfying
$$
U\in C^0(\mathbb{R}^+,\mathcal{H}).
$$
Moreover, if $U_0\in D(\mathcal{A}_{EPE})$, the system \eqref{evolution} admits a unique strong solution $U$ satisfying
$$
U\in C^1(\mathbb{R}^{+},\mathcal{H})\cap C^0(\mathbb{R}^+,D(\mathcal{A}_{EPE})).
$$

\end{theoreme}
\subsection{Strong Stability}\label{SStability}
Now the following result is about the strong stability of \eqref{EPE}.
\begin{theoreme}\label{SS-EPE}
The $C_0-$semigroup of contraction $\left(e^{t\mathcal{A}_{EPE}}\right)$ is strongly stable in $\mathcal{H}$; i.e., for all $U_0\in \mathcal{H}$, the solution of \eqref{evolution} satisfies
$$
\lim_{t\to \infty}\|e^{t\mathcal{A}_{EPE}}U_0\|_{\mathcal{H}}=0.
$$	
\end{theoreme}
\begin{proof} Since the resolvent of $\mathcal{A}_{EPE}$ is compact in $\mathcal{H}$, it follows from the Arendt-Batty's theorem (see page 837 in \cite{Arendt01}) that
 the system \eqref{EPE} is strongly
stable if and only if $\mathcal{A}_{EPE}$ does not have pure imaginary eigenvalues, i.e.  $\sigma(\mathcal{A}_{EPE})\cap i\mathbb{R}=\emptyset$. From Section \ref{WP-EPE}, it is already know that
 $0\in \rho(\mathcal{A}_{EPE})$. Therefore, only $\sigma(\mathcal{A}
_{EPE})\cap i\mathbb{R}^{\ast}=\emptyset$ \footnote{as usual, $\mathbb{R}^{\ast}=\mathbb{R}\backslash \{0\}$.} must be proved. For this purpose, suppose  that
there exists a real number $\la\neq 0$ and $U=(u,u^1,v,z,p,q,y,y^1)\in D(\mathcal{A}
_{EPE})$ such that
\begin{equation}\label{SS-EQ1}
\mathcal{A}_{EPE}U=i\la U.	
\end{equation}
which is equivalent to the following system
\begin{equation}\label{SS-EQ2}
u^1=i\la u\ \text{in}\ (0,l_1),\quad z=i\la v\ \text{in}\ (l_1,l_2),\quad q=i\la p\ \text{in}\ (l_1,l_2),\quad 	y^1=i\la y\ \text{in}\ (l_2,L),
\end{equation}
and
\begin{equation}\label{SS-EQ3}
\left\{\begin{array}{ll}
\la^2u+c_1u_{xx}=0,& x\in (0,l_1),\\
\rho \la^2v+\alpha v_{xx}-\gamma\beta p_{xx}-d_2z=0,&x\in(l_1,l_2),\\
\mu \la^2p+\beta p_{xx}-\gamma \beta v_{xx}=0,&x\in(l_1,l_2),\\
\la^2 y+c_2 y_{xx}=0,&x\in(l_2,L), 	
\end{array}
\right.	
\end{equation}
From \eqref{diss}, \eqref{LD-P} and  \eqref{SS-EQ1},
\begin{equation}\label{SS-EQ4}
0=\Re \left(i\la U,U\right)_{\mathcal{H}}=\Re\left(\mathcal{A}_{EPE}
U,U\right)_{\mathcal{H}}=-\int_{l_1}^{l_2}d_2\abs{z}^2dx.
\end{equation}
On the other hand, from \eqref{SS-EQ2}, \eqref{SS-EQ4}, \eqref{LD-P} and the fact that $\la\neq 0$, we have
\begin{equation}\label{SS-EQ5}
d_2 z=0\ \text{in}\ (l_1,l_2)\ \text{and consequently}\quad z=v=0, \quad x\in(a_2,b_2).
\end{equation}
By $\alpha=\alpha_1+\gamma^2\beta$ and \eqref{SS-EQ5} in $\eqref{SS-EQ3}_2$,
\begin{equation}\label{SS-EQ6}
\rho \la^2 v+\alpha_1 v_{xx}+\gamma\left(\gamma \beta v_{xx}-\beta p_{xx}\right)=0,\quad x\in (l_1,l_2). 	
\end{equation}
Combining $\eqref{SS-EQ3}_{3}$ and \eqref{SS-EQ6} leads to
\begin{equation}\label{SS-EQ7}
\la^2(\rho v+\gamma\mu p)+\alpha_1 v_{xx}=0,\quad \quad x\in (l_1,l_2). 	
\end{equation}
Next, by \eqref{SS-EQ5} in \eqref{SS-EQ7} and  $\la\neq 0$ \ma{we get $p=0$ in $(a_2,b_2)$},  the third equation in \eqref{SS-EQ2} yields
\begin{equation}\label{SS-EQ8}
p=q=0\quad \text{in}\quad (a_2,b_2). 	
\end{equation}
Since $v,p\in H^2(l_1,l_2)\subset C^1([a_2,b_2])$,
\begin{equation}\label{SS-EQ9}
v(\zeta)=v_x(\zeta)=p(\zeta)=p_x(\zeta)=0, \quad \zeta\in \{a_2,b_2\}. 	
\end{equation}
Now, combining \eqref{SS-EQ7} and $\eqref{SS-EQ3}_3$, the following reduced system is obtained
\begin{eqnarray}
v_{xx}&=&-\la^2\alpha_1^{-1}\left(\rho v+\gamma\mu p\right),\quad x\in (l_1,l_2)\label{SS-EQ10}\\
p_{xx}&=&-\la^2\alpha_1^{-1}\left(\gamma\rho v+\mu\alpha\beta^{-1}p\right),\quad x \in (l_1,l_2)\label{SS-EQ11}.	
\end{eqnarray}
Let $U_{piezo}=(v,v_x,p,p_x)^{\top}$. From \eqref{SS-EQ9},  $U_{piezo}(b_2)=0$. Now, the system \eqref{SS-EQ10}-\eqref{SS-EQ11} can be written in $(b_2,l_2)$ as the following
\begin{equation}\label{SS-EQ12}
\left(U_{piezo}\right)_x=B U_{piezo}\quad \text{in}\quad (b_2,l_2),	
\end{equation}
where
$$
B=\begin{pmatrix}
0&1&0&0\\
-\rho \alpha_1^{-1}\la^2&0&-\gamma\mu \alpha_1^{-1}\la^2&0\\
0&0&0&1\\
-\rho\gamma\alpha_1^{-1}\la^2&0&-\mu\alpha\beta^{-1}\alpha_1^{-1}\la^2&0	
\end{pmatrix}.
$$
The solution of the differential equation \eqref{SS-EQ12} is given by
\begin{equation}\label{SS-EQ13}
U_{piezo}(x)=e^{B(x-b_2)}U_{piezo}(b_2)=0\quad \text{in}\quad (b_2,l_2).	
\end{equation}
Analogously, it can be proved that $U_{piezo}=0$ in $(l_1,a_2)$. Consequently,  $v=p=0$ in $(l_1,l_2)$. Since $v,p\in H^2(l_1,l_2)\subset C^1([l_1,l_2])$,
\begin{equation}\label{SS-EQ14}
v(\zeta)=v_x(\zeta)=p(\zeta)=p_x(\zeta)=0\quad \text{where}\quad \zeta\in \{l_1,l_2\}. 	
\end{equation}
By $U\in D(\mathcal{A}_{EPE})$, the continuity and transmission conditions,
\begin{equation}\label{SS-EQ15}
u(0)=u(l_1)=u_x(l_1)=y(l_2)=y_x(l_2)=y(L)=0.
\end{equation}
Finally, by $\eqref{SS-EQ3}_1$, $\eqref{SS-EQ3}_4$ and \eqref{SS-EQ15} it is easy to conclude that
$
u=0$ in $(0,l_1)$ and $y=0$ in $(l_2,L)$.
Hence,  $U=0$. The proof is thus complete.
\end{proof}
\subsection{Exponential Stability} The aim of the subsection is to prove the Thereom \ref{EPE-EXP}, which is the exponential
stability of System \eqref{EPE} \sn{under the sole assumption  \eqref{LD-P}}. 

 Before diving into the technicality of the proof of Theorem \ref{EPE}, recall from, e.g. \cite{Huang01}, \cite{pruss01}, that a $C_0-$semigroup of contractions
$\left(e^{t\mathcal{A}_{EPE}}\right)_{t\geq 0}$ on $\mathcal{H}$ must satisfy two conditions
\eqref{EPE-EXP-EQ1} if
\begin{equation}\label{M1}\tag{$\rm{M1}$}
i\mathbb{R}\subset \rho\left(\mathcal{A}_{EPE}\right)	
\end{equation}
\begin{equation}\label{M2}\tag{$\rm{M2}$}
\sup_{\la \in \mathbb{R}}\|\left(i\la I-\mathcal{A}_{EPE}\right)^{-1}\|_{\mathcal{L}(\mathcal{H})}<+\infty. 	
\end{equation}
\ma{Since we already proved in Theorem \ref{SS-EPE} that} $i\mathbb{R}\subset \rho(\mathcal{A}_{EPE})$,  condition \eqref{M1} is satisfied. Now only the condition \eqref{M2} must be proved. We follow a contradiction argument,  for this purpose, suppose that \eqref{M2} is false, then there exists $\{(\la^n,U^n)\}_{n\geq 1}\subset \mathbb{R}^{\ast}\times D(\mathcal{A}_{EPE})$ with
\begin{equation}\label{EPE-EXP-EQ2}
\abs{\la^n}\to \infty\quad \text{and}\quad \|U^n\|_{\mathcal{H}}=\|\left(u^n,u^{1,n},v^n,z^n,p^n,q^n,y^n,y^{1,n}\right)^{\top}\|_{\mathcal{H}}=1, 	
\end{equation}
such that
\begin{equation}\label{EPE-EXP-EQ3}
\left(i\la^nI-\mathcal{A}_{EPE}\right)U^n= \mathcal{F}^n:=\left(f^{1,n},f^{2,n},f^{3,n},f^{4,n},f^{5,n},f^{6,n},f^{7,n},f^{8,n}\right)^{\top}\to 0\quad \text{in}\quad \mathcal{H}.
\end{equation}
For simplicity, let the index $n$ be dropped. Then, \eqref{EPE-EXP-EQ3} is equivalent to
\begin{equation}\label{P1-EXP}
\left\{\begin{array}{lll}
i\la u-u^1=f^1\to 0&\text{in}&H^1_L(0,l_1),\\
i\la v-z=f^3\to 0&\text{in}&H^1(l_1,l_2),\\
i\la p-q=f^5\to 0&\text{in}&H^1_{\ast}(l_1,l_2),\\
i\la y-y^1=f^7\to 0&\text{in}&H^1_{R}(l_2,L),\\
\end{array}
\right.
\end{equation}
and
\begin{equation}\label{P2-EXP}
\left\{\begin{array}{lll}
i\la u^1-c_1u^1_{xx}=f^2\to 0&\text{in}&L^2(0,l_1),\\
i\la \rho z-\alpha v_{xx}+\gamma \beta p_{xx}+d_2z=\rho f^4\to 0&\text{in}&L^2(l_1,l_2),\\
i\la \mu q-\beta p_{xx}+\gamma \beta v_{xx}=\mu f^6\to 0&\text{in}&L^2(l_1,l_2),\\
i\la y^1-c_2y_{xx}=f^8\to 0&\text{in}&L^2(l_2,L).
\end{array}
\right.	
\end{equation}
Merging \eqref{P1-EXP} and \eqref{P2-EXP}, a more compact system of equations is obtained
\begin{equation}\label{P3-EXP}
\left\{\begin{array}{lll}
\la^2u+c_1u_{xx}=F^1,\\
\la^2\rho v+\alpha v_{xx}-\gamma\beta p_{xx}-i\la d_2v=F^2,\\
\la^2\mu p+\beta p_{xx}-\gamma\beta v_{xx}=F^3,\\
\la^2y+c_2y_{xx}=F^4,
\end{array}
\right.
\end{equation}
where
\begin{equation}\label{F1234}
\left\{\begin{array}{l}
\displaystyle
F^1=-\left(f^2+i\la f^1\right),\ F^2=-\left(\rho f^4+ d_2 f^3+i\la \rho f^3\right),\\[0.1in] \displaystyle
F^3=-\left(\mu f^6+i\la \mu f^5\right)\quad \text{and}\quad F^4=-(f^8+i\la f^7).
\end{array}
\right.
\end{equation}
By $\alpha=\alpha_1+\gamma^2\beta$ in $\eqref{P3-EXP}_{2}$,
\begin{equation*}
\la^2\rho v+\alpha_1v_{xx}+\gamma\left(\gamma\beta v_{xx}-\beta p_{xx}\right)-i\la d v=F^2.	
\end{equation*}
Now, combining $\eqref{P3-EXP}_3$ and the above equality lead to
$$
\alpha_1 v_{xx}=-\la^2\rho v-\gamma\la^2\mu p+i\la d_2 v+F^2+\gamma F^3.
$$
Inserting the above equation in $\eqref{P3-EXP}_3$, the system is reduced to
\begin{equation}\label{P4-EXP}
\left\{\begin{array}{l}
\la^2u+c_1u_{xx}=F^1,\\
\la^2\rho v+\alpha_1v_{xx}+\gamma\mu \la^2p-i\la d_2v=F^5,\\
\la^2\mu\alpha p+\alpha_1\beta p_{xx}+\rho\gamma\beta\la^2v-i\la \gamma\beta d_2v=F^6,\\
\la^2y+c_2y_{xx}=F^4,
\end{array}
\right.	
\end{equation}
where
\begin{equation}\label{F56}
F^5=F^2+\gamma F^3\quad \text{and}\quad F^6=\alpha F^3+\gamma\beta F^2. 	
\end{equation}
At this moment, the following series of technical lemmas, as consequences of the dissipativity property of the solutions $(u,u^1,v,z,p,q,y,y^1)$ of the system \eqref{P1-EXP}-\eqref{P2-EXP}, are needed to finish the proof of Theorem \ref{EPE}.
\begin{lemma}\label{Lemma1}
The solution $(u,u^1,v,z,p,q,y,y^1)$ of the system \eqref{P1-EXP}-\eqref{P2-EXP} satisfies the following estimates
\begin{equation}\label{Lemma1-EQ1}
\int_{l_1}^{l_2}d_2\abs{z}^2dx=o(1),\quad \int_{l_1}^{l_2}d_2\abs{\la v}^2dx=o(1)\quad \text{and}\quad \int_{a_2}^{b_2}\abs{\la v}^2dx=o(1).
\end{equation}
\end{lemma}
\begin{proof}
To get the first estimate in \eqref{Lemma1-EQ1}, take the inner product of \eqref{EPE-EXP-EQ3} with $U$ in $\mathcal{H}$, and use  $\|U\|_{\mathcal{H}}=1$ and $\|\mathcal{F}\|_{\mathcal{H}}=o(1)$ so that
\begin{equation}\label{Lemma1-EQ2}
\int_{l_1}^{l_2}d_2\abs{z}^2dx=-\Re\left(\left<\mathcal{A}_{EPE}U,U\right>_{\mathcal{H}}\right)= \Re\left<(i\la I-\mathcal{A}_{EPE})U,U\right>_{\mathcal{H}} = \Re\left<\mathcal{F},U\right>_{\mathcal{H}}=o(1). 	
\end{equation}
Next, by multiplying $\eqref{P1-EXP}_{2}$ by $\sqrt{d_2}$, using the first estimation in \eqref{Lemma1-EQ1}, and $\|\mathcal{F}\|_{\mathcal{H}}=o(1)$, the second estimate in \eqref{Lemma1-EQ1} is immediate. Finally, by \eqref{LD-P} and the second estimate in \eqref{Lemma1-EQ1}, the third estimate  in \eqref{Lemma1-EQ1} is obtained.  	
\end{proof}
\newline
\\
\noindent Note that for all $0<\varepsilon<\frac{b_2-a_2}{4}$, the following cut-off functions are fixed
\begin{enumerate}
\item[$\bullet$] $\theta_k\in C^2([l_1,l_2])$, $k\in \left\{1,2\right\}$ such that $0\leq \theta_k(x)\leq 1$, for all $x\in [l_1,l_2]$ and
$$
\theta_k(x)=\left\{\begin{array}{lll}
1&\text{if}&x\in [a_2+k\varepsilon,b_2-k\varepsilon],\\
0&\text{if}&x\in [\ma{l_1},a_2+(k-1)\varepsilon]\cup [b_2+(1-k)\varepsilon,\ma{l_2}]. 	
\end{array}
\right.
$$
\end{enumerate}
\ma{Observe  that $\theta_1\equiv 1$ on the support of $\theta_2$.}
\begin{lemma}\label{Lemma2}
The solution $(u,u^1,v,z,p,q,y,y^1)$ of the system \eqref{P1-EXP}-\eqref{P2-EXP} satisfies the following estimates
\begin{equation}\label{Lemma2-EQ1}
\int_{l_1}^{l_2}\theta_1\abs{\la p}^2dx=o(1),\quad \int_{D_{\varepsilon}}\abs{\la p}
^2dx=o(1),\quad \text{and}\quad \int_{D_{\varepsilon}}\abs{q}
^2dx=o(1),	
\end{equation}
where $D_{\varepsilon}:=(a_2+\varepsilon,b_2-\varepsilon)$ with a positive real number $\varepsilon$ small enough such that $\varepsilon<\frac{b_2-a_2}{4}$.	
\end{lemma}

\begin{proof}
First, multiply $\eqref{P4-EXP}_2$ by $\beta \theta_1\overline{p}$, integrate over $(l_1,l_2)$ by parts, and use definition of $\theta_1$ to get
\begin{equation}\label{Lemma2-EQ2}
\begin{array}{l}
\displaystyle
\la^2\rho\beta\int_{l_1}^{l_2}\theta_1v\overline{p}dx-\alpha_1\beta\int_{l_1}^{l_2}\theta_1v_x\overline{p}_xdx-\alpha_1\beta\int_{l_1}^{l_2}\theta_1'v_x\overline{p}dx\\
\displaystyle
+\gamma\mu\beta\int_{l_1}^{l_2}\theta_1\abs{\la p}^2dx-i\la \beta\int_{l_1}^{l_2}d_2\theta_1 v\overline{p}dx=\beta \int_{l_1}^{l_2}\theta_1F^5\overline{p}dx.	
\end{array}	
\end{equation}	
It is known that  $\|U\|_{\mathcal{H}}=1$ and $\|\mathcal{F}\|_{\mathcal{H}}=o(1),$ which implies in particular that $(\la p)$ is uniformly bounded in $L^2(l_1,l_2)$ due to $\eqref{P1-EXP}_{3}$). Therefore, by Cauchy-Schwarz inequality, Lemma \ref{Lemma1}, the definition of $\theta_1$ the following is deduced
\begin{equation*}
\left|\la^2\rho\beta\int_{l_1}^{l_2}\theta_1v\overline{p}dx\right|=o(1),\  \left|\int_{l_1}^{l_2}\theta_1'v_x\overline{p}dx\right|=O(\la^{-1})=o(1),\  \left|i\la \beta\int_{l_1}^{l_2}d_2\theta_1 v\overline{p}dx\right|=o(\la^{-1}),\  \left|\int_{l_1}^{l_2}\theta_1F^5\overline{p}dx\right|=o(1).	
\end{equation*}
Inserting the above estimates into  \eqref{Lemma2-EQ2} and taking the real part leads to
\begin{equation}\label{Lemma2-EQ3}
\gamma\mu\beta\int_{l_1}^{l_2}\theta_1\abs{\la p}^2dx-\alpha_1\beta\Re\left(\int_{l_1}^{l_2}\theta_1v_x\overline{p}_xdx\right)=o(1).	
\end{equation}
Analogously, multiply $\eqref{P4-EXP}_3$ by $-\theta_1\overline{v}$, integrate over $(l_1,l_2)$ by parts to get
\begin{equation}\label{Lemma2-EQ4}
\begin{array}{l}
\displaystyle 	
-\la^2\mu\alpha\int_{l_1}^{l_2}\theta_1p\overline{v}dx+\alpha_1\beta\int_{l_1}
^{l_2}\theta_1'p_x\overline{v}dx+\alpha_1\beta\int_{l_1}
^{l_2}\theta_1p_x\overline{v_x}dx\\
\displaystyle
-\rho\gamma\beta\int_{l_1}^{l_2}\theta_1\abs{\la v}^2dx+i\la \gamma\beta\int_{l_1}
^{l_2}\theta_1d_2\abs{v}^2dx=-\int_{l_1}^{l_2}\theta_1F^6\overline{v}dx.
\end{array}
\end{equation}
By the  definition of $\theta_1, $ $\|U\|_{\mathcal{H}}=1$, $\|\mathcal{F}\|_{\mathcal{H}}=o(1)$,  Cauchy-Schwarz inequality, Lemma \ref{Lemma1}
\begin{equation}\label{Lemma2-EQ5}
\left|\la^2\mu\alpha\int_{l_1}^{l_2}\theta_1 p\overline{v}dx\right|=o(1), \left|\int_{l_1}
^{l_2}\theta_1'p_x\overline{v}dx\right|=o(\la^{-1})\quad \text{and}\quad  \left|
\int_{l_1}^{l_2}\theta_1F^6\overline{v}dx\right|=o(1). 	
\end{equation}
Inserting the estimates above  into \eqref{Lemma2-EQ4} together with Lemma \ref{Lemma1}  yields
\begin{equation}\label{Lemma2-EQ6}
\Re\left(\int_{l_1}^{l_2}\theta_1p_x\overline{v_x}dx\right)=o(1).	
\end{equation}
Thus, the combination of \eqref{Lemma2-EQ3} and \eqref{Lemma2-EQ6} gets the first estimate in
\eqref{Lemma2-EQ1}, and together with which, and the definition
of $\theta_1$,  the second estimate in \eqref{Lemma2-EQ1} is immediate. Finally, by the second
estimate in \eqref{Lemma2-EQ1}, $\eqref{P1-EXP}_3$ and the fact that $\|\mathcal{F}\|
_{\mathcal{H}}=o(1)$,  the third estimate in \eqref{Lemma2-EQ1} is obtained.
\end{proof}
\begin{lemma}\label{Lemma3}
The solution $(u,u^1,v,z,p,q,y,y^1)$ of the system \eqref{P1-EXP}-\eqref{P2-EXP} satisfies the following estimate
\begin{equation}\label{Lemma3-EQ1}
\int_{l_1}^{l_2}\theta_1\abs{v_x}^2dx=o(1),\quad \text{and consequently}\quad \int_{D_{\varepsilon}}\abs{v_x}^2dx=o(1). 	
\end{equation}	
\end{lemma}
\begin{proof}
Multiplying $\eqref{P4-EXP}_2$ by $-\theta_1\overline{v}$ and integrating over $(l_1,l_2)$  by parts yield
\begin{equation}\label{Lemma3-EQ2}
\begin{array}{l}
\displaystyle
-\rho\int_{l_1}^{l_2}\theta_1\abs{\la v}^2dx+\alpha_1\int_{l_1}^{l_2}\theta_1'v_x\overline{v}dx+\alpha_1\int_{l_1}^{l_2}\theta_1\abs{v_x}^2dx-\gamma\mu \int_{l_1}^{l_2}\theta_1\la^2p\overline{v}dx\\
\displaystyle
+i\la \gamma\beta \int_{l_1}^{l_2}\theta_1 d_2\abs{v}^2dx=-\int_{l_1}^{l_2}\theta_1F^5\overline{v}dx.
\end{array}	
\end{equation}
By Cauchy-Schwarz inequality, Lemma \ref{Lemma1}, the definition of $\theta_1$, $\|U\|_{\mathcal{H}}=1,$ and $\|\mathcal{F}\|_{\mathcal{H}}=o(1)$, the following hold
$$
\left|\int_{l_1}^{l_2}\theta_1'v_x\overline{v}dx\right|=o(\la^{-1}),\ \left|\int_{l_1}^{l_2}\theta_1\la^2p\overline{v}dx\right|=o(1) \quad \text{and}\quad \int_{l_1}^{l_2}\theta_1F^5\overline{v}dx=o(1).
$$	
Finally, inserting these estimates in \eqref{Lemma3-EQ2} and by \eqref{Lemma1-EQ1}, the first estimate, and therefore the second estimate, in \eqref{Lemma3-EQ1} are obtained.
\end{proof}
\begin{lemma}\label{Lemma4}
The solution $(u,u^1,v,z,p,q,y,y^1)$ of the system \eqref{P1-EXP}-\eqref{P2-EXP} satisfies the
following estimates
\begin{equation}\label{Lemma4-EQ1}
\int_{l_1}^{l_2}\theta_2\abs{p_x}^2dx=o(1),\quad \text{and consequently }\quad \int_{D_{2\varepsilon}}
\abs{p_x}^2dx=o(1)
\end{equation}
\ma{where $D_{2\varepsilon}:=(a_2+2\varepsilon,b_2-2\varepsilon)$ with a positive real number $\varepsilon$ small enough so that $\varepsilon<\frac{b_2-a_2}{4}$}.		
\end{lemma}
\begin{proof}
Multiplying $\eqref{P4-EXP}_3$ by $-\theta_2\overline{p}$ integrating over $(l_1,l_2)$ by parts lead to
\begin{equation}\label{Lemma4-EQ2}
\begin{array}{l}
\displaystyle
-\mu \alpha \int_{l_1}^{l_2}\theta_2\abs{\la p}^2dx+\alpha_1\beta \int_{l_1}
^{l_2}\theta_2'p_x\overline{p}dx+\alpha_1\beta\int_{l_1}^{l_2}\theta_2\abs{p_x}^2dx-
\rho\gamma\beta\int_{l_1}^{l_2}\la^2\theta_2v\overline{p}dx\\
\displaystyle
+i\la \gamma\beta\int_{l_1}^{l_2}d_2\theta_2v\overline{p}dx=-\int_{l_1}^{l_2}
F^6\theta_2\overline{p}dx.
\end{array}	
\end{equation}
By Cauchy-Schwarz inequality, Lemmas \ref{Lemma1}, \ref{Lemma3}, the definition of $
\theta_2$, $\|U\|_{\mathcal{H}}=1,$ and $\|\mathcal{F}\|_{\mathcal{H}}=o(1)$,
$$
\left|\int_{l_1}^{l_2}\theta_2'p_x\overline{p}dx\right|=o(\la^{-1}),\ \left|\int_{l_1}
^{l_2}\la^2\theta_2v\overline{p}dx\right|=o(1),\ \left|i\la\int_{l_1}^{l_2}
d_2\theta_2v\overline{p}dx\right|=o(\la^{-1})\quad \text{and}\quad \left|\int_{l_1}
^{l_2}F^6\theta_2\overline{p}dx\right|=o(1).
$$
Finally, inserting the above estimates into \eqref{Lemma4-EQ2} and by Lemma \ref{Lemma2}, the first estimate, and therefore the second estimate, in \eqref{Lemma4-EQ1} are obtained.
\end{proof}
\begin{lemma}\label{TL}
Let $g\in C^1([l_1,l_2])$. 	The solution $(u,u^1,v,z,p,q,y,y^1)$ of the system \eqref{P1-EXP}-\eqref{P2-EXP} satisfies the
following estimate
\begin{equation}\label{TL-EQ1}
\begin{array}{l}
\displaystyle
\int_{l_1}^{l_2}g'\left(\rho\abs{\la v}^2+\alpha\abs{v_x}^2+\mu\abs{\la p}^2+\beta\abs{p_x}^2\right)dx-2\gamma\beta\Re\left(\int_{l_1}^{l_2}g'p_{x}\overline{v_x}dx\right)+\mathcal{J}_1(l_1)-\mathcal{J}_1(l_2)\\[0.1in]
\displaystyle
=\mathcal{J}_2(l_2)-\mathcal{J}_2(l_1)+o(1)	
\end{array}
\end{equation}
where $\zeta\in \{l_1,l_2\}$ and
\begin{equation}\label{TL-J}
\left\{\begin{array}{l}
\displaystyle
\mathcal{J}_1(\zeta)=g(\zeta)\left(\rho\abs{\la
v(\zeta)}^2+\alpha\abs{v_x(\zeta)}^2+\mu\abs{\la p(\zeta)}^2+\beta\abs{p_x(\zeta)}^2\right)-2\gamma\beta\Re\left(g(\zeta)p_x(\zeta)\overline{v_x}(\zeta)\right),	\\[0.1in]
\displaystyle
\mathcal{J}_2(\zeta)=2\Re\left(i\la \rho g(\zeta)f^3(\zeta)\overline{v}(\zeta)\right)+2\Re\left(i\la \mu g(\zeta)f^5(\zeta)\overline{p}(\zeta)\right).
\end{array}	
\right.
\end{equation}
\end{lemma}
\begin{proof}
First, multiply $\eqref{P3-EXP}_2$ by $-2g\overline{v}_x$, and integrate over $(l_1,l_2)$ to get
\begin{equation*}
\begin{array}{l}
\displaystyle
-\rho \int_{l_1}^{l_2}g\left(\abs{\la v}^2\right)_xdx-\alpha \int_{l_1}^{l_2}g\left(\abs{v_x}^2\right)_xdx+2\gamma\beta\Re\left(\int_{l_1}^{l_2}gp_{xx}\overline{v_x}dx\right)\\
\displaystyle
+2\Re\left(i\la \int_{l_1}^{l_2}gd_2v\overline{v_x}dx\right)=-2\Re\left(\int_{l_1}^{l_2}gF^2\overline{v_x}dx\right). 	
\end{array}	
\end{equation*}
By several integration by parts and the definition of $F^2$ in  \eqref{F1234}, the following is obtained
\begin{equation}\label{TL-EQ2}
\begin{array}{l}
\displaystyle
\int_{l_1}^{l_2}g'\left(\rho\abs{\la v}^2+\alpha\abs{v_x}^2\right)dx-
g(l_2)\left[\rho\abs{\la v(l_2)}^2+\alpha\abs{v_x(l_2)}^2\right]+g(l_1)\left[\rho\abs{\la
v(l_1)}^2+\alpha\abs{v_x(l_1)}^2\right]\\[0.1in]
\displaystyle
-2\gamma\beta\Re\left(\int_{l_1}^{l_2}g'p_{x}\overline{v_x}dx\right)-2\gamma\beta\Re\left(\int_{l_1}^{l_2}gp_{x}\overline{v_{xx}}dx\right)+2\gamma\beta\Re\left(g(l_2)p_x(l_2)\overline{v_x}(l_2)\right)\\[0.1in]
\displaystyle
-2\gamma\beta\Re\left(g(l_1)p_x(l_1)\overline{v_x}(l_1)\right)-2\Re\left(i\la \int_{l_1}^{l_2}gd_2v\overline{v_x}dx\right)=2\Re\left(\int_{l_1}^{l_2}g(f^4+d_2f^3)\overline{v_x}dx\right)\\[0.1in]
\displaystyle
-2\Re\left(i\la \rho\int_{l_1}^{l_2}(f^3g)_x\overline{v}dx\right)+2\Re\left(i\la \rho g(l_2)f^3(l_2)\overline{v}(l_2)\right)-2\Re\left(i\la \rho g(l_1)f^3(l_1)\overline{v}(l_1)\right).
\end{array}	
\end{equation}
Since $v_x$ is uniformly bounded in $L^2(l_1,l_2)$ and $\|\mathcal{F}\|_{\mathcal{H}}=o(1)$, by Cauchy-Schwarz inequality and Lemma \ref{Lemma1}
\begin{equation*}
\left|\Re\left(i\la \int_{l_1}^{l_2}gd_2v\overline{v_x}dx\right)\right|=o(1), \left|Re\left(\int_{l_1}^{l_2}g(f^2+d_2f^3)\overline{v_x}dx\right)\right|=o(1), \left|\Re\left(i\la \rho\int_{l_1}^{l_2}(f^3g)_x\overline{v}dx\right)\right|=o(1).	
\end{equation*}
Substituting  the estimatation above into \eqref{TL-EQ2} leads to
\begin{equation}\label{TL-EQ3}
\begin{array}{l}
\displaystyle
\int_{l_1}^{l_2}g'\left(\rho\abs{\la v}^2+\alpha\abs{v_x}^2\right)dx-
g(l_2)\left[\rho\abs{\la v(l_2)}^2+\alpha\abs{v_x(l_2)}^2\right]+g(l_1)\left[\rho\abs{\la
v(l_1)}^2+\alpha\abs{v_x(l_1)}^2\right]\\[0.1in]
\displaystyle
-2\gamma\beta\Re\left(\int_{l_1}^{l_2}g'p_{x}\overline{v_x}dx\right)-2\gamma\beta\Re\left(\int_{l_1}^{l_2}gp_{x}\overline{v_{xx}}dx\right)+2\gamma\beta\Re\left(g(l_2)p_x(l_2)\overline{v_x}(l_2)\right)\\[0.1in]	
\displaystyle
-2\gamma\beta\Re\left(g(l_1)p_x(l_1)\overline{v_x}(l_1)\right)=2\Re\left(i\la \rho g(l_2)f^3(l_2)\overline{v}(l_2)\right)-2\Re\left(i\la \rho g(l_1)f^3(l_1)\overline{v}(l_1)\right)+o(1).
\end{array}	
\end{equation}
Analogously, multiply $\eqref{P3-EXP}_3$ by $-2g\overline{p}_x$ and integrate over $(l_1,l_2)$ to obtain
and  taking the real part, we get
\begin{equation*}
\begin{array}{l}
\displaystyle
-\mu\int_{l_1}^{l_2}g(\abs{\la p}^2)_xdx-\beta\int_{l_1}^{l_2}g\left(\abs{p_x}^2\right)_xdx+2\gamma\beta\Re\left(\int_{l_1}^{l_2}gv_{xx}\overline{p_x}dx\right)=-2\Re\left(\int_{l_1}^{l_2}gF^3\overline{p_x}dx\right).	
\end{array}	
\end{equation*}
By several integration by parts and the definition of $F^3$ given in \eqref{F1234}, the following holds
\begin{equation}\label{TL-EQ4}
\begin{array}{l}
\displaystyle
\int_{l_1}^{l_2}g'\left(\mu\abs{\la p}^2+\beta\abs{p_x}^2\right)dx-g(l_2)\left[\mu\abs{\la p(l_2)}^2+\beta\abs{p_x(l_2)}^2\right]+g(l_1)\left[\mu\abs{\la p(l_1)}^2+\beta\abs{p_x(l_1)}^2\right]\\[0.1in]
\displaystyle
+2\gamma\beta\Re\left(\int_{l_1}^{l_2}gv_{xx}\overline{p_x}dx\right)=2\Re\left(\int_{l_1}^{l_2}gf^6\overline{p_x}dx\right)-2\Re\left(i\la \mu\int_{l_1}^{l_2}(gf^5)_x\overline{p}dx\right)\\[0.1in]
\displaystyle
+2\Re\left(i\la \mu g(l_2)f^5(l_2)\overline{p}(l_2)\right)-2\Re\left(i\la \mu g(l_1)f^5(l_1)\overline{p}(l_1)\right).	
\end{array}	
\end{equation}
Since $p_x$ and $\la p$ are uniformly bounded in $L^2(l_1,l_2)$ and $\|\mathcal{F}\|_{\mathcal{H}}=o(1),$  Cauchy-Schwarz inequality is implemented to obtain
$$
\left|\Re\left(\int_{l_1}^{l_2}gf^6\overline{p_x}dx\right)\right|=o(1)\quad \text{and}\quad  \left|\Re\left(i\la \mu\int_{l_1}^{l_2}(gf^5)_x\overline{p}dx\right)\right|=o(1).
$$
Substituting the estimates above into \eqref{TL-EQ4} results in
\begin{equation}\label{TL-EQ5}
\begin{array}{l}
\displaystyle
\int_{l_1}^{l_2}g'\left(\mu\abs{\la p}^2+\beta\abs{p_x}^2\right)dx-g(l_2)\left[\mu\abs{\la p(l_2)}^2+\beta\abs{p_x(l_2)}^2\right]+g(l_1)\left[\mu\abs{\la p(l_1)}^2+\beta\abs{p_x(l_1)}^2\right]\\[0.1in]
\displaystyle
+2\gamma\beta\Re\left(\int_{l_1}^{l_2}gv_{xx}\overline{p_x}dx\right)=2\Re\left(i\la \mu g(l_2)f^5(l_2)\overline{p}(l_2)\right)-2\Re\left(i\la \mu g(l_1)f^5(l_1)\overline{p}(l_1)\right)+o(1).
\end{array}	
\end{equation}
Finally, adding \eqref{TL-EQ3} and \eqref{TL-EQ5}, the desired result \eqref{TL-EQ1} is obtained.
\end{proof}
\begin{lemma}\label{Lemma5}
The solution $(u,u^1,v,z,p,q,y,y^1)$ of the system \eqref{P1-EXP}-\eqref{P2-EXP} satisfies the
following estimates
\begin{equation}\label{Lemma5-EQ1}
\int_{l_1}^{l_2}\abs{\la v}^2dx=o(1),\ \int_{l_1}^{l_2}\abs{v_x}^2=o(1),\ \int_{l_1}
^{l_2}\abs{\la p}^2dx=o(1)\quad \text{and}\quad \int_{l_1}^{l_2}\abs{p_x}^2dx=o(1). 	
\end{equation}	
\end{lemma}
\begin{proof}
First, define two cut-off functions  $\theta_3,\theta_4\in C^2([l_1,l_2])$ such that $0\leq \theta_3,\theta_4\leq 1$ for all $x\in [l_1,l_2]$ and
\begin{equation}\label{theta34}
\theta_3(x)=\left\{\begin{array}{lll}
1&\text{if}&x\in [l_1,a_1+2\varepsilon],\\
0&\text{if}&x\in [a_2-2\varepsilon,l_2]	
\end{array}
\right.,	\quad \text{and}\quad \theta_4(x)=\left\{\begin{array}{lll}
0&\text{if}&x\in [l_1,a_1+2\varepsilon],\\
1&\text{if}&x\in [a_2-2\varepsilon,l_2].		
\end{array}
\right.
\end{equation}
By taking $g(x)=(x-l_1)\theta_3(x)+(x-l_2)\theta_4(x)$,  it is easy to see that
$$
g'(x)=\theta_3(x)+\theta_4(x)+(x-l_1)\theta_3'(x)+(x-l_2)\theta_4'(x)\quad \text{and}\quad g(l_1)=g(l_2)=0.
$$
Setting $\tilde{g}=(x-l_1)\theta_3'+(x-l_2)\theta_4'$ and considering  $g$ defined above  in \eqref{TL-EQ1} lead to
\begin{equation}\label{Lemma5-EQ2}
\begin{array}{l}
\displaystyle
\int_{l_1}^{l_2}(\theta_3+\theta_4)\left(\rho\abs{\la v}^2+\alpha\abs{v_x}^2+\mu\abs{\la
p}^2+\beta\abs{p_x}^2\right)dx-2\gamma\beta\Re\left(\int_{l_1}^{l_2}(\theta_3+
\theta_4)p_{x}\overline{v_x}dx\right)\\
\displaystyle
-\underbrace{\int_{l_1}^{l_2}\tilde{g}\left(\rho\abs{\la v}^2+\alpha\abs{v_x}^2+\mu\abs{\la
p}^2+\beta\abs{p_x}^2\right)dx}_{:=\mathcal{I}_1}+\underbrace{2\gamma\beta\Re\left(\int_{l_1}^{l_2}\tilde{g}p_{x}\overline{v_x}dx\right)}_{:=\mathcal{I}_2}+o(1).	
\end{array}
\end{equation}
Now, adopting \eqref{Lemma1-EQ1}, \eqref{Lemma2-EQ1}, \eqref{Lemma3-EQ1}, \eqref{Lemma4-EQ1}, the definitions of $\theta_3$ and $\theta_4$ and Cauchy-Schwarz inequality result in
\begin{equation}\label{Lemma5-EQ3}
\abs{\mathcal{I}_1}=o(1)\quad \text{and}\quad \abs{\mathcal{I}_2}=o(1).	
\end{equation}
On the other hand, since $\alpha=\alpha_1+\gamma^2\beta$, it is easy to see that
\begin{equation}\label{Lemma5-EQ4}
\begin{array}{l}
\displaystyle
\int_{l_1}^{l_4}(\theta_3+\theta_4)\left(\alpha\abs{v_x}^2+\beta\abs{p_x}^2\right)dx-2\gamma\beta\Re\left(\int_{l_1}^{l_2}(\theta_3+
\theta_4)p_{x}\overline{v_x}dx\right)=\\[0.1in]
\displaystyle
\alpha_1\int_{l_1}^{l_4}(\theta_3+\theta_4)\abs{v_x}^2dx+\beta\int_{l_1}^{l_4}(\theta_3+\theta_4)\abs{\gamma v_x-p_x}^2dx.
\end{array}
\end{equation}
Therefore, substitution of \eqref{Lemma5-EQ3} and \eqref{Lemma5-EQ4} into \eqref{Lemma5-EQ2} result in
\begin{equation}\label{Lemma5-EQ5}
\int_{l_1}^{l_2}(\theta_3+\theta_4)\left(\rho\abs{\la v}^2+\alpha_1\abs{v_x}^2+\mu \abs{\la p}^2+\beta\abs{\gamma v_x-p_x}^2\right)dx=o(1).	
\end{equation}
Finally, by  \eqref{Lemma1-EQ1}, \eqref{Lemma2-EQ1}, \eqref{Lemma3-EQ1}, \eqref{Lemma4-EQ1}, the definitions of $\theta_3$ and $\theta_4$ in \eqref{Lemma5-EQ5},  the desired result \eqref{Lemma5-EQ1} is obtained. 	
\end{proof}
\begin{lemma}\label{Lemma6}
The solution $(u,u^1,v,z,p,q,y,y^1)$ of the system \eqref{P1-EXP}-\eqref{P2-EXP} satisfies the
following estimates
\begin{equation}\label{Lemma6-EQ1}
\abs{v(l_1)}^2=o(\la^{-2}),\quad \abs{v(l_2)}^2=o(\la^{-2}),\quad \abs{v_x(l_1)}^2=o(1)\quad \text{and}\quad \abs{v_x(l_2)}^2=o(1). 	
\end{equation}	
\end{lemma}
\begin{proof}
Define  $g\in C^1([l_1,l_2])$ such that
\begin{equation}\label{newg}
g(l_2)=-1,\quad g(l_1)=1, \max_{x\in [l_1,l_2]}\abs{g(x)}=m_g\quad \max_{x\in [l_1,l_2]}\abs{g'(x)}=m_g'.	
\end{equation}
By using $g$ in \eqref{TL-EQ1},
\begin{equation}\label{lemma6-EQ1}
\begin{array}{l}
\displaystyle
\mathcal{J}_1(l_1)-\mathcal{J}_1(l_2)=-\int_{l_1}^{l_2}g'\left(\rho\abs{\la v}^2+\alpha\abs{v_x}^2+\mu\abs{\la p}^2+\beta\abs{p_x}^2\right)dx+2\gamma\beta\Re\left(\int_{l_1}^{l_2}g'p_{x}\overline{v_x}dx\right)+\\[0.1in]
\displaystyle
\hspace{3cm}\mathcal{J}_2(l_2)-\mathcal{J}_2(l_1)+o(1),
\end{array}
\end{equation}
which, together with \eqref{newg} and \eqref{Lemma5-EQ1}, implies that
\begin{equation*}
\int_{l_1}^{l_2}g'\left(\rho\abs{\la v}^2+\alpha\abs{v_x}^2+\mu\abs{\la p}^2+\beta\abs{p_x}^2\right)dx=o(1)\quad \text{and}\quad \left|\Re\left(\int_{l_1}^{l_2}g'p_{x}\overline{v_x}dx\right)\right|=o(1). 	
\end{equation*}
Substituting these estimates into \eqref{lemma6-EQ1} and by \eqref{TL-J} leads to
\begin{equation*}
\mathcal{M}(l_1)+\mathcal{M}(l_2)= 2\gamma\beta \Re\left(p_x(l_1)\overline{v_x}(l_1)\right)-2\gamma\beta \Re\left(p_x(l_2)\overline{v_x}(l_2)\right)+\mathcal{J}_2(l_2)-\mathcal{J}_2(l_1)+o(1)
\end{equation*}
where
\begin{equation}\label{Mzeta}
\mathcal{M}(\zeta)=\rho\abs{\la
v(\zeta)}^2+\alpha\abs{v_x(\zeta)}^2+\mu\abs{\la p(\zeta)}^2+\beta\abs{p_x(\zeta)}^2,	
\end{equation}
and therefore,
\begin{equation}\label{Lemma6-Eq3}
\mathcal{M}(l_1)+\mathcal{M}(l_2)\leq 2\gamma\beta \abs{p_x(l_1)}\abs{v_x(l_1)}+2\gamma\beta\abs{p_x(l_2)}\abs{v_x(l_2)}+\abs{\mathcal{J}_2(l_2)}+\abs{\mathcal{J}_2(l_1)}+o(1).	
\end{equation}
Now, use the fact that $f^3\in H^1(l_1,l_2)\subset C([l_1,l_2]),$ $f^5\in H^1_{\ast}(l_1,l_2)\subset C([l_1,l_2]),$ and $\|\mathcal{F}\|_{\mathcal{H}}=o(1)$ (cf. \eqref{EPE-EXP-EQ3}), to obtain
\begin{equation}\label{Lemma6-Eq4}
\abs{f^3(\zeta)}=o(1)\quad \text{and}\quad \abs{f^5(\zeta)}=o(1),\quad \text{where}\quad \zeta\in \{l_1,l_2\}. 	
\end{equation}
Finally,  by \eqref{Lemma6-Eq4}, the definition of $\mathcal{J}_2$ given in \eqref{TL-J}, and Young's inequality,
\begin{equation}\label{Lemma6-Eq5}
\left\{\begin{array}{l}
2\gamma\beta \abs{p_x(\zeta)}\abs{v_x(\zeta)}\leq \gamma^2\beta\abs{v_x(\zeta)}^2+
\beta\abs{p_x(\zeta)}^2,\\[0.1in]
\abs{\mathcal{J}_2(\zeta)}\leq 	\frac{1}{2}\rho \abs{\la v(\zeta)}^2+2\rho\abs{f^3(\zeta)}
^2+\frac{1}{2}\mu\abs{\la p(\zeta)}^2+2\mu\abs{f^5(\zeta)}^2\leq \frac{1}{2}\rho \abs{\la
v(\zeta)}^2+\frac{1}{2}\mu\abs{\la p(\zeta)}^2+o(1),
\end{array}
\right.	
\end{equation}
where $\zeta\in \{l_1,l_2\}$. To get the desired result \eqref{Lemma6-EQ1}, substitute  \eqref{Lemma6-Eq5} into \eqref{Lemma6-Eq3} and use
$\alpha=\alpha_1+\gamma^2\beta$,
\begin{equation}\label{Lemma6-Eq6}
\sum_{j=1}^{2}\left(\frac{\rho}{2}\abs{\la v(l_j)}^2+\frac{\mu}{2}\abs{\la p(l_j)}^2+\alpha_1\abs{v_x(l_j)}^2\right)\leq o(1). 	
\end{equation}
\end{proof}
\begin{lemma}\label{Lemma7}
The solution $(u,u^1,v,z,p,q,y,y^1)$ of the system \eqref{P1-EXP}-\eqref{P2-EXP} satisfies the
following estimates
\begin{equation}\label{Lemma7-Eq1}
\int_0^{l_1}\left(\abs{\la u}^2+c_1\abs{u_x}^2\right)=o(1) \quad \text{and}\quad \int_{l_2}^L\left(\abs{\la y}^2+c_2\abs{y_x}^2\right)dx=o(1). 	
\end{equation}
\end{lemma}
\begin{proof}
First, multiply $\eqref{P3-EXP}_1$ and 	$\eqref{P3-EXP}_4$ by $-2x\overline{u_x}$ and $-2(x-L)\overline{y_x}$,  and integrate over $(0,l_1)$ and $(l_2,L)$, respectively. Since  $\|\mathcal{F}\|_{\mathcal{H}}=o(1),$ and $u_x$ and $y_x$ are uniformly bounded  in $L^2(0,l_1)$ and $L^2(l_2,L)$,  respectively,
\begin{equation*}
-\int_0^{l_1}x\left(\abs{\la u}^2\right)_xdx-c_1\int_0^{l_1}x\left(\abs{u_x}
^2\right)_x=\underbrace{2\Re\left(\int_0^{l_1}xf^2\overline{u_x}dx\right)}_{=o(1)}+2\Re\left(i\la \int_0^{l_1}
xf^1\overline{u_x}dx\right)
\end{equation*}
and
\begin{equation*}
-\int_{l_2}^{L}(x-L)\left(\abs{\la y}^2\right)_xdx-c_2\int_{l_2}^L(x-L)\left(\abs{y_x}
^2\right)_x=\underbrace{2\Re\left(\int_{l_2}^{L}(x-L)f^8\overline{y_x}dx\right)}_{=o(1)}+2\Re\left(i\la
\int_{l_2}^L(x-L)f^7\overline{y_x}dx\right).	
\end{equation*}
Since  $\|\mathcal{F}\|_{\mathcal{H}}=o(1),$ and $(\la u)$ and $(\la y)$ are uniformly bounded in $L^2(0,l_1)$ and $L^2(l_2,L)$, respectively,
\begin{equation}\label{Lemma7-Eq2}
\begin{array}{l}
\displaystyle
\int_0^{l_1}\left(\abs{\la u}^2+c_1\abs{u_x}^2\right)dx-l_1\left(\abs{\la u(l_1)}^2+c_1\abs{u_x(l_1)}^2\right)=\\[0.1in]
\underbrace{-2\Re\left(i\la \int_0^{l_1}
(xf^1)_x\overline{u}dx\right)}_{=o(1)}+2\Re\left(i\la l_1f^1(l_1)\overline{u}(l_1)\right) + o(1)
\end{array}
\end{equation}
and
\begin{equation}\label{Lemma7-Eq3}
\begin{array}{l}
\displaystyle
\int_{l_2}^{L}\left(\abs{\la y}^2+c_2\abs{y_x}^2\right)dx+\left(l_2-L\right)\left(\abs{\la y(l_2)}^2+c_2\abs{y_x(l_2)}^2\right)=\\[0.1in]
\displaystyle
-\underbrace{2\Re\left(i\la
\int_{l_2}^L((x-L)f^7)_x\overline{y}dx\right)}_{=o(1)}-2\Re\left(i\la (l_2-L)f^7(l_2)\overline{y}(l_2)\right) + o(1).	
\end{array}
\end{equation}
Recalling Remark \ref{NewCT}, $u(l_1)=v(l_1)$, $v(l_2)=y(l_2)$, \eqref{Lemma6-EQ1} and the facts that
$$
\abs{f^1(l_1)}\leq \int_0^{l_1}\abs{f^1_x}dx=o(1)\quad \text{and}\quad \abs{f^7(l_2)}\leq \int_{l_2}^L\abs{f^7_x}dx=o(1),
$$
the following is obtained
$$
\left\{\begin{array}{l}
\displaystyle
\abs{\la u(l_1)}^2+c_1\abs{u_x(l_1)}^2=o(1),\quad \abs{\la y(l_2)}^2+c_2\abs{y_x(l_2)}^2=o(1),\\[0.1in]
\displaystyle
\left|\Re\left(i\la l_1f^1(l_1)\overline{u}(l_1)\right)\right|=o(1),\quad \left|\Re\left(i\la (l_2-L)f^7(l_2)\overline{y}(l_2)\right)\right|=o(1).
\end{array}
\right.
$$
Finally, substitution of the estimates above into \eqref{Lemma7-Eq2} and \eqref{Lemma7-Eq3},  the desired result \eqref{Lemma7-Eq1} is obtained.
\end{proof}
Now, we are ready to finally prove  Theorem \ref{EPE-EXP}.
$\newline$
\noindent \textbf{Proof of Theorem \ref{EPE-EXP}}. By \eqref{Lemma5-EQ1} and \eqref{Lemma7-Eq1},  we obtain that $\|U\|_{\mathcal{H}}=o(1)$. This contradicts that $\|U\|_{\mathcal{H}}=1$. Hence, \eqref{M2} holds true, and this makes the proof complete.
{\color{black}
\begin{rem} \label{rem-electro}
It is important to note that electrostatic/quasi-static approaches in the modeling of piezoelectric beams exclude dynamic electromagnetic effects, expressed as $\mu\equiv 0$ in \eqref{EPE}. Consequently, the reduced model aligns with the one derived in \cite{Fatori2}, where exponential stability is achieved by introducing two fully-distributed viscous damping terms in the outer wave equations. Notably, our analysis in this section reveals that under the assumption \eqref{LD-P}, an exponential stability result can also be obtained with solely local damping in the middle layer. The proof remains identical (or even simpler) to the one presented earlier and is left to the reader. This represents a significant advancement over the previous result.
\end{rem}
}

\section{Stability results for the system \eqref{PE}}\label{PE-LD-E}

\sn{Note that the assumption  \eqref{LD-E}  applies to all results in this section. For simplicity, the repetition of the assumption in the results below is avoided unless it is necessary to state.}

\subsection{Well-Posedness}\label{WP-PE} In this section,  the well-posedness of the system \eqref{PE} is established  by a semigroup approach. The natural energy of the system \eqref{PE} is  defined by
$$
E_{PE}(t)=\frac{1}{2}\int_{0}^{l_1}\left(\rho\abs{v_t}^2+\alpha_1\abs{v_x}^2+\mu\abs{p_t}^2+\beta\abs{\gamma v_x-p_x}^2\right)dx+\frac{1}{2}\int_{l_1}^{L}\left(\abs{y_t}^2+c_2\abs{y_x}^2\right)dx.
$$
It is straightforward to show that the energy $E_{PE}(t)$ is dissipative along the smooth enough solutions of \eqref{PE}, i.e.
\begin{equation}\label{dEnergy-PE}
\frac{dE_{PE}(t)}{dt}=-\int_{l_1}^Ld_1\abs{y_t}^2dx. 	
\end{equation}
Define the energy space $\HH_{PE}$
\begin{eqnarray}\label{defHPE}
\HH_{PE}:=\left\{(v,z,p,q,y,y^1)\in \left(H_L^1(0,l_1)\times L^2(0,l_1)\right)^2\times H_R^1(l_1,L)\times L^2(l_1,L), v(l_1)=y(l_1)\right\}
\end{eqnarray}
 equipped by the  norm
\begin{equation}\label{norm-PE}
\begin{array}{ll}
\displaystyle
\|U\|_{\HH_{PE}}^2&=\int_{0}^{l_1}\left(\alpha_1\abs{v_x}^2+\rho\abs{z}^2+\beta\abs{\gamma v_x-p_x}^2+\mu\abs{q}^2\right)dx \\
\displaystyle
 &\qquad\qquad+\int_{l_1}^{L}\left(c_2\abs{y_x}^2+\abs{y^1}^2\right)dx,\quad \ma{\forall U=(v,z,p,q,y,y^1)\in \HH_{PE}}.
\end{array}
\end{equation}
This norm is equivalent to the standard norm of $\HH_{PE}$ (the   arguments in the proof of Lemma \ref{equiv-norm}  can be followed mutatis mutandis). Define the unbounded linear operator $\mathcal{A}_{PE}: D(\mathcal{A}_{PE})\subset \HH_{PE}\rightarrow \HH_{PE}$ by
\begin{eqnarray}\label{defAPE}
\mathcal{A}_{PE}\begin{pmatrix}
v\\ z\\ p\\ q\\ y\\ y^1	
\end{pmatrix}=\begin{pmatrix}
z\\
\frac{1}{\rho}\left(\alpha v_{xx}-\gamma\beta p_{xx}\right)\\
q\\
\frac{1}{\mu}\left(\beta p_{xx}-\gamma\beta v_{xx}\right)\\
y^1\\
c_2y_{xx}-d_1 y^1	
\end{pmatrix},\quad \forall U=(v,z,p,q,y,y^1)\in D(\AA_{PE})
\end{eqnarray}
with the domain
\begin{eqnarray}\label{DAPE}
D(A_{PE})=\left\{\begin{array}{l}
U:=(v,z,p,q,y,y^1)\in \mathcal{H}_{PE};\ z,q\in H^1_L(0,l_1),\ y^1\in H_R^1(0,l_1),\\
\ v, p\in H^2(0,l_1)\cap H_{L}^{1}(0,l_1), ~ y\in H^2(l_1,L)\cap H_R^1(l_1,L),\\
  \alpha v_x(l_1)-\gamma\beta p_x(l_1)=c_2y_x(l_1),\ \beta p_x(l_1)=\gamma \beta v_x(l_1),\quad y^1(l_1)=z(l_1)
\end{array}
\right\}.
\end{eqnarray}
\begin{rem}\label{NewCT-PE}
Obviously as in the previous section, the transmission conditions
\[
 \alpha v_x(l_1)-\gamma\beta p_x(l_1)=c_2y_x(l_1),\hbox{ and } \beta p_x(l_1)=\gamma \beta v_x(l_1),
 \]
 are equivalent to the transmission conditions
\[
\alpha_1 v_x(l_1)=c_2y_x(l_1), \hbox{ and } \alpha_1 p_x(l_1)=c_2\gamma y_x(l_1).
\]	

\end{rem}
\noindent If $(v,p,y)^{\top}$ is a sufficiently regular solution of the system \eqref{PE}, the system can be transformed into the first order evolution equation on the Hilbert space $\mathcal{H}_{PE}$
\begin{equation}\label{evolution-PE}
U_t=\mathcal{A}_{PE}U,\quad U(0)=U_0,	
\end{equation}
with \ma{$U=(v,v_t,p,p_t,y,y_t)$ and $U_0=(v_0,v_1,p_0,p_1,y_0,y_1)$}. \ma{By the analogous arguments in Subsection \ref{WP-EPE}},
the solution of the Cauchy problem \eqref{evolution-PE} admits the following representation
$$
U(t)=e^{t\mathcal{A}_{PE}}U_0,\quad t\geq 0,
$$
which leads to the well-posedness for \eqref{evolution-PE}.
\begin{theoreme}
Let $U_{0}\in \mathcal{H}_{PE}$, the system \eqref{evolution-PE} admits a unique weak solution $U$ satisfying
$$
U\in C^0(\mathbb{R}^+,\mathcal{H}_{PE}).
$$
Moreover, if $U_0\in D(\mathcal{A}_{PE})$, the system \eqref{evolution-PE} admits a unique strong solution $U$ satisfying
$$
U\in C^1(\mathbb{R}^{+},\mathcal{H}_{PE})\cap C^0(\mathbb{R}^+,D(\mathcal{A}_{PE})).
$$
\end{theoreme}
\subsection{Strong Stability}\label{SS-PE} In this section,   the strong stability of the system \eqref{PE} is investigated. Here is the main result.
\begin{theoreme}\label{SS-PE}
The $C_0-$semigroup of contraction $\left(e^{t\mathcal{A}_{PE}}\right)$ is strongly stable in $\mathcal{H}_{PE}$, i.e., for all $U_0\in \mathcal{H}_{PE}$, the solution of \eqref{evolution-PE} satisfies $\displaystyle{\lim_{t\to \infty}\|e^{t\mathcal{A}_{PE}}U_0\|_{\mathcal{H}_{PE}}=0},$ if and only if
\begin{equation}\label{SS-SC-PE}\tag{${\rm SC}$}
\frac{\sigma_{+}}{\sigma_{-}}\neq \frac{2n_+-1}{2n_--1},\quad \forall n_+,n_-\in \mathbb{N}
\end{equation}
where \sn{the two positive real numbers $\sigma_{+}$ and $\sigma_{-}$ are defined by}
\begin{equation}\label{sigma+-}
\sigma_{+}:=\sqrt{\frac{(\rho\beta+\mu\alpha)+ \sqrt{(\rho\beta-\mu\alpha)^2+4\gamma^2\beta^2\mu\rho}}{2\beta\alpha_1}}\quad \text{and}\quad \sigma_{-}:=\sqrt{\frac{(\rho\beta+\mu\alpha)-\sqrt{(\rho\beta-\mu\alpha)^2+4\gamma^2\beta^2\mu\rho}}{2\beta\alpha_1}}. 	
\end{equation}	
\end{theoreme}
\begin{proof}
It follows from the  Arendt-Batty theorem (see page 837 in \cite{Arendt01}), since the resolvent of $\mathcal{A}_{PE}$ is compact in $\mathcal{H}_{PE}$,   the system \eqref{PE} is strongly
stable if and only if $\mathcal{A}_{PE}$ does not have pure imaginary eigenvalues, i.e.
 $\sigma(\mathcal{A}_{PE})\cap i\mathbb{R}=\emptyset$. By Section \ref{WP-PE},
$0\in \rho(\mathcal{A}_{PE})$ is immediate. However,  $\sigma(\mathcal{A}
_{PE})\cap i\mathbb{R}^{\ast}=\emptyset$ must be proved. For this purpose, for a real number $\la\neq 0$ and $U=(v,z,p,q,y,y^1)\in D(\mathcal{A}_{PE}),$ consider
\begin{equation}\label{SS-EQ1-PE}
\mathcal{A}_{PE}U=i\la U,
\end{equation}
which is equivalent to the following system
\begin{equation}\label{SS-EQ2-PE}
z=i\la v\ \text{in}\ (0,l_1),\quad q=i\la p\ \text{in}\ (0,l_1),\quad \text{and}\quad 	y^1=i\la y\ \text{in}\ (l_1,L),
\end{equation}
and
\begin{equation}\label{SS-EQ3-PE}
\left\{\begin{array}{ll}
\rho \la^2v+\alpha v_{xx}-\gamma\beta p_{xx}=0,&\\
\mu \la^2p+\beta p_{xx}-\gamma \beta v_{xx}=0,& x\in (0,l_1),\\
\la^2 y+c_2 y_{xx}-d_1y^1=0,& x\in (l_1,L), 	
\end{array}
\right.	
\end{equation}
From the identity
\[
\Re\left(\mathcal{A}_{PE}
U,U\right)_{\mathcal{H}_{PE}}=-\int_{l_1}^{L}d_1\abs{y^1}^2dx,
\]
 and  \eqref{SS-EQ1-PE},
\begin{equation}\label{SS-EQ4-PE}
0=\Re \left(i\la U,U\right)_{\mathcal{H}_{PE}}=\Re\left(\mathcal{A}_{PE}
U,U\right)_{\mathcal{H}_{PE}}=-\int_{l_1}^{L}d_1\abs{y^1}^2dx.
\end{equation}	
Thus,
\begin{equation}\label{SS-EQ5-PE}
d_1y^1=0\ \text{in}\ (l_1,L),\ \text{and consequently},~~y^1=y=0~~ \text{in}\quad (a_1,b_1)
\end{equation}
by  \eqref{SS-EQ2-PE}, \eqref{LD-E} and \eqref{SS-EQ4-PE}.  Considering $\eqref{SS-EQ3-PE}_{3}$, \eqref{SS-EQ5-PE} and the unique continuation theorem, $y=0$ in $(l_1,L)$. Moreover,  since $y\in H^2(l_1,L)\subset C^1([l_1,L])$, $y(l_1)=y_x(l_1)=0$. It follows from the continuity condition and Remark \ref{NewCT-PE} that $v(l_1)=v_x(l_1)=p_x(l_1)=0$.
Using the fact that $\alpha=\alpha_1+\gamma^2\beta$,  $\eqref{SS-EQ3-PE}_1$ and $\eqref{SS-EQ3-PE}_2$, the system \eqref{SS-EQ2-PE}-\eqref{SS-EQ3-PE} reduced to
\begin{equation}\label{SS-EQ6-PE}
\left\{\begin{array}{l}
v_{xx}=-\la^2\alpha_1^{-1}\left(\rho v+\gamma\mu p\right),\\
p_{xx}=-\la^2\alpha_1^{-1}\left(\gamma\rho v+\mu\alpha\beta^{-1}p\right),\quad x\in (0,l_1),\\
v(0)=p(0)=v(l_1)=v_x(l_1)=p_x(l_1)=0. 	
\end{array}
\right.
\end{equation}
\ma{By differentiating   \eqref{SS-EQ6-PE}$_1$ twice,  using \eqref{SS-EQ6-PE}$_2$ and \eqref{SS-EQ6-PE}$_3$},  the following system is obtained
\begin{equation}\label{SS-EQ7-PE}
\left\{\begin{array}{l}
\alpha_1\beta v_{xxxx}+\la^2(\rho\beta+\mu\alpha)v_{xx}+\mu\rho\la^4 v=0,\\
v(0)=v_{xx}(0)=v(l_1)=v_x(l_1)=v_{xxx}(l_1)=0.	
\end{array}
\right.
\end{equation}
The characteristic polynomial corresponding to  \eqref{SS-EQ7-PE} is
\begin{equation}\label{Carac-Eq}
q(\varkappa)=\alpha_1\beta \varkappa^4+\la^2(\rho\beta+\mu\alpha)\varkappa^2+\mu\rho\la^4.
\end{equation}
and therefore, define
$$
q_0(m):=\alpha_1\beta m^2+\la^2(\rho\beta+\mu\alpha)m+\mu\rho\la^4.
$$
Since $(\rho\beta+\mu\alpha)^2-4\beta\alpha_1\mu\rho=(\rho\beta-\mu\alpha)^2+4\gamma^2\beta^2\mu\rho>0$,  the polynomial $q_0$ has two distinct real roots $m_-$ and $m_+:$
$$
m_+=-\sigma_+^2\lambda^2\quad \text{and}\quad m_-=-\sigma_-^2\lambda^2
$$
where $\sigma_+$ and $\sigma_-$ are defined by \eqref{sigma+-}. Observe that $m_+<0,$ and by $\ma{\alpha>\gamma^2\beta}$, $m_-<0$ is immediate. Setting $\varkappa_+:=\sqrt{-m_+}$ and $\varkappa_-:=\sqrt{-m_-}$,  $q$ has in total of four roots $i\varkappa_+,-i\varkappa_+,i\varkappa_-,-i\varkappa_-$. Hence, the general solution of \eqref{SS-EQ7-PE} is
$$
v(x)=c_1\sin(\varkappa_+x)+c_2\cos(\varkappa_+x)+c_3\sin(\varkappa_-x)+c_4\cos(\varkappa_-x)
$$
where $c_j\in \mathbb{C}$, $j=1,\cdots,4$. By the boundary conditions in \eqref{SS-EQ7-PE} at $x=0$ and  $\varkappa_+^2-\varkappa_-^2\neq 0$, it is deduced that $c_2=c_4=0$. Moreover, by boundary conditions in \eqref{SS-EQ7-PE} at $x=l_1$,
\begin{equation}\label{SS-EQ8-PE}
\left\{\begin{array}{l}
c_1\sin(\varkappa_+l_1)+c_3\sin(\varkappa_-l_1)=0,\\
c_1\varkappa_+\cos(\varkappa_+l_1)+c_3\varkappa_-\cos(\varkappa_-l_1)=0,\\
c_1\varkappa_+^3\cos(\varkappa_+l_1)+c_3\varkappa_-^3\cos(\varkappa_-l_1)=0.	
\end{array}
\right.	
\end{equation}
Now, by $\eqref{SS-EQ8-PE}_2$ and $\eqref{SS-EQ8-PE}_3$,
\begin{equation*}
M\left(c_1,c_3\right)^{\top}=(0,0)^{\top},\quad \text{where}\quad  M=\begin{pmatrix}
\varkappa_+\cos(\varkappa_+l_1)& \varkappa_-\cos(\varkappa_-l_1)\\ 	\varkappa_+^3\cos(\varkappa_+l_1)& \varkappa_-^3\cos(\varkappa_-l_1).
\end{pmatrix}	
\end{equation*}
It is easy to see that $\det(M)=\varkappa_-\varkappa_+(\varkappa_-^2-\varkappa_+^2)\cos(\varkappa_-l_1)\cos(\varkappa_+l_1)$. Utilizing $\varkappa_+^2-\varkappa_-^2\neq 0$, it is observed that $\det(M)$ vanishes if and only if  $\cos(\varkappa_+l_1)=0$ or $\cos(\varkappa_-l_1)=0$. We split this into three cases:\\
$\textbf{Case 1}$: Consider $\cos(\varkappa_-l_1)=0$ and $\cos(\varkappa_+l_1)\neq0$.  It follows from $\eqref{SS-EQ8-PE}_1$ and $\eqref{SS-EQ8-PE}_2$ that $c_1=c_3=0$. Consequently, $U=0$.\\
$\textbf{Case 2}$: Consider $\cos(\varkappa_-l_1)\neq 0$ and $\cos(\varkappa_+l_1)=0$. It follows from $\eqref{SS-EQ8-PE}_1$ and $\eqref{SS-EQ8-PE}_2$ that  $c_1=c_3=0$. Consequently, $U=0$.\\
$\textbf{Case 3}$ Consider   $\cos(\varkappa_+l_1)=0$ and $\cos(\varkappa_-l_1)=0$.  Then, \ma{there exists $n_+,n_-\in \mathbb{N}$ such that }
\begin{equation}\label{SS-EQ9-PE}
\varkappa_+=\frac{2n_++1}{2l_1}\pi\quad \text{and}\quad \varkappa_-=\frac{2n_-+1}{2l_1}\pi .
\end{equation}
By $c_2=c_4=0,$ and  $\eqref{SS-EQ8-PE}_1$, \eqref{SS-EQ9-PE},  the general solution of \eqref{SS-EQ7-PE} is given by
\begin{equation}\label{SS-EQ10-PE}
v(x)=c\left(\sin\left(\frac{2n_++1}{2l_1}\pi x\right)\pm \sin\left(\frac{2n_-+1}{2l_1}\pi x\right)\right)	
\end{equation}
where $c\in \mathbb{C}$. On the other hand, by \eqref{SS-EQ9-PE} again,
\begin{equation}\label{condition-PE}
\frac{\varkappa_+}{\varkappa_-}=\frac{\sigma_+}{\sigma_-}=\frac{2n_++1}{2n_-+1},\quad n_+,n_-\in \mathbb{N},\quad \text{and}\quad \lambda=\frac{2n_++1}{2l_1\sigma_+}
\end{equation}
where $\sigma_+$ and $\sigma_-$ are defined by \eqref{sigma+-}. Hence, $\sigma\left(\AA_{PE}\right)\cap i\mathbb{R}=\emptyset$ if and only if  \eqref{SS-SC-PE} holds.
\end{proof}
\subsection{Exponential and Polynomial Stability Results}\label{SecPol} The aim of this subsection is to prove the exponential and the polynomial stabilities of the system \eqref{PE} if \eqref{LD-E} holds \sn{and under an appropriate assumption on the ratio  $\frac{\sigma_+}{\sigma_-}$, which depends on its arithmetic nature}.
Let us consider the following hypotheses
\begin{enumerate}
\item[$\rm{\mathbf{(H_{Exp})}}$] Assume that $\frac{\sigma_+}{\sigma_-}\in \mathbb{Q}$ is such that $\frac{\sigma_+}{\sigma_-}=\frac{\xi_+}{\xi_-}$ where $\gcd(\xi_+,\xi_-)=1,$ and $\xi_+, \,\xi_-$ are even and odd integers, respectively, or the other way around.
\item[$\rm{\mathbf{(H_{Pol})}}$] Assume that $\frac{\sigma_+}{\sigma_-}$ is an irrational number. Then,  suppose that
\sn{ there exists $\varpi\left(\frac{\sigma_+}{\sigma_-}\right)\geq 2,$ depending on $\frac{\sigma_+}{\sigma_-}$, such that
for all sequences  $\Lambda=(\xi_{1,n}, \xi_{2,n})_{n\in \mathbb{N}}\in (\mathbb{N}\times \mathbb{N}^*)^{\mathbb{N}}$ with  $\xi_{1,n}\sim \xi_{2,n}$ for  sufficiently large $n$,
 there exist
 a positive constant  $c\left(\frac{\sigma_+}{\sigma_-}, \Lambda\right)$
 and a positive integer $N\left(\frac{\sigma_+}{\sigma_-}, \Lambda\right),$
  depending on $\frac{\sigma_+}{\sigma_-}$ and  the sequence $\Lambda,$
  such that
\[
\left|\frac{\sigma_+}{\sigma_-}-\frac{\xi_{1,n}}{\xi_{2,n}}\right|> \frac{c\left(\frac{\sigma_+}{\sigma_-}, \Lambda\right)}{\xi_{2,n}^{\varpi\left(\frac{\sigma_+}{\sigma_-}\right)}}, \forall n\geq
N\left(\frac{\sigma_+}{\sigma_-}, \Lambda\right).\]
}
\end{enumerate}

\begin{rem} \label{RkSN}
(i) \sn{Note that it will be shown in section \ref{sillustration}  that
the number $\varpi\left(\frac{\sigma_+}{\sigma_-}\right)$ is \sn{indeed} an irrationality measure of the quotient $\frac{\sigma_+}{\sigma_-}$. } More explanations on this notion will be under way as well as examples and some references. (ii) \sn{Note also that $\rm{\mathbf{(H_{Exp})}}$  or $\rm{\mathbf{(H_{Pol})}}$ implies that \eqref{SS-SC-PE} holds.}
\end{rem}

\noindent \ma{By  (\cite{Huang01}, \cite{pruss01} for Theorem \ref{PE-EXP}), or (\cite{Borichev01}, \cite{RaoLiu01} for Theorem \ref{Pol-PE}), } the $C_0-$semigroup of contractions
$\left(e^{t\mathcal{A}_{PE}}\right)_{t\geq 0}$ on $\mathcal{H}_{PE}$ satisfies
\eqref{EPE-EXP-EQ1} or \eqref{Pol-PE-Eq1} if the following two conditions hold
\begin{equation}\label{N1-PE}\tag{$\rm{N1}$}
i\mathbb{R}\subset \rho\left(\mathcal{A}_{PE}\right),	
\end{equation}
\begin{equation}\label{N2-PE}\tag{$\rm{N2}$}
\sup_{\la \in \mathbb{R}}\frac{1}{\la^{\ell}}\|\left(i\la I-\mathcal{A}_{PE}\right)^{-1}\|_{\mathcal{L}(\mathcal{H})}<\infty \quad \text{with} \left\{\begin{array}{ll}
\ell=0,& \text{for Theorem \ref{PE-EXP}},\\
\ell=4\varpi\left(\frac{\sigma_+}{\sigma_-}\right)-4,& \text{for Theorem \ref{Pol-PE}},
\end{array}
\right. 	
\end{equation}
Since \ma{it is already proved that} $i\mathbb{R}\subset \rho(\mathcal{A}_{PE})$ (see Section \ref{SS-PE}), \ma{it remains to} prove the condition \eqref{N2-PE}, \ma{for which a} contradiction argument is applicable. Suppose that \eqref{N2-PE} is false. Then, there exists a sequence $\{(\la^n,U^n)\}_{n\geq 1}\subset \mathbb{R}^{\ast}\times D(\mathcal{A}_{PE})$ with
\begin{equation}\label{PE-EXP-EQ2}
\abs{\la^n}\to \infty\quad \text{and}\quad \|U^n\|_{\mathcal{H}}=\|\left(v^n,z^n,p^n,q^n,y^n,y^{1,n}\right)^{\top}\|_{\mathcal{H}_{PE}}=1	
\end{equation}
such that
\begin{equation}\label{PE-EXP-EQ3}
\la_n^{\ell}\left(i\la^nI-\mathcal{A}_{PE}\right)U^n= \mathcal{G}^n:=\left(g^{1,n},g^{2,n},g^{3,n},g^{4,n},g^{5,n},g^{6,n}\right)^{\top}\to 0\quad \text{in}\quad \mathcal{H}_{PE}.
\end{equation}
For simplicity,   index $n$ is dropped for the rest of the proof. Now, \eqref{PE-EXP-EQ3} is equivalent to
\begin{equation}\label{PE-EXP1}
\left\{\begin{array}{lll}
i\la v-z=\la^{-\ell}g^1\to 0&\text{in}&H^1_L(0,l_1),\\
i\la p-q=\la^{-\ell}g^3\to 0&\text{in}&H^1_L(0,l_1),\\
i\la y-y^1=\la^{-\ell}g^5\to 0&\text{in}&H^1_{R}(l_1,L)\\
\end{array}
\right.
\end{equation}
and
\begin{equation}\label{PE-EXP2}
\left\{\begin{array}{lll}
i\la \rho z-\alpha v_{xx}+\gamma \beta p_{xx}=\rho \la^{-\ell}g^2\to 0&\text{in}&L^2(0,l_1),\\
i\la \mu q-\beta p_{xx}+\gamma \beta v_{xx}=\mu \la^{-\ell}g^4\to 0&\text{in}&L^2(0,l_1),\\
i\la y^1-c_2y_{xx}+d_1y^1=\la^{-\ell}g^6\to 0&\text{in}&L^2(l_1,L).
\end{array}
\right.	
\end{equation}
Combining $\eqref{PE-EXP1}_3$ and $\eqref{PE-EXP2}_3$ leads to
\begin{equation}\label{PE-Combining1}
\la^2y+c_2y_{xx}-i\la d_1 y=\la^{-\ell}(-i\la g^5-d_1g^5-g^6). 	
\end{equation}
\begin{lemma}\label{PE-Lemma1}
The solution $(v,z,p,q,y,y^1)$ of the system \eqref{PE-EXP1}-\eqref{PE-EXP2} satisfies the following estimates
\begin{equation}\label{PE-Lemma1-EQ1}
\int_{l_1}^{L}d_1\abs{y^1}^2dx=o(\la^{-\ell}),\quad \int_{l_1}^{L}d_1\abs{\la y}^2dx=o(\la^{-\ell}), \quad \int_{a_1}^{b_1}\abs{\la y}^2dx=o(\la^{-\ell}).
\end{equation}
and
\begin{equation}\label{PE-Lemma1-EQ2}
\int_{D_{\epsilon}}\abs{y_x}^2dx=o(\la^{-\ell})
\end{equation}
where $D_{\epsilon}:=(a_1+\epsilon,b_1-\epsilon)$ with a small enough $\epsilon>0$  such that $\ma{\epsilon}<\frac{b_1-a_1}{2}$.
\end{lemma}
\begin{proof} The proof is split into two steps.\\
\textbf{Step 1}. For obtaining first estimate in \eqref{PE-Lemma1-EQ1},  take the inner product of \eqref{PE-EXP-EQ3} and $U$ in $\mathcal{H}_{PE}$, and use the fact that $\|U\|_{\mathcal{H}_{PE}}=1$ and $\|\mathcal{G}\|_{\mathcal{H}_{PE}}=o(1),$
\begin{equation}\label{PE-Lemma1-EQ3}
\int_{l_1}^{L}d_1\abs{y^1}^2dx=-\Re\left(\left<\mathcal{A}_{PE}U,U\right>_{\mathcal{H}_{PE}}\right)=\Re\left(\left<(i\la I-\mathcal{A}_{PE})U,U\right>_{\mathcal{H}_{PE}}\right)=\frac{1}{\la^{\ell}}\Re\left(\left<\mathcal{G},U\right>_{\mathcal{H}_{PE}}\right)=o(\la^{-\ell}). 	
\end{equation}
Next, multiply $\eqref{PE-EXP1}_{3}$ by $\sqrt{d_1},$ and use the first estimate in \eqref{PE-Lemma1-EQ1} and $\|\mathcal{G}\|_{\mathcal{H}}=o(1)$. This leads to the second estimate in \eqref{PE-Lemma1-EQ1}. Finally, the second estimate in \eqref{PE-Lemma1-EQ1} with  \eqref{LD-E} yields the last estimate in \eqref{PE-Lemma1-EQ1}.\\
\textbf{Step 2.}   For proving \eqref{PE-Lemma1-EQ2}, let $0<\epsilon<\frac{b_1-a_1}{2}$ and  fix the cut-off function $\theta_5\in C^2([l_1,L])$ such that $0\leq \theta_5(x)\leq 1$ for all $x\in [l_1,L],$ and
$$
\theta_5(x)=\left\{\begin{array}{lll}
1,&\text{if}&x\in [a_1+\epsilon,b_1-\epsilon],\\
0,&\text{if}&x\in [l_1,a_1]\cup [b_1,L]. 	
\end{array}
\right.
$$
Now, multiply $\eqref{PE-Combining1}$ by $\theta_5\overline{y}$, \ma{integrate by parts over $(l_1,L)$}, and use the definition of $\theta_5$ to obtain	
\begin{equation}\label{PE-Step2-Lemma1-EQ1}
\int_{l_1}^L\theta_5\abs{\la y}^2dx-c_2\int_{l_1}^L\theta_5\abs{y_x}^2dx-c_2\Re\left(\int_{l_1}^L\theta_5'y_x\overline{y}dx\right)=-\Re \left(\la^{-\ell}\int_{l_1}^L(i\la g^5+d_1g^5+g^6)\theta_5\overline{y}dx\right). 	
\end{equation}
It is know that $\|\mathcal{G}\|_{\mathcal{H}_{PE}}=o(1)$ implies that $(\la y)$ is uniformly bounded in $L^2(l_1,L)$ by $\eqref{PE-EXP1}_{3}$ and $y_x$  is uniformly bounded in $L^2(l_1,L)$. Therefore, by Cauchy-Schwarz inequality, \eqref{PE-Lemma1-EQ1},  and the definition of $\theta_5$, $\|U\|_{\mathcal{H}_{PE}}=1$,
$$
\begin{array}{l}
\displaystyle
\left|\Re\left(\int_{l_1}^L\theta_5'y_x\overline{y}dx\right)\right|=\left|\frac{1}{2}\int_{l_1}^{L}\theta_5'(\abs{y}^2)_xdx\right|=\left|\frac{1}{2}\int_{l_1}^{L}\theta_5{''}\abs{y}^2dx\right|=o(\la^{-\ell - 2}),\\[0.1in]
\displaystyle
\text{and}\quad \left|\la^{-\ell} \int_{l_1}^L(i\la g^5+d_1g^5+g^6)\theta_5\overline{y}dx\right|=o(\la^{-\ell}).
\end{array}
$$
Substituting the above estimates into \eqref{PE-Step2-Lemma1-EQ1} and using \eqref{PE-Lemma1-EQ1} lead to
$$
\int_{l_1}^L\theta_5\abs{y_x}^2dx=o(\la^{-\ell}).
$$
Finally, by the definition of $\theta_5,$ the  desired result  \eqref{PE-Lemma1-EQ2} is obtained.
\end{proof}

\begin{lemma}\label{PE-Lemma2}
The solution $(v,z,p,q,y,y^1)$ of the system \eqref{PE-EXP1}-\eqref{PE-EXP2} satisfies the following estimates
\begin{equation}\label{PE-Lemma2-EQ1}
\int_{l_1}^{L}\left(\abs{\la y}^2+c_2\abs{y_x}^2\right)dx=o(\la^{-\frac{\ell}{2}}),
\end{equation}
\begin{equation}\label{PE-Lemma2-EQ2}
\abs{\la v(l_1)}^2=o(\la^{-\frac{\ell}{2}}),\quad |v_x(l_1)|^2=o(\la^{-\frac{\ell}{2}})\quad \text{and}\quad |p_x(l_1)|^2=o(\la^{-\frac{\ell}{2}}). 	
\end{equation}
\end{lemma}
\begin{proof}
The proof is split into three steps.\\
\textbf{Step 1.} Letting $h_1\in C^1([l_1,L])$, the  following estimate is targeted to prove
\begin{equation}\label{PE-Step1-Lemma2-EQ1}
\begin{array}{l}
\displaystyle
-\int_{l_1}^Lh_1'\left(\abs{\la y}^2+c_2\abs{y_x}^2\right)dx+\ma{h_1(L)\left(\abs{\la y(L)}^2+c_2\abs{y_x(L)}^2\right)}\\
\displaystyle
-h_1(l_1)\left(\abs{\la y(l_1)}^2+c_2\abs{y_x(l_1)}^2\right)= 2 \Re\left(i\la^{1-\ell} h_1(l_1)g^5(l_1)\overline{y}(l_1)\right)+o(\la^{-\frac{\ell}{2}}).
\end{array}	
\end{equation}
First, multiply $\eqref{PE-Combining1}$ by $2 h_1 \overline{y}_x$ and integrate over $(l_1,L)$ to get
\begin{equation*}
\int_{l_1}^Lh_1	\left((\abs{\la y}^2)_x+c_2(\abs{y_x}^2)_x\right)dx-2\Re\left(i\la \int_{l_1}^Ld_1h_1y\overline{y_x}dx\right)=-2\Re\left(\la^{-\ell}\int_{l_1}^L(i\la g^5+d_1g^5+g^6)h_1\overline{y}_xdx\right).
\end{equation*}
As the Integration by parts is implemented,
\begin{equation}\label{PE-Step1-Lemma2-EQ2}
\begin{array}{l}
\displaystyle
-\int_{l_1}^Lh_1'\left(\abs{\la y}^2+c_2\abs{y_x}^2\right)dx+h_1(L)\left(\abs{\la y(L)}^2+c_2\abs{y_x(L)}^2\right)-h_1(l_1)\left(\abs{\la y(l_1)}^2+c_2\abs{y_x(l_1)}^2\right)\\
\displaystyle
-2\Re\left(i\la \int_{l_1}^Ld_1h_1y\overline{y_x}dx\right)=-2 \Re\left(\la^{-\ell}\int_{l_1}^L(d_1g^5+g^6)h\overline{y_x}dx\right)+ 2 \Re\left(i\la^{1-\ell}\int_{l_1}^L(h_1 g^5)_x\overline{y}dx\right)\\
\displaystyle
+ 2 \Re\left(i\la^{1-\ell} h_1(l_1)g^5(l_1)\overline{y}(l_1)\right).
\end{array}
\end{equation}
Since $y_x$ is uniformly bounded in $L^2(l_1,L)$ and $\|\mathcal{G}\|_{\mathcal{H}_{PE}}=o(1),$ by Cauchy-Schwarz inequality and \eqref{PE-Lemma1-EQ1}, the following estimates are immediate
\begin{eqnarray}
\nonumber
&&\left|-2\Re\left(i\la \int_{l_1}^Ld_1h_1y\overline{y_x}dx\right)\right|=o(\la^{-\frac{\ell}{2}}),\quad \left|\Re\left(\la^{-\ell}\int_{l_1}^L(d_1g^5+g^6)h\overline{y_x}dx\right)\right|=o(\la^{-\ell}),\\
&& \left|\Re\left(i\la\int_{l_1}^L(hg^5)_x\overline{y}dx\right)\right|=o(\la^{-\ell}).
\end{eqnarray}
Finally, substituting these in in \eqref{PE-Step1-Lemma2-EQ2} lead to the desired equation \eqref{PE-Step1-Lemma2-EQ1}.\\
\textbf{Step 2.} For this step, \eqref{PE-Lemma2-EQ1} is aimed to be proved. First, define the following cut-off functions $\theta_6, \theta_7\in C^2([l_1,L])$ by
\begin{equation*}
\theta_6(x)=\left\{\begin{array}{lll}
1&\text{if}&x\in [l_1,a_1+\epsilon],\\
0&\text{if}&x\in [a_2-\epsilon,L]	
\end{array}
\right.	\quad \text{and}\quad \theta_7(x)=\left\{\begin{array}{lll}
0&\text{if}&x\in [l_1,a_1+\epsilon],\\
1&\text{if}&x\in [a_2-\epsilon,L].		
\end{array}
\right.
\end{equation*}
so that $0\leq \theta_6,\theta_7\leq 1$ for all $x\in [l_1,L]$.

Following the the same arguments as in \ma{Lemma \ref{Lemma5}, take $h_1(x)=(x-l_1)\theta_6(x)+(x-L)\theta_7(x)$ in \eqref{PE-Step1-Lemma2-EQ1}. This,  combined with  Lemma \ref{PE-Lemma1}}, results in  \eqref{PE-Lemma2-EQ1}.\\
\textbf{Step 3.} Finally, to prove \eqref{PE-Lemma2-EQ2}, take $h_1(x)=x-L$ in \eqref{PE-Step1-Lemma2-EQ1} and use \eqref{PE-Lemma2-EQ1} to obtain
\begin{equation*}\label{PE-Step3-EQ1}
\abs{\la y(l_1)}^2+c_2\abs{y_x(l_1)}^2 = 2 \Re\left(-i\la^{1-\ell} g^5(l_1)\overline{y}(l_1)\right)+o(\la^{-\frac{\ell}{2}}).	
\end{equation*}
 It follows from Young's inequality that
\begin{equation*}
\abs{\la y(l_1)}^2+c_2\abs{y_x(l_1)}^2\leq 2 \abs{\la^{-\ell}g^5(l_1)}\abs{\la y(l_1)}+o(\la^{-\frac{\ell}{2}})\leq 2 \la^{-2\ell}\abs{g^5(l_1)}^2+\frac{1}{2}\abs{\la y(l_1)}^2+o(\la^{-\frac{\ell}{2}}).
\end{equation*}
Now, use  $g^5\in H^1_R(l_1,L)\subset C([l_1,L])$ and $\|\mathcal{G}\|_{\mathcal{H}_{PE}}=o(1)$ to obtain $\abs{g^5(l_1)}=o(1)$. This, together with the estimate above,  provides the following estimate
 $$
 \frac{1}{2}\abs{\la y(l_1)}^2+c_2\abs{y_x(l_1)}^2\leq o(\la^{-\frac{\ell}{2}}).
 $$
Hence, \eqref{PE-Lemma2-EQ2} is concluded from recalling Remark \eqref{NewCT-PE}.
\end{proof}

\noindent For the next result,  substitute $\eqref{PE-EXP1}_1$ and $\eqref{PE-EXP1}_2$ in $\eqref{PE-EXP2}_1$ and $\eqref{PE-EXP2}_2$,  respectively, so that the following system is obtained,
\begin{equation}\label{PE-Combining2}
\left\{\begin{array}{l}
\rho \la^2 v+\alpha v_{xx}-\gamma\beta p_{xx}=-\la^{-\ell}\left(\rho g^2+i\la \rho g^1\right)\\
\mu \la^2 p+\beta p_{xx}-\gamma\beta v_{xx}=-\la^{-\ell}(\mu g^4+i\la \mu g^3). 	
\end{array}
\right. 	
\end{equation}
\begin{lemma}\label{PE-Lemma3}
The solution $(v,z,p,q,y,y^1)$ of the system \eqref{PE-EXP1}-\eqref{PE-EXP2} satisfies the following estimate
\begin{equation}\label{PE-Lemma3-EQ1}
\rho\int_0^{l_1}\abs{\la v}^2dx+\mu\int_0^{l_1}\abs{\la p}^2+\alpha_1\int_0^{l_1}\abs{v_x}^2dx+\beta\int_0^{l_1}\abs{\gamma v_x-p_x}^2dx-2\mu l_1\abs{\la p(l_1)}^2\leq o(\la^{-\frac{\ell}{2}}).
\end{equation}
\end{lemma}
\begin{proof}
First, use the multipliers $-2x\overline{v}_x$ and $-2x\overline{p}_x$  for $\eqref{PE-Combining2}_1$ and $\eqref{PE-Combining2}_2,$ respectively,  and integrate over $(0,l_1)$ to get
\begin{equation*}
-\rho\int_{0}^{l_1}x(\abs{\la v}^2)_xdx-\alpha \int_{0}^{l_1}x(\abs{v_x}^2)_xdx+2\gamma\beta\Re\left(\int_0^{l_1}xp_{xx}\overline{v}_xdx\right)=2\Re\left(\la^{-\ell}\int_0^{l_1}x\left(\rho g^2+i\la\rho g^1\right)\overline{v}_xdx\right),
\end{equation*}
\begin{equation*}
-\mu\int_{0}^{l_1}x(\abs{\la p}^2)_xdx-\beta \int_{0}^{l_1}x(\abs{p_x}^2)_xdx+2\gamma\beta\Re\left(\int_0^{l_1}xv_{xx}\overline{p}_xdx\right)=2\Re\left(\la^{-\ell}\int_0^{l_1}x(\mu g^4+i\la \mu g^3)\overline{p}_xdx\right).
\end{equation*}
Next, integrate by parts  the identities \ma{identities} above, and use \eqref{PE-Lemma2-EQ2} and $\alpha=\alpha_1+\gamma^2\beta$,
\begin{equation}\label{PE-Lemma3-EQ2}
\begin{array}{l}
\displaystyle
\rho\int_0^{l_1}\abs{\la v}^2dx+\alpha_1\int_0^{l_1}\abs{v_x}^2dx+\gamma^2\beta \int_0^{l_1}\abs{v_x}^2dx-2\gamma\beta\Re\left(\int_0^{l_1}p_{x}\overline{v}_x dx\right)\\
\displaystyle
-2\gamma\beta\Re\left(\int_0^{l_1}xp_{x}\overline{v}_{xx}dx\right)=2\rho\Re\left(\la^{-\ell}\int_0^{l_1}xg^2\overline{v}_xdx\right)-2\Re\left(i\la^{1-\ell}\rho \int_0^{l_1}\left(x g^1\right)_x\overline{v} dx\right)+o(\la^{-\frac{\ell}{2}}),
\end{array}
\end{equation}
\begin{equation}\label{PE-Lemma3-EQ3}
\begin{array}{l}
\displaystyle  \mu\int_0^{l_1}\abs{\la p}^2 dx-\mu \abs{\la p(l_1)}^2+\beta \int_{0}^{l_1}\abs{p_x}^2 dx+2\gamma\beta\Re\left(\int_0^{l_1}xv_{xx}\overline{p}_x dx\right)=\\
\displaystyle
2\mu\Re\left(\la^{-\ell}\int_0^{l_1}xg^4\overline{p}_x dx\right)-2\Re\left(i\la^{1-\ell}\mu\int_0^{l_1}\left(x g^3\right)_x\overline{p} dx\right)+2\Re\left(i\la^{1-\ell} l_1\mu  g^3(l_1) \overline{p}(l_1)\right)+o(\la^{-\frac{\ell}{2}}).
\end{array}	
\end{equation}
Since $\|\mathcal{G}\|_{\mathcal{H}_{PE}}=o(1),$ and $v_x, p_x, \la v, \la p$ are uniformly bounded in $L^2(0,l_1),$
\begin{equation*}
\begin{array}{l}
\displaystyle
\left|\Re\left(\la^{-\ell}\int_0^{l_1}xg^2\overline{v}_x dx\right)\right|=o(\la^{-\ell}),\quad  \left|2\Re\left(i\la^{1-\ell}\rho \int_0^{l_1}\left(x g^1\right)_x\overline{v} dx\right)\right|=o(\la^{-\ell}),\\[0.1in]
\displaystyle
\left|\Re\left(\la^{-\ell}\int_0^{l_1}xg^4\overline{p}_x dx\right)\right|=o(\la^{-\ell}),\quad \left|\Re\left(i\la^{{1-\ell}} \mu\int_0^{l_1}\left(x g^3\right)_x\overline{p} dx\right)\right|=o(\la^{-\ell}).
\end{array}
\end{equation*}
Thus, \eqref{PE-Lemma3-EQ2} and \eqref{PE-Lemma3-EQ3}, together with the last two estimates, reduce to
\begin{equation}\label{PE-Lemma3-EQ4}
\begin{array}{l}
\displaystyle
\rho\int_0^{l_1}\abs{\la v}^2dx+\mu\int_0^{l_1}\abs{\la p}^2+\alpha_1\int_0^{l_1}\abs{v_x}^2dx+\beta\int_0^{l_1}\abs{\gamma v_x-p_x}^2 dx\\
\displaystyle
-\mu l_1\abs{\la p(l_1)}^2=2\Re\left(i\la^{1-\ell} l_1\mu  g^3(l_1) \overline{p}(l_1)\right)+o(\la^{-\frac{\ell}{2}}).
\end{array}	
\end{equation}
On the other hand, since $g^3 \in H^1_R(0,l_1)\subset C([0,l_1])$ and $\|\mathcal{G}\|_{\mathcal{H}_{PE}}=o(1)$, by Young's inequality
$$
\left|2\Re\left(i\la^{1-\ell} l_1\mu  g^3(l_1) \overline{p}(l_1)\right)\right|\leq l_1\mu \abs{\la p(l_1)}^2+l_1\mu\la^{-2\ell}\abs{g^3(l_1)}^2\leq l_1\mu \abs{\la p(l_1)}^2+o(\la^{-2\ell}).
$$
Finally, substituting the estimate above in \eqref{PE-Lemma3-EQ4} lead to \eqref{PE-Lemma3-EQ1}.
\end{proof}

$\newline$
\noindent For the next result,  another form of  \eqref{PE-Combining2}  is needed by considering $\alpha=\alpha_1+\gamma^2\beta$:
\begin{equation}\label{PE-Combining3}
\left\{\begin{array}{l}
\la^2\rho v+\alpha_1v_{xx}+\gamma\mu \la^2p=\la^{-\ell}G^1+i\la^{1-\ell} G^2,\\
\la^2\mu\alpha p+\alpha_1\beta p_{xx}+\rho\gamma\beta\la^2v=\la^{-\ell}G^3+i\la^{1-\ell} G^4,
\end{array}
\right.	
\end{equation}
where
\begin{equation}\label{G1234}
\begin{array}{l}
\displaystyle
G^1=-\left(\rho g^2+\gamma\mu g^4\right),\quad G^2=-\left(\rho g^1+\gamma\mu g^3\right),\\
\displaystyle
G^3=-\left(\alpha\mu g^4+\rho\gamma\beta g^2\right),\quad G^4=-\left(\alpha \mu g^3+\rho\gamma\beta g^1\right).
\end{array}
\end{equation}

\begin{lemma}\label{PE-Lemma4}
The solution $(v,z,p,q,y,y^1)$ of the system \eqref{PE-EXP1}-\eqref{PE-EXP2} satisfies the following asymptotic estimate
\begin{equation}\label{PE-Lemma4-EQ1}
\begin{array}{l}
\displaystyle
e_1(l_1) p(l_1)=-e_2(l_1) o(\la^{-(\frac{\ell}{4}+1)})-e_3(l_1) o(\la^{-\frac{\ell}{4}})-e_4(l_1) o(\la^{-\frac{\ell}{4}})\\
\displaystyle
-\la^{-\ell}\int_0^{l_1}\left(e_3(s)G^1(s)+e_4(s)G^3(s)\right)ds-i\la^{1-\ell} \int_0^{l_1}\left(e_3(s)G^2(s)+e_4(s)G^4(s)\right)ds,
\end{array}
\end{equation}
\begin{equation}\label{PE-Lemma4-EQ2}
\begin{array}{l}
\displaystyle
e_5(l_1) p(l_1)= -e_6(l_1) o(\la^{-(\frac{\ell}{4}+1)})-e_7(l_1) o(\la^{-\frac{\ell}{4}})-e_8(l_1) o(\la^{-\frac{\ell}{4}})\\
\displaystyle
-\la^{-\ell}\int_0^{l_1}\left(e_7(s)G^1(s)+e_8(s)G^3(s)\right)ds-i\la^{1-\ell} \int_0^{l_1}\left(e_7(s)G^2(s)+e_8(s)G^4(s)\right)ds,
\end{array}
\end{equation}
with
\begin{equation}\label{ei}
\left\{\begin{array}{l}
\displaystyle
e_1(s)=\frac{\cos(s \varkappa_+)-\cos(s \varkappa_-)}{b_+-b_-},\quad e_2(s)=\frac{b_+\cos(s \varkappa_-)-b_-\cos(s \varkappa_+)}{b_+-b_-},\\
\displaystyle
e_3(s)=-\frac{b_+\varkappa_+\sin(s\varkappa_-)-b_-\varkappa_-\sin(s\varkappa_+)}{\varkappa_- \varkappa_+(b_+-b_-)},\quad e_4(s)=-\frac{\varkappa_-\sin(s\varkappa_+)-\varkappa_+\sin(s\varkappa_-)}{\varkappa_-\varkappa_+(b_+-b_-)},\\
\displaystyle
e_5(s)=\frac{b_+\cos(s\varkappa_+)-b_-\cos(z\varkappa_-)}{b_+-b_-},\quad e_6(s)=-b_+b_-\frac{\cos(s \varkappa_+)-\cos(z \varkappa_-)}{b_+-b_-},\\
\displaystyle
e_7(s)=-\frac{b_+b_-\left(\varkappa_+\sin(s\varkappa_-)-\varkappa_-\sin(s\varkappa_+)\right)}{\varkappa_-\varkappa_+(b_+-b_-)},\quad e_8(s)=-\frac{b_+\varkappa_-\sin(s\varkappa_+)-b_-\varkappa_+\sin(s\varkappa_-)}{\varkappa_-\varkappa_+(b_+-b_-)},
\end{array}\right.
\end{equation}
and
\begin{equation}\label{kappa+-b+-}
\varkappa_-=\lambda \sigma_-\quad \varkappa_+=\lambda \sigma_+,\quad b_+=\frac{\alpha_1\varkappa_+^2-\la^2\rho}{\gamma \mu \la^2}\quad \text{and}\quad b_-= \frac{\alpha_1\varkappa_-^2-\la^2\rho}{\gamma \mu \la^2},
\end{equation}
where $\sigma_+$ and $\sigma_-$ defined in \eqref{sigma+-}.
\end{lemma}

\begin{proof} Firstly, note that
\sn{\eqref{sigma+-} and \eqref{kappa+-b+-} directly imply that
\begin{equation}\label{PE-Lemma5-EQ2}
\left\{\begin{array}{l}
\displaystyle
b_+=\frac{\alpha\mu-\rho\beta+\sqrt{(\alpha\mu-\rho\beta)^2+4\gamma^2\beta^2\mu\rho}}{2\beta\gamma\mu}\neq 0,\quad b_-=\frac{\alpha\mu-\rho\beta-\sqrt{(\alpha\mu-\rho\beta)^2+4\gamma^2\beta^2\mu\rho}}{2\beta\gamma\mu}\neq 0,\\[0.1in]
\displaystyle
b_+-b_-=\frac{\sqrt{(\alpha\mu-\rho\beta)^2+4\gamma^2\beta^2\mu\rho}}{\beta\gamma\mu}\neq 0,\quad b_+b_-= - \frac{\rho}{\mu}\neq 0,\\[0.1in]
\displaystyle
\varkappa_+\varkappa_-=\la^2\sigma_+\sigma_-\neq 0,\quad \text{and}\quad  \frac{\varkappa_+}{\varkappa_-}=\frac{\sigma_+}{\sigma_-}\neq 0.
\end{array}
\right.	
\end{equation}
which, indeed, better explains the expressions of $e_i, i=1, \cdots, 8$ in \eqref{ei}.}

Let $U^{PE}=(v,v_x,p,p_x)^{\top}.$ By \eqref{PE-Lemma2-EQ2},
\begin{equation}\label{PE-Lemma4-U0Ul1}
U^{PE}(0)=(0,v_x(0),0,p_x(0)),\qquad U^{PE}(l_1)=\left(o(\la^{-(\frac{\ell}{4}+1)}),o(\la^{-\frac{\ell}{4}}),p(l_1),o(\la^{-\frac{\ell}{4}})\right)^{\top},	
\end{equation}
and therefore, the system \eqref{PE-Combining3} \ma{can} be written as
\begin{equation}\label{PE-Lemma4-EQ3}
U^{PE}_x=N^{PE}U^{PE}+G	
\end{equation}
where
\begin{equation}\label{NG}
N^{PE}=\begin{pmatrix}
0&1&0&0 \\
-\frac{\la^2\rho}{\alpha_1}&0&-\frac{\la^2\gamma\mu}{\alpha_1}&0\\
0&0&0&1\\
-\frac{\la^2\rho\gamma}{\alpha_1}&0&\frac{-\la^2\mu\alpha}{\alpha_1\beta}&0
\end{pmatrix}
\quad \text{and}\quad G=\begin{pmatrix}
0\\ \la^{-\ell} G^1+ i \la^{1-\ell} G^2\\ 0\\ \la^{-\ell} G^3 + i \la^{1-\ell} G^4	
\end{pmatrix}.
\end{equation}
Notice that the eigenvalues $\varkappa$ of the matrix $N^{PE}$ are the roots of the following characteristic equation
$$
\ma{\varsigma(\varkappa)=\frac{q(\varkappa)}{\alpha_1\beta}}
$$
\ma{where $q(\varkappa)$ is defined in \eqref{Carac-Eq}.} This characteristic equation has four distinct pure imaginary roots $i\varkappa_-,-i\varkappa_-,i\varkappa_+$, $-i\varkappa_+$ where $\varkappa_+$ and $\varkappa_-$ are defined in \eqref{kappa+-b+-}. Since the eigenvalues of $N^{PE}$ are simple, $N^{PE}$ is a diagonalizable matrix, i.e., $N^{PE}$ can be written as $N^{PE}=PN_1^{PE}P^{-1}$ such that
$$
P=\begin{pmatrix}
1&1&1&1\\
i\varkappa_+&-i\varkappa_+&i\varkappa_-&-i\varkappa_-\\
b_+&b_+&b_-&b_-\\
i\varkappa_+b_+&-i\varkappa_+b_+&i\varkappa_-b_-&-i\varkappa_-b_-	
\end{pmatrix},\quad N_1^{PE}=\begin{pmatrix}
i \varkappa_+&0&0&0\\
0&-i \varkappa_+&0&0\\
0&0&i \varkappa_-&0\\
0&0&0&i \varkappa_-
\end{pmatrix},
$$
and
$$
P^{-1}=\frac{1}{2(b_+-b_-)}\begin{pmatrix}
-b_-&\frac{i b_-}{\varkappa_+}&1&-\frac{i}{\varkappa_+}\\
-b_-&-\frac{i b_-}{\varkappa_+}&1&\frac{i}{\varkappa_+}\\
b_+&-\frac{i b_+}{\varkappa_-}&-1&\frac{i}{\varkappa_-}\\
b_+&\frac{i b_+}{\varkappa_-}&-1&-\frac{i}{\varkappa_-}	
\end{pmatrix}
$$
where $b_+$ and $b_-$ are defined in \eqref{kappa+-b+-}. Therefore, for all $s\in \mathbb{R}$,  $E(s):=e^{N^{PE}s}=Pe^{N_1^{PE}s}P^{-1}$, or equivalently,
\begin{equation}\label{PE-Lemma4-EQ4}
E(s):=e^{N^{PE}s}=\left(E_{ij}(s)\right)_{1\leq i\leq 4\
1\leq j\leq 4},	
\end{equation}
with
\begin{equation}\label{Eij}
\left\{\begin{array}{ll}
\displaystyle
E_{11}(s)= E_{22}(s)=\frac{b_+\cos(s \varkappa_-)-b_-\cos(s \varkappa_+)}{b_+-b_-},&\quad
E_{12}(s)= \frac{b_+\varkappa_+\sin(s\varkappa_-)-b_-\varkappa_-\sin(s\varkappa_+)}{\varkappa_- \varkappa_+(b_+-b_-)},\\
\displaystyle
E_{13}(s)=\frac{\cos(s \varkappa_+)-\cos(s \varkappa_-)}{b_+-b_-},&\quad E_{14}(s)=\frac{\varkappa_-\sin(s\varkappa_+)-\varkappa_+\sin(s\varkappa_-)}{\varkappa_-\varkappa_+(b_+-b_-)},\\
\displaystyle
E_{21}(s)=\frac{\varkappa_+b_-\sin(s\varkappa_+)-\varkappa_-b_+\sin(s\varkappa_-)}{b_+-b_-},&\quad
E_{23}(s)=\frac{\varkappa_-\sin(s\varkappa_-)-\varkappa_+\sin(s\varkappa_+)}{b_+-b_-},\\
\displaystyle
E_{32}(s)=\frac{b_+b_-\left(\varkappa_+\sin(s\varkappa_-)-\varkappa_-\sin(s\varkappa_+)\right)}{\varkappa_-\varkappa_+(b_+-b_-)},&\quad
E_{33}(s)=E_{44}(s)=\frac{b_+\cos(s\varkappa_+)-b_-\cos(s\varkappa_-)}{b_+-b_-},\\
\displaystyle
E_{34}(s)=\frac{b_+\varkappa_-\sin(s\varkappa_+)-b_-\varkappa_+\sin(s\varkappa_-)}{\varkappa_-\varkappa_+(b_+-b_-)},&\quad
E_{43}(s)=\frac{b_-\varkappa_-\sin(s\varkappa_-)-b_+\varkappa_+\sin(s\varkappa_+)}{b_+-b_-},\\
\displaystyle
E_{24}(s)=E_{13}(s),\quad E_{31}(s)=-b_+b_-E_{13}(s), \\
E_{41}(s)=-b_+ b_- E_{23}(s),\quad E_{42}(s)=-b_+ b_- E_{24}(s).
\end{array}
\right.
\end{equation}
By the classical arguments from the theory of ordinary differential equations, the solution of \eqref{PE-Lemma4-EQ3} is given by
$$
U^{PE}(x)=e^{N^{PE}(x-l_1)}U^{PE}(l_1)-\int_x^{l_1}e^{N^{PE}(x-s)}G(s)ds.
$$
with
$$
U^{PE}(0)=e^{-N^{PE}l_1}U^{PE}(l_1)-\int_0^{l_1}e^{-N^{PE}s}G(s)ds.
$$
Next, substitute \eqref{PE-Lemma4-U0Ul1} and \eqref{NG} and \eqref{Eij} into the above equation to obtain
$$
\begin{pmatrix}
0\\ v_x(0)\\ 0\\ p_x(0)
\end{pmatrix}=E(-l_1)\begin{pmatrix}
o(\la^{-(\frac{\ell}{4}+1)})\\ o(\la^{-\frac{\ell}{4}})\\ p(l_1)\\ o(\la^{-\frac{\ell}{4}})	
\end{pmatrix}+\int_0^{l_1}E(-s)\begin{pmatrix}
0\\ \la^{-\ell}G^1(s)+i\la^{1-\ell} G^2(s)\\ 0\\ \la^{-\ell}G^3(s)+i\la^{1-\ell} G^4(s)	
\end{pmatrix}
ds.
$$
This together with \eqref{Eij}  yields
\begin{equation*}
\begin{array}{l}
\displaystyle
E_{13}(-l_1)p(l_1)=-E_{11}(-l_1)o(\la^{-(\frac{\ell}{4}+1)})-E_{12}(-l_1)o(\la^{-\frac{\ell}{4}})-E_{14}(-l_1)o(\la^{-\frac{\ell}{4}})\\
\displaystyle
-\la^{-\ell}\int_0^{l_1}\left(E_{12}(-s)G^1(s)+E_{14}(-s)G^3(s)\right)ds-i\la^{1-\ell} \int_0^{l_1}\left(E_{12}(-s)G^2(s)+E_{14}(-s)G^4(s)\right) ds,	
\end{array}	
\end{equation*}
and
\begin{equation*}
\begin{array}{l}
E_{33}(-l_1)p(l_1)=-E_{31}(-l_1)o(\la^{-(\frac{\ell}{4}+1)})-E_{32}(-l_1)o(\la^{-\frac{\ell}{4}})-E_{34}(-l_1)o(\la^{-\frac{\ell}{4}})\\
\displaystyle
-\la^{-\ell}\int_0^{l_1}\left(E_{32}(-s)G^1(s)+E_{34}(-s)G^3(s)\right)ds-i\la^{1-\ell} \int_0^{l_1}\left(E_{32}(-s)G^2(s)+E_{34}(-s)G^4(s)\right) ds,
\end{array}
\end{equation*}
which, thus,  lead to \eqref{PE-Lemma4-EQ1} and \eqref{PE-Lemma4-EQ2}.
\end{proof}
\begin{lemma}\label{PE-Lemma5}
The solution $(v,z,p,q,y,y^1)$ of the system \eqref{PE-EXP1}-\eqref{PE-EXP2} satisfies the following asymptotic estimates
\begin{equation}\label{PE-Lemma5-EQ1}
\begin{array}{l}
\left(\cos(l_1\varkappa_+)-\cos(l_1\varkappa_-)\right)\la p(l_1)=o(\la^{-\frac{\ell}{4}}),\quad \left(b_+\cos(l_1\varkappa_+)-b_-\cos(l_1\varkappa_-)\right)\la p(l_1)=o(\la^{-\frac{\ell}{4}}).	\end{array}
\end{equation}
\end{lemma}
\begin{proof}
First, the following estimates are immediate by \eqref{ei}, \eqref{PE-Lemma5-EQ2}, \eqref{G1234}, and $\|\mathcal{G}\|_{\mathcal{H}_{PE}}=o(1)$,
\begin{equation}\label{PE-Lemma5-EQ3}
\begin{array}{l}
\abs{e_2(l_1) o(\la^{-(\frac{\ell}{4}+1)})}=o(\la^{-(\frac{\ell}{4}+1)}),\quad \abs{e_3(l_1)o(\la^{-\frac{\ell}{4}})}=o(\la^{-(\frac{\ell}{4}+1)}),\quad \abs{e_4(l_1)o(\la^{-\frac{\ell}{4}})}=o(\la^{-(\frac{\ell}{4}+1)}),\\
\abs{e_6(l_1)o(\la^{-(\frac{\ell}{4}+1)})}=o(\la^{-(\frac{\ell}{4}+1)}),\quad \abs{e_7(l_1)o(\la^{-\frac{\ell}{4}})}=o(\la^{-(\frac{\ell}{4}+1)}),\quad \abs{e_8(l_1)o(\la^{-\frac{\ell}{4}})}=o(\la^{-(\frac{\ell}{4}+1)}),\\
\displaystyle
\left|\la^{-\ell}\int_0^{l_1}\left(e_3(s)G^1(s)+e_4(s)G^3(s)\right)ds\right|=o(\la^{-(1+\ell)}),\\
\displaystyle
\left|\la^{-\ell}\int_0^{l_1}\left(e_7(s)G^1(s)+e_8(s)G^3(s)\right)ds\right|=o(\la^{-(1+\ell)}).
\end{array}
\end{equation}
Now, integrate by parts to obtain
\begin{equation*}
\begin{array}{l}
\displaystyle
i\la^{1-\ell} \int_0^{l_1}e_3(s)G^2(s)ds=\frac{i\la^{1-\ell} G^2(l_1)}{\varkappa_-\varkappa_+(b_+-b_-)}\left[b_+\frac{\varkappa_+}{\varkappa_-}\cos(l_1\varkappa_-)-b_-\frac{\varkappa_-}{\varkappa_+}\cos(l_1\varkappa_+)\right]\\[0.1in]
\displaystyle
-\frac{i\la^{1-\ell}}{\varkappa_-\varkappa_+(b_+-b_-)}\int_0^{l_1}\left[b_+\frac{\varkappa_+}{\varkappa_-}\cos(s\varkappa_-)-b_-\frac{\varkappa_-}{\varkappa_+}\cos(s\varkappa_+)\right]G^2_s(s)ds,\\[0.1in]
\displaystyle
i\la^{1-\ell} \int_0^{l_1}e_4(s)G^4(s)ds= \frac{i\la^{1-\ell} G^4(l_1)}{\varkappa_-\varkappa_+(b_+-b_-)}\left[\frac{\varkappa_-}{\varkappa_+}\cos(l_1 \varkappa_+)-\frac{\varkappa_+}{\varkappa_-}\cos(l_1 \varkappa_-)\right]\\[0.1in]
\displaystyle
-\frac{i\la^{1-\ell} }{\varkappa_-\varkappa_+(b_+-b_-)}\int_0^{l_1}\left[\frac{\varkappa_-}{\varkappa_+}\cos(l_1 \varkappa_+)-\frac{\varkappa_+}{\varkappa_-}\cos(l_1 \varkappa_-)\right]G^4_s(s)ds, 	
\end{array}
\end{equation*}
\begin{equation*}
\begin{array}{l}
\displaystyle
i\la^{1-\ell} \int_0^{l_1}e_7(s)G^2(s)ds=\frac{i\la^{1-\ell} b_+b_-G^2(l_1)}{\varkappa_-\varkappa_+(b_+-b_-)}\left[\frac{\varkappa_+}{\varkappa_-}\cos(l_1\varkappa_-)-\frac{\varkappa_-}{\varkappa_+}\cos(l_1\varkappa_+)\right]\\[0.1in]
\displaystyle
-\frac{i\la^{1-\ell} b_+b_-}{\varkappa_-\varkappa_+(b_+-b_-)}\int_0^{l_1}\left[\frac{\varkappa_+}{\varkappa_-}\cos(s\varkappa_-)-\frac{\varkappa_-}{\varkappa_+}\cos(s\varkappa_+)\right]G^2_s(s)ds,\\[0.1in]
\displaystyle
i\la^{1-\ell} \int_0^{l_1}e_8(s)G^4(s)ds=\frac{i\la^{1-\ell} G^4(l_1)}{\varkappa_-\varkappa_+(b_+-b_-)}\left[b_+\frac{\varkappa_-}{\varkappa_+}\cos(l_1\varkappa_+)-b_-\frac{\varkappa_+}{\varkappa_-}\cos(l_1\varkappa_-)\right]\\[0.1in]
\displaystyle
-\frac{i\la^{1-\ell} }{\varkappa_-\varkappa_+(b_+-b_-)}\int_0^{l_1}\left[b_+\frac{\varkappa_-}{\varkappa_+}\cos(s\varkappa_+)-b_-\frac{\varkappa_+}{\varkappa_-}\cos(s\varkappa_-)\right]G^4_s(s)ds.
\end{array}
\end{equation*}
Next, by \eqref{PE-Lemma5-EQ3} and  $G^2,G^4\in H_L^1(0,l_1)\subset C([0,l_1])$,
\begin{equation}\label{PE-Lemma5-EQ4}
\begin{array}{l}
\displaystyle
\left|i\la^{1-\ell} \int_0^{l_1}e_3(s)G^2(s)ds\right|\leq \left|\frac{i\la^{1-\ell} G^2(l_1)}{\varkappa_-\varkappa_+(b_+-b_-)}\left[b_+\frac{\varkappa_+}{\varkappa_-}\cos(l_1\varkappa_-)-b_-\frac{\varkappa_-}{\varkappa_+}\cos(l_1\varkappa_+)\right]\right|\\[0.1in]
\displaystyle
+\left|\frac{i\la^{1-\ell}}{\varkappa_-\varkappa_+(b_+-b_-)}\int_0^{l_1}\left[b_+\frac{\varkappa_+}{\varkappa_-}\cos(s\varkappa_-)-b_-\frac{\varkappa_-}{\varkappa_+}\cos(s\varkappa_+)\right]G^2_s(s)ds\right|\leq o(\la^{-(1+\ell)}),\\[0.1in]
\displaystyle
\left|i\la^{1-\ell} \int_0^{l_1}e_4(s)G^4(s)ds\right|\leq \left| \frac{i\la^{1-\ell} G^4(l_1)}{\varkappa_-\varkappa_+(b_+-b_-)}\left[\frac{\varkappa_-}{\varkappa_+}\cos(l_1 \varkappa_+)-\frac{\varkappa_+}{\varkappa_-}\cos(l_1 \varkappa_-)\right]\right|\\[0.1in]
\displaystyle
+\left|\frac{i\la^{1-\ell} }{\varkappa_-\varkappa_+(b_+-b_-)}\int_0^{l_1}\left[\frac{\varkappa_-}{\varkappa_+}\cos(l_1 \varkappa_+)-\frac{\varkappa_+}{\varkappa_-}\cos(l_1 \varkappa_-)\right]G^4_s(s)ds\right|\leq o(\la^{-(1+\ell)}), 	
\end{array}
\end{equation}
\begin{equation}\label{PE-Lemma5-EQ5}
\begin{array}{l}
\displaystyle
\left|i\la^{1-\ell} \int_0^{l_1}e_7(s)G^2(s)ds\right|\leq \left|\frac{i\la^{1-\ell} b_+b_-G^2(l_1)}{\varkappa_-\varkappa_+(b_+-b_-)}\left[\frac{\varkappa_+}{\varkappa_-}\cos(l_1\varkappa_-)-\frac{\varkappa_-}{\varkappa_+}\cos(l_1\varkappa_+)\right]\right|\\[0.1in]
\displaystyle
+\left|\frac{i\la^{1-\ell} b_+b_-}{\varkappa_-\varkappa_+(b_+-b_-)}\int_0^{l_1}\left[\frac{\varkappa_+}{\varkappa_-}\cos(s\varkappa_-)-\frac{\varkappa_-}{\varkappa_+}\cos(s\varkappa_+)\right]G^2_s(s)ds\right|\leq o(\la^{-(1+\ell)}),\\[0.1in]
\displaystyle
\left|i\la^{1-\ell} \int_0^{l_1}e_8(s)G^4(s)ds\right|\leq \left|\frac{i\la^{1-\ell} G^4(l_1)}{\varkappa_-\varkappa_+(b_+-b_-)}\left[b_+\frac{\varkappa_-}{\varkappa_+}\cos(l_1\varkappa_+)-b_-\frac{\varkappa_+}{\varkappa_-}\cos(l_1\varkappa_-)\right]\right|\\[0.1in]
\displaystyle
+\left|\frac{i\la^{1-\ell} }{\varkappa_-\varkappa_+(b_+-b_-)}\int_0^{l_1}\left[b_+\frac{\varkappa_-}{\varkappa_+}\cos(s\varkappa_+)-b_-\frac{\varkappa_+}{\varkappa_-}\cos(s\varkappa_-)\right]G^4_s(s)ds\right|\leq o(\la^{-(1+\ell)}).
\end{array}
\end{equation}
Finally, \eqref{PE-Lemma5-EQ1} follows from \eqref{PE-Lemma5-EQ2} and substituting  \eqref{PE-Lemma5-EQ3}, \eqref{PE-Lemma5-EQ4} and \eqref{PE-Lemma5-EQ5} into \eqref{PE-Lemma4-EQ1} and \eqref{PE-Lemma4-EQ2}.
 \end{proof}

\begin{lemma}
Assume \ma{\eqref{LD-E} }. Let $\ell=0$ (if $\ma{\rm{\mathbf{(H_{Exp})}}}$ holds) or $\ell=4\varpi\left(\frac{\sigma_+}{\sigma_-}\right)-4$ (if $\ma{\rm{\mathbf{(H_{Pol})}}}$ holds).  Then,  the solution $(v,z,p,q,y,y^1)$ of the system \eqref{PE-EXP1}-\eqref{PE-EXP2} satisfies the following estimate
\begin{equation}\label{PE-Lemma6-EQ1}
\abs{\la p(l_1)}=o(1). 	
\end{equation}
\end{lemma}
\begin{proof}
Since $\ell \geq 0$,  it is easy to see that $\abs{\la p(l_1)}=O(1)$ from \eqref{PE-Lemma3-EQ4}. Now assume that \eqref{PE-Lemma6-EQ1} does not hold. Then, there exists  a positive constant $cst$
and a subsequence   such that $\ma{\abs{\la p(l_1)}\geq cst}$.  By \eqref{PE-Lemma5-EQ1},
\begin{equation*}\label{PE-Lemma6-EQ7}
\cos(l_1\varkappa_+)-\cos(l_1\varkappa_-)=o(\la^{-\frac{\ell}{4}})\quad \text{and}\quad b_+\cos(l_1\varkappa_+)-b_-\cos(l_1\varkappa_-)=o(\la^{-\frac{\ell}{4}}),
\end{equation*}
from which, the following is obtained
$$
\mathcal{M}(b_+,b_-)\begin{pmatrix}
\cos(l_1\varkappa_+)\\ 	\cos(l_1\varkappa_-)
\end{pmatrix}=\begin{pmatrix}
o(\la^{-\frac{\ell}{4}})\\ o(\la^{-\frac{\ell}{4}})	
\end{pmatrix},\quad \text{and}\,\, \mathcal{M}(b_+,b_-)=\begin{pmatrix}
1&-1\\ b_+&-b_-	
\end{pmatrix}.
$$
From \eqref{PE-Lemma5-EQ2}, $\det(\mathcal{M}(b_+,b_-))=b_+-b_-\neq 0$, and thus,
$$
\cos(l_1\varkappa_+)=o(\la^{-\frac{\ell}{4}})\quad \text{and}\quad \cos(l_1\varkappa_-)=o(\la^{-\frac{\ell}{4}}).
$$
This together with \eqref{kappa+-b+-} imply that there exists $n^1_+,n^1_-\in \mathbb{Z}$ such that
$$
\la =\frac{(2n^1_++1)\pi}{2\sigma_+l_1}+o(\la^{-\frac{\ell}{4}})\quad\text{and}\quad \la =\frac{(2n^1_-+1)\pi}{2\sigma_-l_1}+o(\la^{-\frac{\ell}{4}}).
$$
Since $\la $ is large enough, i.e., $\la \sim \frac{(2n^1_++1)\pi}{2\sigma_+l_1}\sim \frac{(2n^1_-+1)\pi}{2\sigma_-l_1}$,
$$
\frac{(2n^1_++1)\pi}{2\sigma_+l_1}+o(\la^{-\frac{\ell}{4}})=\frac{(2n^1_-+1)\pi}{2\sigma_-l_1}+o(\la^{-\frac{\ell}{4}}).
$$
and therefore,
\begin{equation}\label{PE-Lemma6-EQ9}
(2n^1_++1)-(2n^1_-+1)\frac{\sigma_+}{\sigma_-}=o(\la^{-\frac{\ell}{4}}).	
\end{equation}
$\bullet$ Assume that $\rm{\mathbf{(H_{Exp})}}$ holds and take $\ell=0$. Then, by  \eqref{PE-Lemma6-EQ9}
$$
\frac{(2n_++1)\xi_--(2n_-+1)\xi_+}{\xi_-}=o(1).
$$
It is known that $\abs{(2n^1_++1)\xi_--(2n^1_-+1)\xi_+}\geq 1$ since $(2n^1_-+1)\xi_+$ is an even number and $(2n^1_++1)\xi_-$ is an odd number,   or $(2n^1_-+1)\xi_+$ is an odd number and $(2n^1_++1)\xi_-$ is an even number. However, this contradicts with $\frac{1}{\xi_-}\leq o(1)$. Consequently,  \eqref{PE-Lemma6-EQ1} is obtained as \ma{$\rm{\mathbf{(H_{Exp})}}$} holds.\\
$\bullet$ Assume $\rm{\mathbf{(H_{Pol})}}$ and choose $\ell=4\varpi\left(\frac{\sigma_+}{\sigma_-}\right)-4$. Since $\la \sim \frac{(2n^1_++1)\pi}{2\sigma_+l_1}\sim \frac{(2n^1_-+1)\pi}{2\sigma_-l_1}$ and by \eqref{PE-Lemma6-EQ9},
$$
\frac{\sigma_+}{\sigma_-}-\frac{2n^1_++1}{2n^1_-+1}=o(\la^{-\varpi\left(\frac{\sigma_+}{\sigma_-}\right)}).
$$
However, the contradiction $c\left(\frac{\sigma_+}{\sigma_-},\Lambda\right)\leq o(1)$ is immediate by $\rm{\mathbf{(H_{Pol})}}$ with the sequence
$\Lambda=((2n_+^1+1,2n_-^1+1))_{n\in \mathbb{N}}$.  Consequently, \eqref{PE-Lemma6-EQ1} is obtained as  $\rm{\mathbf{(H_{Pol})}}$ holds.

\end{proof}
$\newline$
\\
\textbf{Proof of Theorem \ref{PE-EXP}.} Assume $\rm{\mathbf{(H_{Exp})}}$. Then, \eqref{PE-Lemma2-EQ1}, \eqref{PE-Lemma3-EQ1} and \eqref{PE-Lemma6-EQ1} results in $\|U\|_{\mathcal{H}_{PE}}=o(1)$, which contradicts  by \eqref{PE-EXP-EQ2}. Consequently, the  condition \eqref{N2-PE} holds true. \\
\textbf{Proof of Theorem \ref{Pol-PE}}. Assume $\rm{\mathbf{(H_{Pol})}}$ and take $\ell=4\varpi(\frac{\sigma_+}{\sigma_-})-4$. Then, \eqref{PE-Lemma2-EQ1}, \eqref{PE-Lemma3-EQ1} and \eqref{PE-Lemma6-EQ1}  result in$\|U\|_{\mathcal{H}_{PE}}=o(1)$, which contradicts  by \eqref{PE-EXP-EQ2}. Consequently, the condition \eqref{N2-PE} holds true.

{\color{black}
\begin{rem} \label{rem-electro2}
Following Remark \ref{rem-electro}, it is also important to note that electrostatic/quasi-static approaches in the modeling of piezoelectric beams exclude dynamic electromagnetic effects, expressed as $\mu\equiv 0$ in \eqref{PE}. Our analysis in this section reveals that under the assumption \eqref{LD-E}, an exponential stability result can also be obtained with solely local damping. The proof  is left to the reader. 
\end{rem}
}

\section{Illustration of the hypothesis $\rm{\mathbf{(H_{Pol})}}$} \label{sillustration}

In this section,   some examples are provided for the   hypothesis $\rm{\mathbf{(H_{Pol})}}$ to hold true. For this purpose, we start with the notion of badly approximable real numbers.

\begin{definition} \cite[Definition 1.3]{Bugeaud} \label{BAN}
A real number $\xi$  is  badly approximable if there is a positive constant $c(\xi)$  such that for every rational number $\frac{p}{q}\ne \xi,$
\begin{equation}\label{BAN1}
\left|\xi -\frac{p}{q}\right|> \frac{c(\xi)}{q^2}. 	
\end{equation}
\end{definition}
\noindent It is well known (see \cite[Theorem 1.1 and Corollary 1.2]{Bugeaud}) that rational and irrational quadratic  numbers are badly approximable. However the set $\mathcal{B}$ of   badly approximable numbers is larger since $\xi\in \mathcal{B}$ if and only if the sequence$ \{x_n\}_{n\in \mathbb{N}}$ is bounded,
denoting $[x_0,x_1,\cdots,x_n,\cdots]$ its expansion as a continued fraction, see  \cite[Theorem 1.9]{Bugeaud}. Note also that the Lebesgue measure of $\mathcal{B}$ is equal to zero.

Now, by Definition \ref{BAN}, it safe to deduce that if $\frac{\sigma_+}{\sigma_-}$
is a badly approximable irrational number,  $\rm{\mathbf{(H_{Pol})}}$ holds with
$\varpi(\frac{\sigma_+}{\sigma_-})=2$, and consequently, Theorem \ref{Pol-PE} yields a polynomial energy decay  in $t^{-1/2}$.
\noindent The case $\frac{\sigma_+}{\sigma_-}\not\in \mathcal{B}$, though, requires the notion of the irrationality measure (sometimes called the Liouville-Roth constant or irrationality exponent). For this, recall the following result from  \cite{Bugeaud:18}.

\begin{definition} \label{irrationality measure}
Let $\xi$ be an irrational real number. Then, the real number
$\mu\geq 1$ is called to be the irrationality measure of $\xi,$ if there exists a positive constant $C(\xi, \mu, \varepsilon)$ for every positive real number $\varepsilon$ such that
 \begin{equation}\label{defBurgeaud}
\left|\xi -\frac{p}{q}\right|> \frac{C(\xi,\mu ,\varepsilon)}{q^{\mu+\varepsilon}},\qquad
\forall  p,q \in \mathbb{Z} \hbox{    with }q\geq 1.
\end{equation}
The irrationality exponent
$\mu(\xi)$ of $\xi$ is defined as the infimum of the irrationality measures of $\xi$.
\end{definition}
\noindent Notice that $\mu(\xi)$ is always  $\geq 2$, see \cite[Theorem E.2]{Bugeaud}. A direct consequence of this definition is that if $\frac{\sigma_+}{\sigma_-}\not\in \mathcal{B}$ is an
irrational real number such that its irrationality exponent
$\mu(\xi)$ is finite, then
 $\rm{\mathbf{(H_{Pol})}}$ holds with
$\varpi(\frac{\sigma_+}{\sigma_-})=\mu(\frac{\sigma_+}{\sigma_-})+\varepsilon$ for any $\varepsilon>0$. Let us then give some examples of irrational real numbers with finite irrationality exponent. First by the Roth's theorem, for every algebraic number of degree $\geq 2$, $\mu(\xi)=2$, see \cite{Bugeaud}.
However, for many irrational real numbers $\xi$, the exact value of $\mu(\xi)$ is not explicitly known but  some upper bound for their  irrationality exponent
is available, see Table \ref{table 1}. Note that  $\nu(\xi)$ is an upper bound of $\mu(\xi)$ if
$\mu(\xi)\leq \nu(\xi)$, therefore, we automatically have
 \[
\left|\xi -\frac{p}{q}\right|> \frac{C(\xi,\mu,\varepsilon)}{q^{\nu(\xi)+\varepsilon}},\qquad
\forall  p,q \in \mathbb{Z} \hbox{    with }q\geq 1,
\]
for all $\varepsilon>0$. Consequently if $\nu(\xi)$ is finite,  $\rm{\mathbf{(H_{Pol})}}$ holds with
$\varpi(\frac{\sigma_+}{\sigma_-})=\nu(\frac{\sigma_+}{\sigma_-})+\varepsilon$, for any $\varepsilon>0$.

\noindent Now, in order to give other illustrations stated in the literature, an equivalent formulation of
the irrationality measure, frequently used as the definition of the irrationality measure, may be needed, e.g. see \cite{Irrationality,ZeilbergerZudilin:2020}.

\begin{lemma}\label{dmb}
Fix  an irrational real number $\xi$ and a real number
$\mu\geq 1$. Then, the following are equivalent
\\
(1) For every positive real number  $\varepsilon$, there exists a positive constant $C(\xi,\mu ,\varepsilon)$ such that
\eqref{defBurgeaud} holds.
\\
(2) For every positive real number  $\varepsilon$,  there exists a positive integer $N(\xi,\mu ,\varepsilon)$ such that
\begin{equation}\label{defstandard}
\left|\xi -\frac{p}{q}\right|> \frac{1}{q^{\mu+\varepsilon}},\qquad
\forall  p,q \in \mathbb{Z} \hbox{    with }q\geq N(\xi,\mu ,\varepsilon).
\end{equation}
\end{lemma}
\begin{proof}
 ``$(2) \Rightarrow (1)$:''  Fix $\varepsilon>0$ and suppose that (1) does not hold. Then, for all   $n\in \mathbb{N}^*$, there exists $p_n,q_n \in \mathbb{Z}$ with $q_n\geq 1$
such that
\[
 \left|\xi -\frac{p_n}{q_n}\right|\leq  \frac{1}{n q_n^{\mu+\varepsilon}},\qquad
 \forall n\in \mathbb{N}^*.
\]
This trivially implies that
\[
 \left|\xi -\frac{p_n}{q_n}\right|\leq  \frac{1}{n},\qquad
 \forall n\in \mathbb{N}^*,
\]
and consequently, $\frac{p_n}{q_n}$ converges to $\xi$ as $n$ goes to infinity. Therefore,
$q_n$ (as well as $p_n$) approaches infinity as $n$ goes  to infinity.
For large enough $n$ , $q_n$ will be greater than $N(\xi,\mu ,\varepsilon)$.
As (2) holds, by \eqref{defstandard}, the following is deduced
\[
 \frac{1}{q_n^{\mu+\varepsilon}}< \left|\xi -\frac{p}{q}\right| \leq  \frac{1}{n q_n^{\mu+\varepsilon}},\qquad
 \forall n\in \mathbb{N}^*: q_n\geq N(\xi, \mu, \varepsilon),
 \]
which leads to a contradiction by letting $n$ go to infinity.
\\ ``$(1) \Rightarrow (2)$:''
For a fixed $\varepsilon>0$,
assume (1) with $\frac{\varepsilon}{2}$. Then, there exists
 a positive constant $C(\xi,\mu, \frac{\varepsilon}{2})$ such that
 \[
\left|\xi -\frac{p}{q}\right|> \frac{C(\xi,\mu,\frac{\varepsilon}{2})}{q^{\mu+\frac{\varepsilon}{2}}},\qquad
\forall  p,q \in \mathbb{Z} \hbox{    with }q\geq 1,
\]
which is equivalent to
 \begin{equation}\label{defBurgeaud2}
\left|\xi -\frac{p}{q}\right|> \frac{C(\xi, \mu, \frac{\varepsilon}{2})q^{\frac{\varepsilon}{2}}}{q^{\mu+\varepsilon}},\qquad
\forall  p,q \in \mathbb{Z} \hbox{    with }q\geq 1.
\end{equation}
Since $C(\xi,{\color{red}\mu?},\frac{\varepsilon}{2})q^{\frac{\varepsilon}{2}}$ approaches infinity as $q$ goes to infinity, restrict ourselves to $q$ such that
\[
C(\xi,\mu, \frac{\varepsilon}{2})q^{\frac{\varepsilon}{2}}\geq 1,
\]
or equivalently
\[
q\geq C(\xi,\mu, \frac{\varepsilon}{2})^{-\frac{2}{\varepsilon}}.
\]
By choosing $N(\xi, \varepsilon)\geq C(\xi,\mu, \frac{\varepsilon}{2})^{-\frac{2}{\varepsilon}},$
\eqref{defBurgeaud2}implies \eqref{defstandard} for $q\geq N(\xi, \varepsilon)$.
\end{proof}

 In relation to Remark \ref{RkSN},  the following equivalent statements can be formulated \begin{lemma}
Let us fix  $\xi$ an irrational real number and a real number
$\nu\geq 1$. Then the following results  are equivalent:
\\
(1)  There exists a positive constant $C(\xi,\nu)$ such that
\[
\left|\xi -\frac{p}{q}\right|> \frac{C(\xi,\nu)}{q^{\nu}},\qquad
\forall  p,q \in \mathbb{Z} \hbox{    with }q\geq 1.
\]
\\
 (2) For all sequences  $\Lambda=(p_{n}, q_{n})_{n\in \mathbb{N}}\in (\mathbb{N}\times \mathbb{N}^*)^{\mathbb{N}}$ with $p_{n}\sim q_{n}$ for sufficiently large $n$,
 there exist
 a positive constant  $c_1\left(\xi, \nu, \Lambda\right)$
 and a positive integer $N_1\left(\xi, \nu, \Lambda\right)$
  such that
\[
\left|\xi-\frac{p_{n}}{q_{n}}\right|> \frac{c_1\left(\xi, \nu, \Lambda\right)}{q_{n}^{\nu}},\qquad  \forall n\geq
N_1\left(\xi, \nu, \Lambda\right).\]
\\
(3) For all sequences  $\Lambda=(p_{n}, q_{n})_{n\in \mathbb{N}}\in (\mathbb{N}\times \mathbb{N}^*)^{\mathbb{N}}$ for which $\frac{p_{n}}{q_{n}}$ approaches $\xi$ as   $n$ goes to infinity,
 there exist
 a positive constant  $c_2\left(\xi, \nu, \Lambda\right)$
 and a positive integer $N_2\left(\xi, \nu, \Lambda\right)$
  such that
\[
\left|\xi-\frac{p_{n}}{q_{n}}\right|> \frac{c_2\left(\xi, \nu, \Lambda\right)}{q_{n}^{\nu}}, \qquad \forall n\geq
N_2\left(\xi, \nu, \Lambda\right).\]
\end{lemma}
\begin{proof}
Obviously $(1)\Rightarrow (2) \Rightarrow (3)$. Hence it suffices to show that
$(3) \Rightarrow (1)$. We prove this by a contradiction argument. Assume that (3) holds but not (1). Then, for all   $n\in \mathbb{N}^*$, there exists $p_n,q_n \in \mathbb{Z}$ with $q_n\geq 1$
such that
\[
 \left|\xi -\frac{p_n}{q_n}\right|\leq  \frac{1}{n q_n^{\mu+\varepsilon}},\qquad
 \forall n\in \mathbb{N}^*.
\]
As in Lemma \ref{dmb}, this trivially implies that
 $\frac{p_n}{q_n}$ approaches $\xi$ as $n$ goes to infinity. Combining this with
(3) results in
\[
 \frac{c_2(\xi, \nu, \Lambda)}{q_n^{\nu}}< \left|\xi -\frac{p_n}{q_n}\right| \leq  \frac{1}{n q_n^{\nu}},\qquad
 \forall n\in \mathbb{N}^*: q_n\geq N_2(\xi, \nu, \Lambda),
 \]
 where $\Lambda=((p_n, q_n))_{n\in \mathbb{N}}$.
This leads to a contradiction by letting $n$ go to infinity.
\end{proof}

This result shows in particular that if $\rm{\mathbf{(H_{Pol})}}$ holds,
$\varpi(\frac{\sigma_+}{\sigma_-})$ is an irrationality mesure for $\frac{\sigma_+}{\sigma_-}$. Moreover, this result also shows
that the condition \eqref{defBurgeaud} in the definition of an irrationality measure
can be replaced by a condition on sequences (as in  (2) or (3) above).

Finally,  upper bounds of some   irrationality exponents are presented in Table \ref{table 1}. To keep it short, only a few of them are provided. The interested readers can refer to, e.g.\cite{Irrationality}, for various other upper bounds.

\begin{table}
\begin{center}
\begin{tabular}{|c|c|c|}
\hline
$\xi$&upper bound of $\mu\left(\xi\right)$&Reference\\
\hline
$\pi$&7.10320...& \cite{ZeilbergerZudilin:2020} \\
\hline
$\pi^2$ or $\frac{\pi^2}{6}$&5.09541...&  \cite{Zudilin:2013}\\
\hline
$\ln(2)$&3.57455391&\cite{MR2539543}\\
\hline	
$\ln(3)$&5.116201&\cite{MR3873184}\\
\hline
$\zeta(3)$ Ap\'erys constant &5.513891& \cite{MR1826005}\\
\hline
Ln$_q(2)$ q-harmonic series &2.9384& \cite{MR2196997}\\
\hline
$h_q(1)$ q-harmonic series &2.4650& \cite{MR2138454}\\
\hline
$\displaystyle{T_2(b^{-1})=\sum_{n=1}^{\infty}t_nb^{n-1}}$, $b\geq 2$ and &&\\
&4&\cite{MR2557148}\\
$t_n$ is the $n-th$ term of the True-Morse sequence&&  \\
\hline
\end{tabular}
\end{center}
 \caption{Upper bounds of some irrationality exponents \label{table 1}}
 \end{table}

\begin{rem}
As a final remark, notice that  there exist irrational numbers $\xi$ for which
$\mu(\xi)=\infty$. These numbers are called Liouville numbers.
For example,   $L = \displaystyle{\sum_{k=1}^{\infty}b^{-k!}}$ is a Liouville number.
Therefore if  $\frac{\sigma_+}{\sigma_-}$ is a  Liouville number,
the decay rate of the energy  of the system \eqref{PE} is   still an open problem. 	
\end{rem}

\section{Conclusions and open problems}

In this paper, two different transmission problems are investigated: (i) a transmission problem of an Elastic-Piezoelectric-Elastic design with only one local damping acting on the longitudinal displacement of the center line of the piezoelectric layer and  (ii) a transmission problem of a Piezoelectric-Elastic design with only one local damping acting on the elastic part.
An exponential stability result is immediate for (i). However, for (ii), the nature of the stability (polynomial or exponential) entirely  depends on the arithmetic nature of a quotient involving all the physical parameters of the system.

{\color{black} An interesting open problem, which is currently under investigation \cite{ANOS}, is the stability of the following Piezoelectric-Elastic-Piezoelectric design}
\begin{equation}\label{PEP}\tag{${\rm P/E/P}$}
\left\{\begin{array}{ll}
\rho_1 v^1_{tt}-\alpha v^1_{xx}+\gamma_1\beta_1 p^1_{xx}+b_1(x)v^1_t=0,& (x,t) \in (0,l_1)\times (0,\infty)\\
\mu_1 p^1_{tt}-\beta_1 p^1_{xx}+\gamma_1 \beta_1 v^1_{xx}=0,& (x,t) \in (0,l_1)\times (0,\infty),\\
u_{tt}-c_1u_{xx}+b_2(x)u_t=0,& (x,t) \in (l_1,l_2)\times (0,\infty),\\
\rho_2 v^2_{tt}-\tilde{\alpha} v^2_{xx}+\gamma_2\beta_2 p^2_{xx}+b_3(x)v^2_t=0,& (x,t) \in (l_2,L)\times (0,\infty),\\
\mu_2 p^2_{tt}-\beta_2 p^2_{xx}+\gamma_2 \beta_2 v^2_{xx}=0,& {\color{black} (x,t) \in (l_2,L)\times (0,\infty),} \\
v^1(0,t)=p^1(0,t)=v^2(L,t)=p^2(L,t)=0,&  \\
v^1(l_1,t)=u(l_1,t),& \\
\alpha v^1_x(l_1,t)-\gamma_1\beta_1 p^1_x(l_1,t)=c_1u_x(l_1,t),&\\
\beta_1 p^1_x(l_1,t)=\gamma_1\beta_1 v^1_x(l_1,t),&\\
v^2(l_2,t)=u(l_2,t),&\\
\tilde{\alpha} v^2_x(l_2,t)-\gamma_2\beta_2 p^2_x(l_2,t)=c_1u_x(l_2,t),\\
\beta_2 p^2_x(l_2,t)=\gamma_2\beta_2 v^2_x(l_2,t),&   t\in (0,\infty),\\
(v^1, p^1, u, v^2, p^2) (\cdot,0)=(v^1_0, p^1_0, u_0, v^2_0, p^2_0)(\cdot),\\
(v^1_t, p^1_t, u_t ,v^2_t,p^2_t)(\cdot,0)=(v^1_1, p^1_1, u_1, v^2_1, p^2_1)(\cdot)
\end{array}
\right.	
\end{equation}
with the assumptions $b_2=0$, $b_1\in L^{\infty}(0,l_1)$, $b_3\in L^{\infty}(l_2,L),$ and
$$
\left\{\begin{array}{ll}
b_1(x)\geq b_{1,0}>0\quad \text{in}\quad (r_1,r_2)\subset (0,l_1),\ \text{and}\  b_1(x)=0, \quad x \in (0,l_1)\backslash (r_1,r_2),\\
b_3(x)\geq b_{3,0}>0\quad \text{in}\quad (r_5,r_6)\subset (l_2,L),\ \text{and}\  b_3(x)=0\quad \text{in}\quad (l_2,L)\backslash (r_5,r_6).
\end{array}
\right.
$$
By adopting analogous arguments as in Section \ref{Section-EPE-LD-P}, we conjecture that one may prove that the system \eqref{PEP} is exponentially stable.
Furthermore, by assuming that $b_1=b_3=0$ and $b_2\in L^{\infty}(l_1,l_2)$ such that
$$
b_2(x)\geq b_{2,0}>0\quad \text{in}\quad (r_3,r_4)\subset (l_1,l_2),\ \text{and}\  b_2(x)=0\quad \text{in}\quad (l_1,l_2)\backslash (r_3,r_4),
$$
analogous stability results as  in Section \ref{PE-LD-E} may also be obtained. \\
\\

\noindent Another open problem, which  deserves to be investigated, is the stability of the Elastic-Piezoelectric-Elastic design with two dampings terms acting only on the elastic part
\begin{equation*}
\left\{\begin{array}{ll}
u_{tt}-c_1u_{xx}+k_1(x)u_t=0,&(x,t) \in(0,l_1)\times (0,\infty),\\
\rho v_{tt}-\alpha v_{xx}+\gamma\beta p_{xx}=0,& (x,t) \in(l_1,l_2)\times (0,\infty),\\
\mu p_{tt}-\beta p_{xx}+\gamma \beta v_{xx}=0,& (x,t) \in (l_1,l_2)\times (0,\infty),\\
y_{tt}-c_2y_{xx}+k_2(x)y_t=0,& (x,t) \in (l_2,L)\times (0,\infty),\\
u(0,t)=y(L,t)=0,\\
v(l_1,t)=u(l_1,t),\\
v(l_2,t)=y(l_2,t),\\
\alpha v_x(l_1,t)-\gamma\beta p_x(l_1,t)=c_1u_x(l_1,t),&\\
\alpha v_x(l_2,t)-\gamma\beta p_x(l_2,t)=c_2y_x(l_2,t),&\\
\beta p_x(l_1,t)=\gamma\beta v_x(l_1,t),&\\
\beta p_x(l_2,t)=\gamma\beta v_x(l_2,t), & t\in (0,\infty),\\
(u, v, p, y, u_t, v_t, p_y, y_t) (\cdot,0))=(u_0, v_0, p_0, y_0, u_1, v_1, p_1, y_1)(\cdot),
\end{array}
\right.	
\end{equation*}
with $k_1\in L^{\infty}(0,l_1)$ and $k_2\in L^{\infty}(l_2,L)$ such that
$$
\left\{\begin{array}{l}
k_1(x)\geq k_{1,0}>0\quad \text{in}\quad (r_1,r_2)\subset (0,l_1),\ \text{and}\  k_1(x)=0\quad \text{in}\quad (0,l_1)\backslash (r_1,r_2),\\
k_2(x)\geq k_{2,0}>0\quad \text{in}\quad (r_5,r_6)\subset (l_2,L),\ \text{and}\  k_2(x)=0\quad \text{in}\quad (l_2,L)\backslash (r_5,r_6).
\end{array}
\right.
$$


\begin{thebibliography}{10}
\bibitem{MR2557148}B. Adamczewski, and  T. Rivoal. Irrationality measures for some automatic real numbers. {\em Math. Proc. Cambridge Philos. Soc.}. \textbf{147}, 659-678 (2009).

\bibitem{Abdelaziz1} M. Afilal, A. Soufyane, and M. Santos. Piezoelectric beams with magnetic effect and localized damping. {\em Math. Control Relat. Fields}. \textbf{13}, 250-264 (2023).

\bibitem{ANOS} M. Akil, S. Nicaise, A.\"{O}. \"{O}zer, H. Saleh, {\sl Advancing Insights into the Stabilization of Novel Serially-Connected Magnetizable Piezoelectric and Elastic Beams}, submitted, 2024.

\bibitem{https://doi.org/10.48550/arxiv.2204.00283}
M. Akil. Stability of piezoelectric beam with magnetic effect under (Coleman or Pipkin)-Gurtin thermal law. {\em Z. Angew. Math. Phys.} \textbf{73}, Paper No. 236, 31 (2022).

\bibitem{MR4534451}M. Akil, A. Soufyane, and Y. Belhamadia. Stabilization Results of a Piezoelectric Beams with Partial Viscous Dampings and Under Lorenz Gauge Condition. {\em Appl. Math. Optim.}. \textbf{87}, Paper No. 26 (2023).


\bibitem{AnLiuKong}
Y.~An, W.~Liu, and A.~Kong.
\newblock Stability of piezoelectric beams with magnetic effects of fractional
  derivative type and with/without thermal effects.
\newblock 2021.

\bibitem{Arendt01}
W.~Arendt, and C.~J.~K. Batty.
\newblock  Tauberian theorems and
  stability of one-parameter semigroups.
\newblock {\em Trans. Amer. Math. Soc.}, 306(2):837--852, 1988.



\bibitem{Batty01}
C.~J.~K. Batty, and T.~Duyckaerts.
\newblock Non-uniform stability for bounded semi-groups on {B}anach spaces.
\newblock {\em J. Evol. Equ.}, 8(4):765--780, 2008.

\bibitem{MR3873184}I. Bondareva, M. Luchin, and V.  Salikhov. Symmetrized polynomials in a problem of estimating the irrationality measure of the number $\ln 3$. {\em Chebyshevskii Sb.} \textbf{19}, 15-25 (2018).

\bibitem{Borichev01}
A.~Borichev, and Y.~Tomilov.
\newblock Optimal polynomial decay of functions and operator semigroups.
\newblock {\em Math. Ann.}, 347(2):455--478, 2010.

\bibitem{Bugeaud} Y. Bugeaud. Approximation by algebraic numbers. (Cambridge University Press, Cambridge, 2004).

\bibitem{Bugeaud:18}
Y. Bugeaud.
\newblock Effective irrationality measures for real and {$p$}-adic roots
              of rational numbers close to 1, with an application to
              parametric families of {T}hue-{M}ahler equations.
\newblock {\em Math. Proc. Cambridge Philos. Soc.}, 164(1):99--108, 2018.

\bibitem{Chen} G. Chen, {\sl Energy decay estimates and exact boundary value controllability for wave equation in a bounded domain,} J. Math. Pures Appl., 5, 249-274, 1979.
\bibitem{Chen2} F. Chen et al, Smart Materials and  Structures, 27,  095022, 2018.

\bibitem{Dag1} C. Dagdeviren, P. Joe, O.L.Tuzman, K. Park, K.J. Lee, Y. Shi, Y. Huang, and J.A. Rogers. { Recent Progress in Flexible and Stretchable Piezoelectric Devices for Mechanical Energy Harvesting, Sensing and Actuation}, Extreme Mechanics Letter, {9(1)}, 269-281, 2016.

\bibitem{Dag2} C. Dagdeviren, and  L. Zhang. {  Methods and Apparatus for Imaging with Conformable Ultrasound Patch,} US Patent App. 16/658, 237, 2020.
\bibitem{Dong} W. Dong, L. Xiao, W. Hu, C.  Zhu, Y. Huang, Z. Yin,  {  Wearable human–machine interface based on PVDF piezoelectric sensor}, Transactions of the Institute of Measurement and Control, 39-4, 398–403, 2017.


\bibitem{DosSan} M.J. Dos Santos, {A.\"O. \"Ozer}, M.M. Freitas, A.J.A Ramos, and D.S. Almeida Junior. {  On global attractors for a novel nonlinear piezoelectric beam model with magnetic effects and long-range memory,} { Zeitschrift für angewandte Mathematik und Physik,} 73, Article Number: 136, 2022.



\bibitem{Ala} S. El Alaoui, A.{\"O}. {\"O}zer, and M. Ouzahra, {   Boundary Feedback Stabilization of Nonlinear Piezoelectric Extensible Beams,}  {\sl Zeitschrift f\"ur angewandte Mathematik und Physik}, 74-15, 2023.

\bibitem{Fatori1} L.H. Fatori, E. Lueders, and J.E.L Rivera. {  Transmission problem
for hyperbolic thermoelastic systems}, Journal of Thermal Stresses 26, 739–
763, 2003.
\bibitem{Fatori2}  L.H. Fatori, and C.L. Antonio. {  Exponential decay of serially connected elastic wave,} Boletim da Sociedade Paranaense de Matemática, 28-1, 2010.

\bibitem{Feng1} B. Feng,  and {A.\"O. \"Ozer.} {  Exponential stability results for the boundary-controlled fully-dynamic piezoelectric beams with various distributed and boundary delays,} { Journal of Mathematical Analysis and Applications}.  508-1, 125845, 2022.


\bibitem{Feng2} B. Feng, and {A.\"O. \"Ozer.} {  Long-time behavior of a nonlinearly damped Rao-Nakra sandwich beam,} { Applied Mathematics and Optimization}, 87-19, 2023.









\bibitem{Freitas1} M.M. Freitas,  {A.\"O. \"Ozer,} and A. J. A. Ramos. {  Long time  dynamics and   upper semi-continuity  with respect to magnetic permeability of attractors  to electrostatic piezoelectric beams with nonlinear boundary dissipation,} { ESAIM: Control, Optimisation and Calculus of Variations,} Volume 28, Article Number 39, 2022.

\bibitem{Lag1} J. Lagnese, {\sl Decay of solutions of wave equations in a bounded region with boundary dissipation,} J. Diff. Eqns., 50, 163-182, 1983.


\bibitem{Lions} G. Duvaut, and J.L. Lions. {  Inequalities in Mechanics and Physics,} (Springer-
365 Verlag-1976).

\bibitem{Ebrah} F. Ebrahimi, and  M.R. Barati. {  Vibration analysis of smart piezoelectrically actuated nano-beams subjected to magneto-electrical field in thermal environment}, Journal of Vibration and Control, {(24-3)} (2016), 549--564.

    \bibitem{Huang01}
F.~L. Huang.
\newblock Characteristic conditions for exponential stability of linear
  dynamical systems in {H}ilbert spaces.
\newblock {\em Ann. Differential Equations}, 1(1):43--56, 1985.


\bibitem{Lag} J. E. Lagnese, {  Boundary Stabilization of Thin Plates,} SIAM, Philadelphia, 1989.

\bibitem{Leugering} J. E. Lagnese, G. Leugering, and E.J.P.G. Schmidt. {  Modeling, Analysis and Control of Dynamic Elastic
Multi-Link Structures}, Birkhäuser, Boston (1994).

\bibitem{Ling} M. Ling, L. L. Howell, J. C. , G. Chen, {\sl  Kinetostatic and Dynamic Modeling of Flexure-Based Compliant Mechanisms: A Survey}, Appl. Mech. Rev., 72(3): 030802, 2020.


\bibitem{RaoLiu01}
Z.~Liu, and B.~Rao.
\newblock Characterization of polynomial decay rate for the solution of linear
  evolution equation.
\newblock {\em Z. Angew. Math. Phys.}, 56(4):630--644, 2005.


\bibitem{Liu} G. Liu, {A.\"O. \"Ozer,} and M. Wang. {  Longtime dynamics for a novel piezoelectric beam model with creep and thermo-viscoelastic effects,} {Nonlinear Analysis: Real World Applications}, (68), 103666, 2022.

\bibitem{MR2539543}R. Marcovecchio.  The Rhin-Viola method for $\log 2$. {\em Acta Arith.}. \textbf{139}, 147-184 (2009).

\bibitem{MR2196997}T. Matala-aho, K. V\"an\"anen, and W. Zudilin. New irrationality measures for $q-$logarithms. {\em Math. Comp.} \textbf{75}, 879-889 (2006).


\bibitem{Naso1} A. Marzocchi, J.E.M Rivera, and M.G. Naso. {  Asymptotic behaviour and exponential stability for a transmission problem in thermoelasticity}, Mathematical Methods in the Applied Sciences 25-11, 955-980, 2002.

\bibitem{Morris-Ozer2013}
K.~Morris, and A.{\"O}. {\"O}zer.
\newblock Strong stabilization of piezoelectric beams with magnetic effects.
\newblock pages 3014--3019, 2013.

\bibitem{Morris-Ozer2014}
K.~A. Morris, and A.{\"O}. {\"O}zer.
\newblock Modeling and stabilizability of voltage-actuated piezoelectric beams
  with magnetic effects.
\newblock {\em SIAM J. Control. Optim.}, 52-4,  2371--2398, 2014.

\bibitem{OzerMCSS}
A.{\"O}. {\"O}zer. \newblock Further stabilization and exact observability results for voltage-actuated piezoelectric beams with magnetic effects.
newblock {\em Math. Control Signals Syst. } 27, 219–244, 2015.

\bibitem{OzerIEEETAC}
A.{\"O}. {\"O}zer.  \newblock Modeling and controlling an active constrained layer (ACL) beam actuated by two voltage sources with/without magnetic effects,
\newblock {\em IEEE Transactions on Automatic Control.} 62-12,  6445-6450, 2017.



\bibitem{pruss01}
J.~Pr\"uss.
\newblock {On the spectrum of
  {$C_{0}$}-semigroups}.
\newblock {\em Trans. Amer. Math. Soc.}, 284(2):847--857, 1984.

\bibitem{Ramos2018}
{A.J.A. Ramos}, C. S. L. Gon\c{c}alves,  and {S. S. Corr\^ea Neto.}
\newblock Exponential stability and numerical treatment for piezoelectric beams
  with magnetic effect.
\newblock {\em ESAIM: M2AN}, 52(1):255--274, 2018.

\bibitem{Ramos} A.J.A. Ramos,  et al. {  Equivalence between exponential stabilization and boundary observability for piezoelectric beams with magnetic effect}, {\sl Zeitschrift f\"ur angewandte Mathematik und Physik}, (70-60),  2019.
			
\bibitem{Rivera} J.E.M Rivera, and H.P. Oquendo, {  The Transmission Problem of Viscoelastic Waves}, {\sl Acta Applicandae Mathematicae}, 62, 1–21 (2000).

\bibitem{MR1826005}G. Rhin, and  C. Viola. The group structure for $\zeta(3)$. {\em Acta Arith.} \textbf{97}, 269-293 (2001)


\bibitem{Ru} C. Ru, X. Liu, and Y. Sun. {  Nanopositioning Technologies: Fundamentals and Applications,} (Springer International, 2016).

\bibitem{ROZENDAAL2019359}
J.~Rozendaal, D.~Seifert, and R.~Stahn.
\newblock Optimal rates of decay for operator semigroups on {H}ilbert spaces.
\newblock {\em Advances in Mathematics}, 346:359 -- 388, 2019.

\bibitem{Abdelaziz2}
A.~Soufyane, M.~Afilal, and M.~L. Santos.
\newblock Energy decay for a weakly nonlinear damped piezoelectric beams with
  magnetic effects and a nonlinear delay term.
\newblock {\em Zeitschrift f\"ur angewandte Mathematik und Physik}, 72(4),
  Aug. 2021.


\bibitem{Shi} Q. Shi, T. Wang, and C. Lee. {  MEMS Based Broadband Piezoelectric Ultrasonic Energy Harvester (PUEH) for Enabling Self-Powered Implantable Biomedical Devices}, Sci Rep 6 (2016), 24946.
\bibitem{Smith} R.C. Smith, {   Smart Material Systems}, (Society for Industrial and Applied Mathematics, 2005).


\bibitem{Voss} T. Voss, and J.M.A. Scherpen, {  Stabilization and shape control of a 1D piezoelectric Timoshenko beam,} Automatica, (47-12),
			2780-85, 2011.

\bibitem{Irrationality} E. Weisstein. ``Irrationality Measure''. From \textit{MathWorld}$--$A Wolfram Web Resource. \url{https://mathworld.wolfram.com/IrrationalityMeasure.html}

\bibitem{Wilson} A.\"{O}. \"{O}zer, and 	W. Horner. {  Uniform boundary observability of Finite Difference approximations of non-compactly-coupled piezoelectric beam equations,} {\sl Applicable Analysis,} (101-5), 1571-1592.	

 \bibitem{Yang} J. Yang. { {An Introduction to the Theory of Piezoelectricity,}}  New York (Springer-2005).

\bibitem{MR4450079} H.Zhang,  G. Xu,  \& Z. Han. Stability of multi-dimensional nonlinear piezoelectric beam with viscoelastic infinite memory. {\em Zeitschrift f\"ur angewandte Mathematik und Physik} \textbf{73}, No. 159, 18 (2022).

\bibitem{ZeilbergerZudilin:2020} D.Zeilberger  \&  W.Zudilin. The irrationality measure of $\pi$ is at most 7.103205334137\dots  . {\em Mosc. J. Comb. Number Theory}. \textbf{9}, 407-419 (2020).

\bibitem{Zudilin:2013}W. Zudilin,  On the irrationality measure of $\pi^2$. {\em Uspekhi Mat. Nauk}. \textbf{68}, 171-172 (2013).

\bibitem{MR2138454}W. Zudilin. Approximations to $q-$logarithms and $q-$dilogarithms, with applications to $q-$zeta values. {\em Zap. Nauchn. Sem. S.$-$Petetsburg. Otdel. Mat. Inst. Steklov. (POMI)}. \textbf{322}, 107-124, 253-254 (2005).

\end{thebibliography}

\noindent {\bf Acknowledgement:} (i) The authors would like to thank Prof. Y. Bugeaud (IRMA, Université de Strasbourg) on fruitfull discussions about the notion of irrationality measures. (ii) A.\"{O}. \"{O}zer  gratefully acknowledges the financial support of the National Science Foundation of USA under Cooperative Agreement No: 1849213.\\
$\newline$

\noindent {\bf Declarations:}\\ \textbf{Competing interest} The authors have not disclosed any competing interests.

\end{document}